\documentclass[11pt, a4paper]{article}
\usepackage{amsmath,amssymb,amsthm}
\usepackage{tikz}
\usepackage{subfig}
\usepackage{stmaryrd} 

\usepackage[auth-sc, affil-it]{authblk}

\usepackage{verbatim}
\usepackage[utf8]{inputenc}
\usepackage[english]{babel}
\usepackage[T1]{fontenc}
\usepackage{lmodern}
\usepackage{microtype}
\usepackage[
    linktocpage=true,
    colorlinks=true,
    linkcolor=purple,
    citecolor=black,
    urlcolor=purple,
    bookmarksnumbered=true,
]{hyperref}

\usepackage[totalwidth=160mm, totalheight=247mm]{geometry}

\usepackage{enumitem}
\setlist{noitemsep}
\setlist[enumerate]{label=(\arabic*)}

\let\oldbibliography\thebibliography
\renewcommand{\thebibliography}[1]{%
  \oldbibliography{}%
  \small%
  \setlength{\itemsep}{0pt}%
}

\makeatletter
\def\@cite#1#2{[{{\bfseries#1}\if@tempswa , #2\fi}]} 
\renewcommand{\@biblabel}[1]{[{\bfseries{#1}}]~} 
\makeatother

\usepackage[format=plain,justification=justified]{caption}
\theoremstyle{plain}
\newtheorem{theo}{Theorem}[section]
\newtheorem{prop}[theo]{Proposition}
\newtheorem{lemm}[theo]{Lemma}
\newtheorem{coro}[theo]{Corollary}
\theoremstyle{definition}
\newtheorem{defi}[theo]{Definition}
\newtheorem{exam}[theo]{Example}
\newtheorem{rema}[theo]{Remark}

\renewcommand{\leq}{\leqslant}
\renewcommand{\geq}{\geqslant}
\newcommand{\bbN}{\mathbb N}
\newcommand{\bbZ}{\mathbb Z}
\newcommand{\bbQ}{\mathbb Q}
\newcommand{\bbR}{\mathbb R}
\newcommand{\bbC}{\mathbb C}

\newcommand{\bbT}{\mathbb T}
\newcommand{\bbK}{\mathbb K}

\newcommand{\mcA}{\mathcal A}
\newcommand{\mcF}{\mathcal F}
\newcommand{\mcG}{\mathcal G}

\newcommand{\mcP}{\mathcal P}
\newcommand{\mcU}{U}
\newcommand{\mcV}{\mathcal V}
\newcommand{\mcW}{\mathcal W}
\newcommand{\mcD}{\mathcal D}
\newcommand{\mcO}{\mathcal O}
\newcommand{\mcL}{\mathcal L}

\newcommand{\mcT}{\mathcal T}
\newcommand{\mcX}{\mathcal X}

\newcommand{\LBrun}{\mathcal L^\mathsf{Br}}
\newcommand{\LJP}{\mathcal L^\mathsf{JP}}

\newcommand{\FBrun}{\mathcal F^\mathsf{Br}}
\newcommand{\GBrun}{\mathcal G^\mathsf{Br}}
\newcommand{\GBrunmini}{\widetilde{\mathcal G}^\mathsf{Br}}
\newcommand{\HBrun}{\mathcal H^\mathsf{Br}}
\newcommand{\FJP}{\mathcal F^\mathsf{JP}}
\newcommand{\GJP}{\mathcal G^\mathsf{JP}}
\newcommand{\HJP}{\mathcal H^\mathsf{JP}}
\newcommand{\VBrun}{V^\mathsf{Br}}
\newcommand{\VJP}{V^\mathsf{JP}}
\newcommand{\sBrun}{\sigma^\mathsf{Br}}
\newcommand{\sJP}{\sigma^\mathsf{JP}}
\newcommand{\SBrun}{\Sigma^\mathsf{Br}}
\newcommand{\SJP}{\Sigma^\mathsf{JP}}

\newcommand{\bfM}{\mathbf M}
\newcommand{\bfP}{\mathbf P}
\newcommand{\bfE}{\mathbf E}
\newcommand{\bfe}{\mathbf e}

\newcommand{\bfv}{\mathbf v}

\newcommand{\bfx}{\mathbf x}
\newcommand{\bfy}{\mathbf y}

\newcommand{\Gv}{\Gamma_\bfv}

\newcommand{\transp}[1]{{{}^\textup{\textbf t} #1}}
\newcommand{\EOS}{{{\mathbf E}_1^\star}}
\newcommand{\EOSS}{{{\mathbf E}_1^\star(\sigma)}}

\newcommand{\Ms}{{\mathbf M}_\sigma}
\newcommand{\Msinv}{{\mathbf M}^{-1}_\sigma}

\newcommand{\pic}{{\pi_{\textup{c}}}}

\DeclareMathOperator{\rad}{rad}

\newcommand{\myvcenter}[1]{\ensuremath{\vcenter{\hbox{#1}}}}

\DeclareTextFontCommand{\tdef}{\itshape\bfseries}

\newcommand{\svect}[3]{%
\Big(\begin{smallmatrix}%
#1 \\ #2 \\ #3%
\end{smallmatrix}\Big)%
}

\title{\textbf{A combinatorial approach to \\ products of Pisot substitutions}}

\author[1]{Val\'erie Berth\'e}
\author[2]{J\'er\'emie Bourdon}
\author[1,3]{Timo Jolivet}
\author[4]{Anne Siegel}

\affil[1]{%
    LIAFA, CNRS, Universit\'e Paris Diderot, Case 7014, 75205 Paris Cedex 13, France
}
\affil[2]{%
    LINA, Universit\'e de Nantes, 2 rue de la Houssini\`ere, 44322 Nantes Cedex, France
}
\affil[3]{%
    FUNDIM, Department of Mathematics, University of Turku 20014, Finland
}
\affil[4]{%
    INRIA, CNRS, Univ. Rennes 1, IRISA, Campus de Beaulieu, 35042 Rennes Cedex, France
}

\date{}

\begin{document}
\maketitle

\vspace{-10mm}
\begin{abstract}
We define a generic algorithmic framework to prove pure discrete spectrum
for the substitutive symbolic dynamical systems associated
with some infinite families of Pisot substitutions.
We focus on the families obtained as finite products of the three-letter substitutions
associated with the multidimensional continued fraction algorithms of Brun and Jacobi-Perron.

Our tools consist in a reformulation of some combinatorial criteria (coincidence conditions),
in terms of properties of discrete plane generation using multidimensional (dual) substitutions.
We also deduce some topological and dynamical properties of the Rauzy fractals,
of the underlying symbolic dynamical systems,
as well as some number-theoretical properties of the associated Pisot numbers.
\end{abstract}
{\small
\tableofcontents
}

\section{Introduction}

Symbolic substitutions are morphisms of the free monoid
which are often used to generate one-dimensional infinite words and symbolic dynamical systems.
They play a prominent role among symbolic dynamical systems with zero entropy,
and are intimately connected with first return maps and self-induced systems
(see for instance~\cite{Fog02,Que10} for more details).
Substitutions can also be defined in the framework of higher-dimensional tiling spaces~\cite{Sol97,Rob04}.

According to the \emph{Pisot substitution conjecture} (see e.g.~\cite{Fog02,BK06,CANT,ABBLS}),
the symbolic dynamical system associated with any irreducible substitution whose expansion factor is a Pisot number
is conjectured to have pure discrete spectrum.
This property is equivalent to being measure-theoretically isomorphic to a translation on a compact abelian group.
Rauzy introduced in~\cite{Rau82} an explicit construction to establish this property
by defining a bounded compact set with fractal boundary (now called the \emph{Rauzy fractal} of the substitution)
as a candidate for a fundamental domain for the underlying translation.
This approach belongs to a wider scientific program which consists in studying
how to code dynamical systems of an arithmetic nature as symbolic systems
preserving their arithmetic properties~\cite{Sidorov},
and conversely, how symbolic dynamical systems can provide good simultaneous approximation
algorithms in Diophantine approximation~\cite{BFZ05}.
Following this line of research, the study of Rauzy fractals and Pisot substitutions has led to many developments
in various domains, including symbolic dynamics and combinatorics on words,
numeration dynamics, hyperbolic dynamics, fractal topology and number theory
(see~\cite{Fog02,CANT} and the references therein).

It is now well established that the pure discrete spectrum property is algorithmically decidable for
a given irreducible Pisot substitution in the unimodular case, that is,
when the abelianization matrix of the substitution has determinant~$\pm 1$~\cite{SieT09,CANT,AL11,AL14}.
Nonetheless, no general strategy exists to check whether the pure discrete spectrum property holds
for all the substitutions within a given \emph{infinite} family.
A particular type of infinite families of substitutions can be obtained by fixing a finite set of substitutions $S$
and by considering all the finite products of substitutions over $S$;
we call them \emph{product families}.
This is for example how Arnoux-Rauzy substitutions are constructed~\cite{AR91}.
The main goal of the present paper is to pursue Rauzy's program by introducing
a \emph{generic computational framework to prove the pure discrete spectrum property}
for some product families of three-letter substitutions associated with
two-dimensional continued fraction algorithms defined as piecewise fractional maps
(according to the formalism developed in~\cite{Bre81,Sch00};
see also ~\cite{Ber11} for a discussion on the way substitutions are associated with continued fraction algorithms).
We focus here on the Brun and Jacobi-Perron continued fraction algorithms,
and we derive generic methods as well as dynamical, topological and number-theoretical implications.

Discrete spectrum results for infinite families have been established
for two-letter unimodular irreducible Pisot substitutions~\cite{BD02,HS03},
and for Arnoux-Rauzy substitutions (see also~\cite{BJS13,BJS12,BSW13}).
Our framework  has also been successfully applied to a problem in discrete geometry
about the critical thickness at which an arithmetic discrete plane is $2$-connected~\cite{BJJP13},
in connection with the ordered fully subtractive continued fraction algorithm.
The modified Jacobi-Perron substitutions are studied as well in~\cite{FIY13}using a similar approach
(the modified Jacobi-Perron algorithm is a two-point extension of the Brun algorithm).
Recently, the pure discrete spectrum property was also
announced to be proved for some infinite families of substitutions satisfying some simple combinatorial
requirements on their first and last letters~\cite{Barge14},
which includes $\beta$-substitutions associated with simple Parry numbers (see also~\cite{Aki00,BBK06})
and some of the families of substitutions considered below.

\paragraph{Combinatorial methods}
Our strategy relies on two main ingredients.
The first one is a combinatorial definition of Rauzy fractals, introduced in~\cite{AI01},
as the Hausdorff limit of planar (in the three-letter case) compact sets
obtained by projecting and renormalizing some three-dimensional objects (``patterns'')
consisting of finite unions of faces of unit cubes placed at integer coordinates.
These patterns are obtained by iterating a geometric, higher-dimensional substitution $\EOSS$.
According to~\cite{AI01}, it can be seen as a \emph{dual substitution} of the original (symbolic) substitution $\sigma$.
A striking fact is that the resulting three-dimensional patterns obtained by iterating $\EOSS$
on the three faces of the lower half unit cube $\mcU := \input{fig/U.tex}$ placed at the origin
always lie within a \emph{discrete plane}~\cite{AI01,Fer09}.
By \emph{discrete plane}, we mean a discretized version of a Euclidean plane;
see Section~\ref{subsec:dissub} for a precise definition.

The second ingredient is the following characterization of the pure discrete spectrum property,
due to~\cite{IR06}, stated in terms of generation of discrete planes by dual substitutions:
the symbolic dynamical system generated by a unimodular irreducible Pisot substitution $\sigma$ has a pure discrete spectrum
if and only if the patterns generated by iterating the dual substitution $\EOSS$ of $\sigma$
cover arbitrarily large disks in the discrete plane when starting with the initial pattern $\mcU$.
This geometric and combinatorial criterion, which we call the \emph{arbitrarily large disks covering property},
is one of the numerous formulations of the so-called \emph{coincidence conditions}
used in substitutive dynamics to prove the pure discrete spectrum property.
A survey is given in~\cite{CANT,ABBLS}, see also the references in~\cite{AL11}.
Note that we focus here on the $\bbZ$-action of the shift of the substitutive symbolic dynamical system,
and not on the $\bbR$-action of the substitutive tiling flow.
However, for irreducible Pisot substitutions, pure discreteness of both actions are equivalent~\cite{Clark-Sadun:03}.

This characterization of pure discrete spectrum allows the study
of finite products of substitutions over $S$ in the following way.
First, we prove a stronger result which applies not only to substitutive words, but also to $S$-adic words,
that is, to the limits of infinite products of substitutions~\cite{BD13}:
the arbitrarily large disks covering property holds when iterating
any admissible \emph{infinite} sequence $(\sigma_{i_n})_{n \geq 1}$ of substitutions over the finite set $S$.
\emph{Admissible} refers to the possible constraints on the allowed infinite sequences,
which occur from example with the substitutions associated with some continued fraction algorithms.
Then, for a given \emph{finite} product $\sigma=\sigma_{i_1} \sigma_{i_2} \cdots \sigma_{i_k}$ over $S$
the same result follows directly by considering the infinite
periodic infinite sequence $(\sigma_{i_1} \sigma_{i_2} \cdots \sigma_{i_k})^\infty$.
Examples of such sequences with the arbitrarily large disks property are shown
in Figure~\ref{fig:gen_goodbad}~p.~\pageref{fig:gen_goodbad}.

To prove the arbitrarily large disks covering property,
we formalize and generalize the approach developed in~\cite{IO94},
where discrete plane generation using dual Jacobi-Perron substitutions is studied.
The intuitive idea is to prove that iterating sufficiently many substitutions
eventually generates a topological annulus made of faces around the initial pattern,
and that these annuli are preserved under iteration.
Several new computational tools are introduced in this article,
including generation graphs to prove that annuli are always eventually generated.
The strategy is described in detail in Section~\ref{subsec:mcfgen_strategy},
and more precise statements are given below and in Section~\ref{sect:mainresults}.

\paragraph{Main results}

We focus on two families of substitutions,
the substitutions associated with the Brun algorithm~\cite{Bru58},
and the substitutions associated with the Jacobi-Perron algorithm~\cite{Sch73}.
Note that some of the results of this paper have been announced in the extended abstract~\cite{BBJS13}.
Our first results concern discrete plane generation
using the dual substitutions $\EOSS$  associated with these multidimensional continued fraction algorithms.
They are stated in Section~\ref{sect:mainresults}.
\begin{itemize}
\item
In \textbf{Theorem~\ref{theo:seeds}} we prove that there exist finite patterns $V$
(``seeds'', not much larger than the three-face pattern $\mcU$)
such that iterating Brun or Jacobi-Perron substitutions from $V$ generates an entire discrete plane
(\emph{i.e.}, arbitrarily large balls centered at the origin).
\item
In \textbf{Theorem~\ref{theo:bad_Brun}} (Brun) and \textbf{Theorem~\ref{theo:bad_JP}} (Jacobi-Perron)
we characterize the infinite sequences for which the pattern $\mcU$ does not suffice to generate an entire discrete plane.
These characterizations are given in terms of finite state automata with few vertices,
which are obtained algorithmically.
\item
In \textbf{Theorem~\ref{theo:balls}} we prove that translates of arbitrarily large disks
always appear in the images of $\mcU$.
\end{itemize}
The above results allow us to deduce several corollaries,
of topological, dynamical and number-theoretical nature,
including our original motivation to prove pure discrete spectrum.
\begin{itemize}
\item
In \textbf{Corollary~\ref{coro:fractal_zero}} we characterize all the finite products of Brun or Jacobi-Perron substitutions
for which the origin is not an inner point of the Rauzy fractal (this corresponds to the case where $U$ is not a seed).
Connectedness results for the corresponding Rauzy fractals are also obtained in \textbf{Corollary~\ref{coro:fractal_connected}}.
These topological properties are related with various number-theoretical properties,
as discussed in Section~\ref{subsec:applis_nt}.
\item
In \textbf{Corollary~\ref{coro:dynprod}} we prove the pure discrete spectrum property
for every admissible finite product of Brun or Jacobi-Perron substitutions.
In \textbf{Corollary~\ref{coro:markov}}
we obtain explicit Markov partitions with connected atoms
for the toral automorphisms associated with the incidence matrices of the substitutions.
\item
In \textbf{Theorem~\ref{theo:cubicfrac}} we prove that for
every cubic number field $\bbK$, there exist generators $\alpha,\beta$ with $\bbK = \bbQ(\alpha,\beta)$
and a three-letter unimodular irreducible Pisot substitution $\sigma$ such that
the toral translation on $\bbT^2$ by $(\alpha,\beta)$ is measure-theoretically conjugate
to the symbolic dynamical system generated by $\sigma$.
The corresponding Rauzy fractal provides a fundamental domain for the translation,
a partition for a natural coding, and bounded remainder sets.
\end{itemize}

An important consequence of the present framework is that it paves the way for the study of $S$-adic words.
In particular, our results are a key ingredient for providing
symbolic representations of two-dimensional toral translations.
Indeed, Brun $S$-adic words are proved in ~\cite{BST14} (based on the results of the present paper)
to provide a symbolic natural coding for~\emph{almost every} given toral translation,
and not only the ones with algebraic parameters.
More precisely, the pure discrete spectrum property for the Brun $S$-adic shift is proved
for almost all of these shifts, and conversely,
almost every two-dimensional toral translation admits a symbolic natural coding
which is provided by the Brun continued fraction algorithm.

\paragraph{Organization of the paper}
Substitutions (and dual substitutions), discrete planes, Rauzy fractals,
as well as the Brun and Jacobi-Perron continued fractions algorithms
are defined in Section~\ref{sect:prelim}. The strategy used for generating discrete planes
is then elaborated in Section~\ref{sect:gen} as generically as possible,
in the sense that it can be applied to other types of continued fraction algorithms.
An outline of the strategy is first given in Section~\ref{subsec:mcfgen_strategy}.
Strong coverings are introduced in Section~\ref{subsec:coverings},
sufficient combinatorial criteria to establish the annulus property are given in Section~\ref{subsec:annulus_property},
and the construction of generation graphs is described in Section~\ref{subsec:generation_graphs}.
Section~\ref{sect:tech_proofs} is devoted to more technical aspects
and to proofs which are specific to the Brun and Jacobi-Perron substitutions.

Our main results are then stated in Section~\ref{sect:mainresults}.
Results about discrete plane generation
using Brun and Jacobi-Perron substitutions are given in Section~\ref{subsec:generation}.
These results are applied in Sections~\ref{subsec:applis_topo},~\ref{subsec:applis_dyn}~and~\ref{subsec:applis_nt}
to prove various properties of the dynamical substitution systems associated with finite products of
Brun and Jacobi-Perron substitutions, as well as other implications of number-theoretical nature.

\paragraph{Computer proofs}
Many of the results of Section~\ref{sect:tech_proofs}
have been proved using the Sage mathematics software system~\cite{Sage}.
The corresponding Sage code is available on the arXiv preprint page
of the current article (\href{http://arxiv.org/abs/1401.0704}{\texttt{arXiv:1401.0704}}),
as an attached file (``ancillary file'').

A Sage implementation of dual substitutions
has been used to perform exhaustive enumerations of small patterns,
in order to prove some properties of the families of substitutions under study
in Sections~\ref{subsec:seeds},~\ref{subsec:covprop}~and~\ref{subsec:propA}.
Another type of results for which computer algebra was used is the algorithmic construction of generation graphs,
in Sections~\ref{subsec:graph_Brun} and~\ref{subsec:graph_JP}.

\paragraph{Acknowledgements}
We would like to thank Pierre Arnoux, Maki Furukado and Shunji Ito for numerous fruitful discussions on this topic.
This work was supported by Agence Nationale de la Recherche and the Austrian Science Fund
through project Fractals and Numeration ANR-12-IS01-0002 and project Dyna3S ANR-13-BS02-0003.

\section{Preliminaries}
\label{sect:prelim}

\subsection{Discrete planes and substitutions}
\label{subsec:dissub}

\paragraph{Substitutions}
Let $\mcA = \{1, \ldots, n\}$ be a finite set of symbols.
We work here mainly with three-letter alphabets ($n=3$).
A \tdef{substitution} is a non-erasing morphism of the free monoid $\mcA^\star$
($\sigma(a)$ is a non-empty word for every $a \in \mcA$).
We denote by $\bfP : \mcA^\star \rightarrow \bbN^n$
the \tdef{abelianization map} defined by
$\bfP(w) = (|w|_1, \ldots ,|w|_n)$,
where $|w|_i$ stands for the number of occurrences of $i$ in $w$.
The \tdef{incidence matrix} $\Ms$ of $\sigma$ is the matrix of size $n \times n$
whose $i$th column is equal to $\bfP(\sigma(i))$ for every $i \in \mcA$.
A substitution $\sigma$ is \tdef{unimodular} if $\det \Ms = \pm 1$,
and it is \tdef{irreducible Pisot} if the characteristic polynomial of $\Ms$
is the minimal polynomial of a Pisot number, that is,
a real algebraic integer larger than $1$ whose other conjugates are smaller than $1$ in modulus.
It can be proved that every irreducible Pisot substitution is \tdef{primitive},
that is, there exists a positive power of its incidence matrix~\cite{CS01}.

\paragraph{Discrete planes}
We denote by $(\bfe_1, \bfe_2, \bfe_3)$ the canonical basis of $\bbR^3$.
Before defining discrete planes we introduce (pointed)
\tdef{faces} $[\bfx,i]^\star$, which are defined as subsets of $\bbR^3$ by
\definecolor{facecolor}{rgb}{0.8,0.8,0.8}
\begin{align*}
\lbrack \bfx, 1 \rbrack^\star & = \{\bfx + \lambda \bfe_2 + \mu \bfe_3 : \lambda,\mu \in [0,1] \} =
    \myvcenter{\begin{tikzpicture}
    [x={(-0.216506cm,-0.125000cm)}, y={(0.216506cm,-0.125000cm)}, z={(0.000000cm,0.250000cm)}]
    \fill[fill=facecolor, draw=black, shift={(0,0,0)}]
    (0, 0, 0) -- (0, 1, 0) -- (0, 1, 1) -- (0, 0, 1) -- cycle;
    \node[circle,fill=black,draw=black,minimum size=1.2mm,inner sep=0pt] at (0,0,0) {};
    \end{tikzpicture}} \\
\lbrack \bfx, 2 \rbrack^\star & = \{\bfx + \lambda \bfe_1 + \mu \bfe_3 : \lambda,\mu \in [0,1] \} =
    \myvcenter{\begin{tikzpicture}
    [x={(-0.216506cm,-0.125000cm)}, y={(0.216506cm,-0.125000cm)}, z={(0.000000cm,0.250000cm)}]
    \fill[fill=facecolor, draw=black, shift={(0,0,0)}]
    (0, 0, 0) -- (0, 0, 1) -- (1, 0, 1) -- (1, 0, 0) -- cycle;
    \node[circle,fill=black,draw=black,minimum size=1.2mm,inner sep=0pt] at (0,0,0) {};
    \end{tikzpicture}} \\
\lbrack \bfx, 3 \rbrack^\star & = \{\bfx + \lambda \bfe_1 + \mu \bfe_2 : \lambda,\mu \in [0,1] \} =
    \myvcenter{\begin{tikzpicture}
    [x={(-0.216506cm,-0.125000cm)}, y={(0.216506cm,-0.125000cm)}, z={(0.000000cm,0.250000cm)}]
    \fill[fill=facecolor, draw=black, shift={(0,0,0)}]
    (0, 0, 0) -- (1, 0, 0) -- (1, 1, 0) -- (0, 1, 0) -- cycle;
    \node[circle,fill=black,draw=black,minimum size=1.2mm,inner sep=0pt] at (0,0,0) {};
    \end{tikzpicture}}
\end{align*}
where $i \in \{1,2,3\}$ is the \tdef{type} of $[\bfx,i]^\star$,
and $\bfx \in \bbZ^3$ is the \tdef{vector} of $[\bfx,i]^\star$.
In this paper we will refer to collections of faces as unions of faces,
and by abuse of language we will often say that a face $f$ belongs to a union of faces
even if $f$ is in fact included in it.
The notation $\bfx+[\bfy,i]^\star$ stands for $[\bfx+\bfy,i]^\star$.

We now define discrete planes.
Denote by $\langle \cdot, \cdot \rangle $ the usual scalar product.
Let $\bfv \in \bbR^3_{>0}$.
The \tdef{discrete plane $\Gv$ of normal vector $\bfv$} is the union of faces $[\bfx, i]^\star$, with $i \in \{1,2,3\}$ and
$\bfx \in \bbZ^3$ satisfying
$0 \leq \langle \bfx, \bfv \rangle < \langle \bfe_i, \bfv \rangle$.

More intuitively, $\Gv$ can also be seen as the boundary of
the union of the unit cubes with integer coordinates
that intersect the lower half-space $\{\bfx \in \bbR^3 : \langle \bfx, \bfv \rangle < 0\}$.
The set of its vertices in $\bbZ^3$ corresponds to the classic notion of a standard arithmetic discrete
plane in discrete geometry~\cite{Rev91}.

Observe that the lower half unit cube
$\mcU = [\mathbf 0,1]^\star \cup [\mathbf 0,2]^\star \cup [\mathbf 0,3]^\star = \input{fig/U.tex}$
is included in every discrete plane since the coordinates of the normal vector $\mathbf v$ of a discrete plane $\Gv$
are assumed to be positive.

A \tdef{pattern} is a finite union of (pointed) faces.
In order to express the fact that some patterns grow by applying (dual) substitutions,
we need to define the \tdef{minimal combinatorial radius} $\rad(P)$
of a pattern $P$ containing $\mcU $.
It is equal to the length
of the shortest sequence of faces $f_1, \ldots, f_n$ in $P$ such that
$f_1 \in \mcU$,
$f_i$ and $f_{i+1}$ share an edge,
and $f_n$ shares an edge with the boundary of $P$.
Intuitively, $\rad(P)$ measures the minimal distance between $\mathbf 0$ and the boundary of $P$.
A sequence of patterns $(P_n)_{n \geq 1}$ is said to \tdef{cover arbitrarily large disks}
if the sequence of bounded sets obtained as their orthogonal projections
onto the antidiagonal plane $\bfx_1+\bfx_2+\bfx_3 = 0$ covers arbitrarily large disks in this latter plane.
In particular, if the minimal combinatorial radius of a sequence of patterns $(P_n)_{n \geq 1}$ tends to infinity,
then the sequence of patterns $(P_n)_{n \geq 1}$ covers arbitrarily large disks centered at the origin.

\paragraph{Dual substitutions}

Let $\sigma$ be a unimodular substitution.
The \tdef{dual substitution} $\EOSS$ is defined for any face $[\bfx, i]^\star$ as
\[
\EOSS([\bfx, i]^\star) \ = \
\bigcup_{(p,j,s) \in \mcA^\star \times \mcA \times \mcA^\star \ : \ \sigma(j) = pis} [\Msinv (\bfx + \bfP(s)), j]^\star.
\]
We extend this definition to unions of faces:
$\EOSS(P\cup Q) = \EOSS(P) \cup \EOSS(Q)$.

The term ``dual'' comes from the fact that $\EOS(\sigma)$
was originally introduced in~\cite{AI01} as the dual of a geometric realization of $\sigma$ as a linear map
in an $n$-dimensional vector space (where $n$ is the size of the alphabet of $\sigma$).
This is where the formula given in the above definition comes from.
A more general setting is introduced in~\cite{SAI01},
where the linear maps $\bfE_k$ and $\bfE_k^\star$ are introduced for every $k \in \{0, \ldots, n\}$.
Intuitively, $\bfE_k(\sigma)$ acts on $k$-dimensional objects
and its dual $\bfE_k^\star(\sigma)$ acts on $(n-k)$-dimensional objects.
Some specific examples of dual substitutions are given at the end of
Sections~\ref{subsec:Brun} and ~\ref{subsec:JP}.
Note also that the notion of a dual substitution has been successfully extended to the setting of
tiling flows ($\bbR$-actions) associated with substitutions and under the formalism of strand spaces in~\cite{BK06}.

Basic properties of dual substitutions are summarized in the proposition below.
The first statement ensures that composition behaves well.
The second statement can be interpreted as a form of ``linearity'' of $\EOS$,
and allows us to specify a mapping $\EOS$ simply by providing $\Ms$
and the images of $[\mathbf 0,1]^\star, [\mathbf 0,2]^\star, [\mathbf 0,3]^\star$.
The last two statements establish fundamental links between
discrete planes and dual substitutions, which will be used throughout this paper.

\begin{prop}[\cite{AI01,Fer06}]
\label{prop:imgplane}
Let $\sigma$ be a unimodular substitution.
We have:
\begin{enumerate}
\item
$\EOS(\sigma \circ \sigma') = \EOS(\sigma') \circ \EOS(\sigma)$
for every unimodular substitution $\sigma'$;
\item
\label{imgplanestatement2}
$\EOSS([\bfx,i]^\star) = \Msinv \bfx + \EOS([\mathbf 0, i]^\star)$
for every face $[\bfx,i]^\star$;
\item
$\EOSS(\Gv) = \Gamma_{\transp \Ms \bfv}$
for every discrete plane $\Gv$;
\item
if $f$ and $g$ are distinct faces in a common discrete plane $\Gv$,
then $\EOSS(f) \cap \EOSS(g)$ contains no face.
\end{enumerate}
\end{prop}

\subsection{Brun substitutions}
\label{subsec:Brun}

Let $\bfv \in \bbR_{\geq 0}^3$ such that $0 \leq \bfv_1 \leq \bfv_2 \leq \bfv_3$.
The \tdef{Brun algorithm}~\cite{Bru58}
is one of the possible natural generalizations of Euclid's algorithm.
Together with the Jacobi-Perron algorithm, it is one of the most classical multidimensional continued fraction algorithms;
they both are defined as piecewise fractional maps according the formalism developed in~\cite{Bre81,Sch00}.
We consider here the additive version of this algorithm, which can be defined as follows in its linear form:
subtract the second largest component of $\bfv$ from the largest and reorder the result.
This yields for the algorithm in ordered form:
\[
\bfv \ \mapsto \
\begin{cases}
    (\bfv_1, \ \bfv_2, \ \bfv_3-\bfv_2)
        & \text{if} \ \bfv_1 \leq  \bfv_2 \leq \bfv_3-\bfv_2  \\
    (\bfv_1, \ \bfv_3-\bfv_2, \ \bfv_2)
        & \text{if} \ \bfv_1 \leq \bfv_3-\bfv_2 < \bfv_2  \\
    (\bfv_3-\bfv_2, \ \bfv_1, \ \bfv_2)
        & \text{if} \ \bfv_3-\bfv_2  <  \bfv_1 \leq  \bfv_2.
\end{cases}
\]
The interest of working with the additive version of the algorithm
is that we will recover a finite set of associated substitutions, as described below.
Iterating this map starting from $\bfv^{(0)} = \bfv$ yields
an infinite sequence of vectors $(\bfv^{(n)})_{n \in {\mathbb N}}$
and the algorithm can be rewritten in matrix form:
\begin{equation}
\bfv^{(n)} = \bfM_{i_n}^{-1}\bfv^{(n-1)},
\end{equation}
where $i_n \in \{1,2,3\}$ and $\bfM_1$, $\bfM_2$ and $\bfM_3$
are the corresponding non-negative integer matrices such that
$\bfM_1^{-1}$ is applied if $ \bfv_1 \leq \bfv_2 \leq \bfv_3-\bfv_2 $,
$\bfM_2^{-1}$ is applied if $\bfv_1 \leq \bfv_3-\bfv_2 < \bfv_2$,
and $\bfM_3^{-1}$ is applied otherwise.
We thus have $\bfv^{} = \bfM_{i_1} \cdots \bfM_{i_n}\bfv^{(n)}$ for all $n$.
We stress on the fact that the matrices $\bfM_i$ are non-negative,
which is needed to associate substitutions with them.

The \tdef{Brun expansion} of a vector $\bfv \in \bbR_{\geq 0}^3$
with $0 \leq \bfv_1 \leq \bfv_2 \leq \bfv_3$
is the infinite sequence $(i_n)_{n \geq 1}$ obtained above.
For example, if $\bfv = (1, \beta, \beta^2)$, with $\beta \approx 3.21$ being
the dominant Pisot cubic root of $x^3 - 3x^2 - x + 1$,
then $(i_n) = 1131~132~132\ldots$ is an eventually periodic sequence
containing infinitely many $1$s, $2$s and $3$s.



\begin{prop}[\cite{Bru58}]
\label{prop:Brun_conv}
The Brun expansion of $\bfv \in \bbR_{\geq 0}^3$ contains infinitely many $3$s if and only if $\bfv$ is totally irrational.
Moreover, for every expansion $(i_n)_{n \geq 1} \in \{1,2,3\}^\bbN$ containing infinitely many $3$s,
there is a unique vector $\bfv$ whose Brun expansion is $(i_n)_{n \geq 1}$.
\end{prop}

The \tdef{Brun substitutions} are defined by
\begin{align*}
\sBrun_1 &:
\begin{cases}
1 \mapsto 1 \\ 2 \mapsto 2 \\3 \mapsto 32
\end{cases}
&
\sBrun_2 &:
\begin{cases}
1 \mapsto 1 \\ 2 \mapsto 3 \\3 \mapsto 23
\end{cases}
&
\sBrun_3 &:
\begin{cases}
1 \mapsto 2 \\ 2 \mapsto 3 \\3 \mapsto 13
\end{cases}
\end{align*}
and we denote their associated dual substitutions by $\SBrun_i = \EOS(\sBrun_i)$ for $i \in \{1,2,3\}$.
They can be explicitly computed using the definition of dual substitutions, as shown below.
\begin{align*}
\definecolor{facecolor}{rgb}{0.8,0.8,0.8}
\SBrun_1 &: \left\{\hspace{-5pt}
    \begin{tabular}{rcl}%
    \myvcenter{%
    \begin{tikzpicture}
    [x={(-0.173205cm,-0.100000cm)}, y={(0.173205cm,-0.100000cm)}, z={(0.000000cm,0.200000cm)}]
    \fill[fill=facecolor, draw=black, shift={(0,0,0)}]
    (0, 0, 0) -- (0, 1, 0) -- (0, 1, 1) -- (0, 0, 1) -- cycle;
    \node[circle,fill=black,draw=black,minimum size=1.2mm,inner sep=0pt] at (0,0,0) {};
    \end{tikzpicture}}%
     & \myvcenter{$\mapsto$} &
    \myvcenter{%
    \begin{tikzpicture}
    [x={(-0.173205cm,-0.100000cm)}, y={(0.173205cm,-0.100000cm)}, z={(0.000000cm,0.200000cm)}]
    \fill[fill=facecolor, draw=black, shift={(0,0,0)}]
    (0, 0, 0) -- (0, 1, 0) -- (0, 1, 1) -- (0, 0, 1) -- cycle;
    \node[circle,fill=black,draw=black,minimum size=1.2mm,inner sep=0pt] at (0,0,0) {};
    \end{tikzpicture}} \\
    \myvcenter{%
    \begin{tikzpicture}
    [x={(-0.173205cm,-0.100000cm)}, y={(0.173205cm,-0.100000cm)}, z={(0.000000cm,0.200000cm)}]
    \fill[fill=facecolor, draw=black, shift={(0,0,0)}]
    (0, 0, 0) -- (0, 0, 1) -- (1, 0, 1) -- (1, 0, 0) -- cycle;
    \node[circle,fill=black,draw=black,minimum size=1.2mm,inner sep=0pt] at (0,0,0) {};
    \end{tikzpicture}}%
     & \myvcenter{$\mapsto$} &
    \myvcenter{%
    \begin{tikzpicture}
    [x={(-0.173205cm,-0.100000cm)}, y={(0.173205cm,-0.100000cm)}, z={(0.000000cm,0.200000cm)}]
    \fill[fill=facecolor, draw=black, shift={(0,0,0)}]
    (0, 0, 0) -- (0, 0, 1) -- (1, 0, 1) -- (1, 0, 0) -- cycle;
    \fill[fill=facecolor, draw=black, shift={(0,0,0)}]
    (0, 0, 0) -- (1, 0, 0) -- (1, 1, 0) -- (0, 1, 0) -- cycle;
    \node[circle,fill=black,draw=black,minimum size=1.2mm,inner sep=0pt] at (0,0,0) {};
    \end{tikzpicture}} \\
    \myvcenter{%
    \begin{tikzpicture}
    [x={(-0.173205cm,-0.100000cm)}, y={(0.173205cm,-0.100000cm)}, z={(0.000000cm,0.200000cm)}]
    \fill[fill=facecolor, draw=black, shift={(0,0,0)}]
    (0, 0, 0) -- (1, 0, 0) -- (1, 1, 0) -- (0, 1, 0) -- cycle;
    \node[circle,fill=black,draw=black,minimum size=1.2mm,inner sep=0pt] at (0,0,0) {};
    \end{tikzpicture}}%
     & \myvcenter{$\mapsto$} &
    \myvcenter{%
    \begin{tikzpicture}
    [x={(-0.173205cm,-0.100000cm)}, y={(0.173205cm,-0.100000cm)}, z={(0.000000cm,0.200000cm)}]
    \draw[thick, densely dotted] (0,0,0) -- (0,1,0);
    \fill[fill=facecolor, draw=black, shift={(0,1,0)}]
    (0, 0, 0) -- (1, 0, 0) -- (1, 1, 0) -- (0, 1, 0) -- cycle;
    \node[circle,fill=black,draw=black,minimum size=1.2mm,inner sep=0pt] at (0,0,0) {};
    \end{tikzpicture}}
    \end{tabular}
\right.
&
\SBrun_2 &: \left\{\hspace{-5pt}
    \begin{tabular}{rcl}%
    \myvcenter{%
    \begin{tikzpicture}
    [x={(-0.173205cm,-0.100000cm)}, y={(0.173205cm,-0.100000cm)}, z={(0.000000cm,0.200000cm)}]
    \fill[fill=facecolor, draw=black, shift={(0,0,0)}]
    (0, 0, 0) -- (0, 1, 0) -- (0, 1, 1) -- (0, 0, 1) -- cycle;
    \node[circle,fill=black,draw=black,minimum size=1.2mm,inner sep=0pt] at (0,0,0) {};
    \end{tikzpicture}}%
     & \myvcenter{$\mapsto$} &
    \myvcenter{%
    \begin{tikzpicture}
    [x={(-0.173205cm,-0.100000cm)}, y={(0.173205cm,-0.100000cm)}, z={(0.000000cm,0.200000cm)}]
    \fill[fill=facecolor, draw=black, shift={(0,0,0)}]
    (0, 0, 0) -- (0, 1, 0) -- (0, 1, 1) -- (0, 0, 1) -- cycle;
    \node[circle,fill=black,draw=black,minimum size=1.2mm,inner sep=0pt] at (0,0,0) {};
    \end{tikzpicture}} \\
    \myvcenter{%
    \begin{tikzpicture}
    [x={(-0.173205cm,-0.100000cm)}, y={(0.173205cm,-0.100000cm)}, z={(0.000000cm,0.200000cm)}]
    \fill[fill=facecolor, draw=black, shift={(0,0,0)}]
    (0, 0, 0) -- (0, 0, 1) -- (1, 0, 1) -- (1, 0, 0) -- cycle;
    \node[circle,fill=black,draw=black,minimum size=1.2mm,inner sep=0pt] at (0,0,0) {};
    \end{tikzpicture}}%
     & \myvcenter{$\mapsto$} &
    \myvcenter{%
    \begin{tikzpicture}
    [x={(-0.173205cm,-0.100000cm)}, y={(0.173205cm,-0.100000cm)}, z={(0.000000cm,0.200000cm)}]
    \draw[thick, densely dotted] (0,0,0) -- (0,1,0);
    \fill[fill=facecolor, draw=black, shift={(0,1,0)}]
    (0, 0, 0) -- (1, 0, 0) -- (1, 1, 0) -- (0, 1, 0) -- cycle;
    \node[circle,fill=black,draw=black,minimum size=1.2mm,inner sep=0pt] at (0,0,0) {};
    \end{tikzpicture}} \\
    \myvcenter{%
    \begin{tikzpicture}
    [x={(-0.173205cm,-0.100000cm)}, y={(0.173205cm,-0.100000cm)}, z={(0.000000cm,0.200000cm)}]
    \fill[fill=facecolor, draw=black, shift={(0,0,0)}]
    (0, 0, 0) -- (1, 0, 0) -- (1, 1, 0) -- (0, 1, 0) -- cycle;
    \node[circle,fill=black,draw=black,minimum size=1.2mm,inner sep=0pt] at (0,0,0) {};
    \end{tikzpicture}}%
     & \myvcenter{$\mapsto$} &
    \myvcenter{%
    \begin{tikzpicture}
    [x={(-0.173205cm,-0.100000cm)}, y={(0.173205cm,-0.100000cm)}, z={(0.000000cm,0.200000cm)}]
    \fill[fill=facecolor, draw=black, shift={(0,0,0)}]
    (0, 0, 0) -- (0, 0, 1) -- (1, 0, 1) -- (1, 0, 0) -- cycle;
    \fill[fill=facecolor, draw=black, shift={(0,0,0)}]
    (0, 0, 0) -- (1, 0, 0) -- (1, 1, 0) -- (0, 1, 0) -- cycle;
    \node[circle,fill=black,draw=black,minimum size=1.2mm,inner sep=0pt] at (0,0,0) {};
    \end{tikzpicture}}
    \end{tabular}
\right.
&
\SBrun_3 &: \left\{\hspace{-5pt}
    \begin{tabular}{rcl}%
    \myvcenter{%
    \begin{tikzpicture}
    [x={(-0.173205cm,-0.100000cm)}, y={(0.173205cm,-0.100000cm)}, z={(0.000000cm,0.200000cm)}]
    \fill[fill=facecolor, draw=black, shift={(0,0,0)}]
    (0, 0, 0) -- (0, 1, 0) -- (0, 1, 1) -- (0, 0, 1) -- cycle;
    \node[circle,fill=black,draw=black,minimum size=1.2mm,inner sep=0pt] at (0,0,0) {};
    \end{tikzpicture}}%
     & \myvcenter{$\mapsto$} &
    \myvcenter{%
    \begin{tikzpicture}
    [x={(-0.173205cm,-0.100000cm)}, y={(0.173205cm,-0.100000cm)}, z={(0.000000cm,0.200000cm)}]
    \draw[thick, densely dotted] (0,0,0) -- (0,1,0);
    \fill[fill=facecolor, draw=black, shift={(0,1,0)}]
    (0, 0, 0) -- (1, 0, 0) -- (1, 1, 0) -- (0, 1, 0) -- cycle;
    \node[circle,fill=black,draw=black,minimum size=1.2mm,inner sep=0pt] at (0,0,0) {};
    \end{tikzpicture}} \\
    \myvcenter{%
    \begin{tikzpicture}
    [x={(-0.173205cm,-0.100000cm)}, y={(0.173205cm,-0.100000cm)}, z={(0.000000cm,0.200000cm)}]
    \fill[fill=facecolor, draw=black, shift={(0,0,0)}]
    (0, 0, 0) -- (0, 0, 1) -- (1, 0, 1) -- (1, 0, 0) -- cycle;
    \node[circle,fill=black,draw=black,minimum size=1.2mm,inner sep=0pt] at (0,0,0) {};
    \end{tikzpicture}}%
     & \myvcenter{$\mapsto$} &
    \myvcenter{%
    \begin{tikzpicture}
    [x={(-0.173205cm,-0.100000cm)}, y={(0.173205cm,-0.100000cm)}, z={(0.000000cm,0.200000cm)}]
    \fill[fill=facecolor, draw=black, shift={(0,0,0)}]
    (0, 0, 0) -- (0, 1, 0) -- (0, 1, 1) -- (0, 0, 1) -- cycle;
    \node[circle,fill=black,draw=black,minimum size=1.2mm,inner sep=0pt] at (0,0,0) {};
    \end{tikzpicture}} \\
    \myvcenter{%
    \begin{tikzpicture}
    [x={(-0.173205cm,-0.100000cm)}, y={(0.173205cm,-0.100000cm)}, z={(0.000000cm,0.200000cm)}]
    \fill[fill=facecolor, draw=black, shift={(0,0,0)}]
    (0, 0, 0) -- (1, 0, 0) -- (1, 1, 0) -- (0, 1, 0) -- cycle;
    \node[circle,fill=black,draw=black,minimum size=1.2mm,inner sep=0pt] at (0,0,0) {};
    \end{tikzpicture}}%
     & \myvcenter{$\mapsto$} &
    \myvcenter{%
    \begin{tikzpicture}
    [x={(-0.173205cm,-0.100000cm)}, y={(0.173205cm,-0.100000cm)}, z={(0.000000cm,0.200000cm)}]
    \fill[fill=facecolor, draw=black, shift={(0,0,0)}]
    (0, 0, 0) -- (0, 0, 1) -- (1, 0, 1) -- (1, 0, 0) -- cycle;
    \fill[fill=facecolor, draw=black, shift={(0,0,0)}]
    (0, 0, 0) -- (1, 0, 0) -- (1, 1, 0) -- (0, 1, 0) -- cycle;
    \node[circle,fill=black,draw=black,minimum size=1.2mm,inner sep=0pt] at (0,0,0) {};
    \end{tikzpicture}}.
    \end{tabular}
\right.
\end{align*}
These substitutions and the Brun algorithm are linked by
$\smash{\bfM_i = \transp\bfM_{\sBrun_i}}$ for $i \in \{1,2,3\}$,
where the $\bfM_i$ are the matrices associated with the Brun algorithm above.
By virtue of Proposition~\ref{prop:imgplane}, it follows that, if $\bfv \in \bbR_{> 0}^3$
with $ 0 < \bfv_1 \leq \bfv_2 \leq \bfv_3$, then $\SBrun_i(\Gamma_\bfv) = \Gamma_{\bfM_i \bfv}$.
An explicit link between the patterns $(\Sigma_{i_1} \cdots \Sigma_{i_n}(U))_{n \geq 1}$
and the discrete plane $\Gv$ (where the expansion of $\bfv$ is $(i_n)_{n \geq 1}$)
is given in Theorem~\ref{theo:bad_Brun}.
Concrete instances of computations using dual substitutions
are found in Example~\ref{exam:badannulus},
and in many of the proofs of Section~\ref{sect:tech_proofs}.

In this paper we are interested in the finite products $\sigma = \sBrun_{i_1} \cdots \sBrun_{i_n}$
such that the dominant eigenvalue of $\Ms$ is a cubic Pisot number.
Such finite products are characterized in Proposition~\ref{prop:Brun_exp} below.
Note that such a simple characterization does not necessarily exist
for general product families of substitutions.

\begin{prop}[\cite{AD13}]
\label{prop:Brun_exp}
For every $(i_1, \ldots, i_n) \in \{1,2,3\}^n$,
the product $\sBrun_{i_1} \cdots \sBrun_{i_n}$ is irreducible Pisot
if and only if $i_k = 3$ at least once.
\end{prop}

Proposition~\ref{prop:Brun_exp} above leads us to consider only the Brun expansions $(i_n)_n \in \{1,2,3\}^{\bbN}$
in which $i_n = 3$ infinitely often.
Such sequences are called \tdef{Brun-admissible}.

\begin{rema}
The definition of Brun substitutions,
and more generally of substitutions associated with a continued fraction algorithm is not canonical,
since the matrices associated with the continued fraction algorithm do not impose a specific ordering
of the letters in the substitutions.
This specific ordering of letters is the one that yields the simplest technical proofs in Section~\ref{sect:tech_proofs},
but, in the case of Brun substitutions, the same results can be obtained for any other ordering using the same techniques.
\end{rema}

\subsection{Jacobi-Perron substitutions}
\label{subsec:JP}

Let $\bfv \in \bbR_{\geq 0}^3$ be such that $\bfv_1 \leq \bfv_3$ and $\bfv_2 \leq \bfv_3$.
The \tdef{Jacobi-Perron algorithm}~\cite{Jac68, Per07} consists in iterating the map
\[
\bfv \ \mapsto \ (\bfv_2 - a \bfv_1, \ \bfv_3 - b \bfv_1, \ \bfv_1), \
\text{ where } a = \lfloor \bfv_2/\bfv_1 \rfloor, \
\text{ and } b = \lfloor \bfv_3/\bfv_1 \rfloor.
\]
Like with the Brun algorithm, we obtain an infinite sequence of vectors $\bfv^{(0)} = \bfv, \bfv^{(1)}, \bfv^{(2)}, \ldots$
such that $\smash{\bfv^{(n)} = \bfM_{a_n, b_n}^{-1}\bfv^{(n-1)}}$
where $\bfM_{a, b}$ is a matrix with non-negative integer entries.
The \tdef{Jacobi-Perron expansion}
of $\bfv$ is the infinite sequence $(a_n, b_n)_{n\geq 1}$.
It can be proved that $(a_n, b_n)_{n\geq 1}$
is the Jacobi-Perron expansion of some vector $\bfv$ if and only if
for every $n \geq 1$, we have $0 \leq a_n \leq b_n$, $b_n \neq 0$,
and $a_n = b_n$ implies $a_{n+1} \neq 0$ (see~\cite{Sch73,Sch00,Bre81}).

A sequence $(a_n, b_n)_{n\geq 1}$
is \tdef{Jacobi-Perron-admissible}
if, for every $n \geq 1$, we have $0 \leq a_n \leq b_n$, $b_n \neq 0$,
and $a_n = b_n$ implies $a_{n+1} \neq 0$.
Here again, we can state a convenient characterization of expansions of totally irrational vectors.

\begin{prop}[\cite{Per07}]
\label{prop:JP_conv}
A sequence $(a_n, b_n)_{n\geq 1}$ is a Jacobi-Perron-admissible sequence
if and only if it is the Jacobi-Perron expansion of totally irrational vector $\bfv \in \bbR_{\geq 0}^3$.
Moreover, for every such expansion $(a_n,b_n)_{n \geq 1}$,
there is a unique vector $\bfv$ whose expansion is $(a_n,b_n)_{n \geq 1}$.
\end{prop}

Let the
\tdef{Jacobi-Perron substitutions}
be defined by $\sJP_{a,b} : 1 \mapsto 3, 2 \mapsto 13^a, 3 \mapsto 23^b$ for all $a, b \geq 0$.
Their associated duals are defined by $\SJP_{a,b} = \EOS(\sJP_{a,b})$ and are given by
\[
\SJP_{a,b} :
\left\{
\begin{array}{rcl}
\left[\mathbf 0,1\right]^\star & \mapsto & [a\bfe_1,2]^\star \\
\left[\mathbf 0,2\right]^\star & \mapsto & [b\bfe_1,3]^\star \\
\left[\mathbf 0,3\right]^\star & \mapsto & [\mathbf 0,1]^\star
    \ \cup \ \bigcup_{k=0}^{a-1} [k\bfe_1,2]^\star
    \ \cup \ \bigcup_{k=0}^{b-1} [k\bfe_1,3]^\star.
\end{array}
\right.
\]
These substitutions are chosen in such a way that
$\smash{\bfM_{a,b} = \transp\bfM_{\sJP_{a,b}}}$ for every $a,b \geq 0$,
so the correspondence with dual substitutions and discrete planes
is analogous to the one described above for the Brun substitutions.
We also have:

\begin{prop}[\cite{DFP04}]
\label{prop:JP_exp}
Every product $\sJP_{a_1,b_1} \cdots \sJP_{a_n,b_n}$ is irreducible Pisot
if $0 \leq a_n \leq b_n$ and $b_n \neq 0$ for all $n \geq 1$.
\end{prop}

Note that, unlike Brun substitutions,
Jacobi-Perron substitutions do not consist of a finite number of substitutions,
but of infinitely many substitutions parametrized by $a,b \in \bbN$.
In essence this reflects the fact that the associated continued fraction algorithm
is \emph{multiplicative}, in opposition with the \emph{additive} version of
the Brun algorithm given in Section~\ref{subsec:Brun}.
Such notions are defined formally in~\cite{Sch00}.

\paragraph{Additive version of Jacobi-Perron substitutions}
When dealing with Jacobi-Perron substitutions in Section~\ref{sect:gen},
we will work directly on the substitutions $\SJP_{a,b}$,
except in Section~\ref{subsec:graph_JP},
where the generation graphs that we will construct apply only to products
of \emph{finitely many} substitutions
(but there are infinitely many substitutions $\sJP_{a,b}$).
Hence we will need to decompose $\sJP_{a,b}$ in products using finitely many different substitutions.
A first natural additive substitutive realization of Jacobi-Perron algorithm is given by
$\tau_1 : 1 \mapsto 1, 2 \mapsto 21, 3 \mapsto 3$,
$\tau_2 : 1 \mapsto 1, 2 \mapsto 2, 3 \mapsto 31$,
$\tau_3 : 1 \mapsto 3, 2 \mapsto 1, 3 \mapsto 2$,
with $\sJP_{a,b}= \tau_3 \tau_1^a \tau_2 ^b$.
(Note that $\tau_1\tau_2 = \tau_2 \tau_1$.)
The problem with this decomposition is that it is not restrictive enough,
which makes the generation graphs associated with this family too difficult to handle.

So we choose another additive decomposition,
which enforces the constraint $0 \leq a \leq b$ and $ b \geq 1$
to be satisfied in the additive products.
By setting
$\theta_1 = \tau_2$,
$\theta_2 = \tau_1 \tau_2$,
$\theta_3 = \tau_3 \tau_2$ and
$\theta_4 = \tau_3 \tau_1 \tau_2$,
we have
\[
\sJP_{a,b} =
\begin{cases}
\theta_3\theta_1^{b-1}
    & \text{ if } a = 0 \\
\theta_3 \theta_1^{b-a-1} \theta_2^{a} \ = \  \theta_4 \theta_1^{b-a} \theta_2^{a-1}
    & \text{ if } 0 < a < b \\
\theta_4 \theta_2^{a-1}
    & \text{ if } a = b.
\end{cases}
\]
Finally, we define $\Theta_i = \EOS(\theta_i)$ for $i \in \{1,2,3,4\}$.
These are the substitutions that will be used in Section~\ref{subsec:generation_graphs}
to construct generation graphs for the Jacobi-Perron substitutions.

\begin{rema} \label{rema:JPAdditive}
Note that the ``rhythm'' is provided by $\theta_3$ and $\theta_4$,
in the sense that every product $\theta_{i_1} \cdots \theta_{i_n}$ starting with $\theta_3$ or $\theta_4$
can be uniquely decomposed into a product of $\sJP_{a_k,b_k}$.
Consequently, if an infinite sequence $(i_n)_{n \geq 1} \in \{1,2,3,4\}^\bbN $
is the additive expansion of an admissible Jacobi-Perron expansion,
then it must contain infinitely many occurrences of $3$ and $4$.
\end{rema}

\subsection{Substitutive dynamics and symbolic codings}\label{subsec:dynamics}

\paragraph{Pure discrete spectrum}
We endow $\mcA^\bbZ$ with the product of discrete topologies.
Let $S$ stand for the (two-sided) shift map over $\mcA^\bbZ$,
that is, $S((u_n)_{n \in \bbZ}) = (u_{n+1})_{n \in \bbZ}$.
Let $\sigma$ be a primitive substitution over $\mcA$
and let $u \in \mcA^\bbZ$ be such that $\sigma^{k} (u)=u$ for some $k \geq 1$
and such that $u_{-1}u_0$ is a factor of some $\sigma^\ell(i)$, $i \in \mcA$
(such a word $u$ exists by primitivity of $\sigma$).
Let $\overline{\mcO(u)}$ be the orbit closure of the two-sided word $u$ under the action of the shift $S$.
The substitutive symbolic dynamical system $(X_\sigma, S)$
generated by $\sigma$ is defined as $X_\sigma = \overline{\mcO(u)}$.
One easily checks by primitivity that
$(X_\sigma, S)$ does not depend on the choice of the two-sided word $u$ fixed by some power of $\sigma$.
For more details, see~\cite{Que10,Fog02}.

Let $\mu$ be a shift-invariant Borel probability measure for $(X_\sigma,S)$.
An eigenfunction for $(X_\sigma, S,\mu)$ is an $L^2(X_\sigma,\mu)$ function $f$
for which there is an associated eigenvalue $\alpha\in\bbC$ (with modulus $1$) such that $f \circ S=\alpha f$.
The symbolic dynamical system $(X_\sigma, S)$ is said to have \tdef{pure discrete spectrum}
if the linear span of the eigenfunctions is dense in $L^2(X_\sigma,\mu)$.
Note that measure-theoretic discrete spectrum and topological
discrete spectrum are proved to be equivalent for primitive substitutive dynamical systems~\cite{hos2}.

\paragraph{Rauzy fractals}
Let $\sigma$ be a unimodular irreducible Pisot substitution with incidence matrix $\Ms$.
We denote by $\pic$ the projection onto the contracting eigenhyperplane of $\Ms$ along the expanding eigendirection of $\Ms$.
The \tdef{Rauzy fractal} $\mcT_\sigma$ associated with $\sigma$ is the Hausdorff limit
of the sequence $(\mcD_n)_{n \geq 0}$ with $\mcD_n = \Ms^n \circ \pic \circ \EOSS^n(\mcU)$ for all $n$,
with $\mcU = [\mathbf 0,1]^\star \cup [\mathbf 0,2]^\star \cup [\mathbf 0,3]^\star = \input{fig/U.tex}$.
It is divided into the \tdef{subtiles} $\mcT_\sigma(i)$,
for $i \in \mcA$, provided by the Hausdorff limit of the sequence of compact sets
$\mcD_n(i) = \Ms^n \circ \pic \circ \EOSS^n([\mathbf 0,i]^\star)$ for $n \geq 0$.
Examples of Rauzy fractals are given in Figure~\ref{fig:rfbrun_zero}~p.~\pageref{fig:rfbrun_zero}.

A Rauzy fractal provides a geometrical realization for the substitutive dynamics
of a unimodular irreducible Pisot substitution $\sigma$.
First, a domain exchange transformation $(\mcT_\sigma,E)$ can be defined acting on the three subtiles of the Rauzy fractal;
it is defined by translating each subtile $\mcT_\sigma(i)$ by the vector $\pic (\bfe_i)$.
A combinatorial condition called the strong coincidence condition
guarantees that the subtiles $\mcT_\sigma(i)$ are disjoint in measure~\cite{AI01}.
The domain exchange transformation $(\mcT_\sigma,E)$ is proved to
be measure-theoretically conjugate to $(X_{\sigma}, S)$ if $\sigma$ satisfies the strong coincidence condition.
Furthermore this domain exchange transformation is conjectured to factor onto a toral translation,
according to the Pisot conjecture.
The underlying lattice for the factorization is given by the vectors $\pic(\bfe_i - \bfe_j)$ for $i\neq j$.
This can be reformulated in terms of tilings:
$(\mcT_\sigma,E)$ is measure-theoretically conjugate to a translation (via this factorization)
if and only if the translates by the vectors of the lattice~$\pic (\bbZ^2)$ of the subtiles of the Rauzy fractal
tile the contracting eigenhyperplane of $\Ms$. For more details, see~\cite{AI01,CS01}.
Second, when pure discrete spectrum holds, the Rauzy fractal and its subtiles
can be used to construct a Markov partition of the toral automorphism provided by
the incidence matrix $\Ms$ of~$\sigma$ (see~\cite{IO93,KV98,Pra99,IR06,Sie00}).
This will de developed in Section~\ref{subsec:applis_dyn}.

\paragraph{Symbolic natural codings and bounded remained sets}
Given an invertible transformation $T$ on a set $X$ and a partition $\mcP = \{P_0, \ldots, P_{k-1}\}$ of $X$,
a \tdef{coding} of $x \in X$ is a sequence $u \in \{0,\cdots, k-1\}^\bbZ$ such that $u_n=i$
if $T^nx \in P_i$, for all $n\in \bbZ$.
A symbolic dynamical system $(\Omega, S)$ is a \tdef{symbolic coding} of $(X,T)$ if there exists a
finite partition of $X$ such that every element of $\Omega$ is a coding of the orbit of some point of $X$,
and if, furthermore, $(\Omega,S)$ and $(X,T)$ are measure-theoretically conjugate.
A \tdef{toral translation} on $\bbT ^2$
is a map $R_\alpha: \bbT^2 /L \rightarrow \bbR ^2 / L$, $x\mapsto x + \alpha \pmod L$,
where $\alpha \in \bbR ^2$, and where $L$ is a full rank lattice in $\bbR ^2$.
A symbolic dynamical system $(\Omega,S)$ is a \tdef{symbolic natural coding} of $(\bbT^2 / L, R_\alpha)$
if it is a symbolic coding defined with respect to a fundamental domain for the lattice $L$ in $\bbR^2$
together with a finite partition of this fundamental domain such that
the restriction of the map $R_\alpha$ is provided by the translation by some vector on each element of the partition.
Lastly, a set $ Y \subset X$ is called a \tdef{bounded remainder set} for the minimal dynamical system $(X,T,\mu)$
if its indicator function yields bounded Birkhoff sums, that is,
if there exists $C >0$ such that for a.e. $x \in X$, we have
$
|\#\{n \in \bbZ : \, |n| \leq N, \ T^n (x) \in Y \} - (2N+1) \mu (Y) | \leq C.
$
In case of pure discrete spectrum of $(X_{\sigma}, S)$, the Rauzy subtiles $\mcT(i)$ provide bounded remainder sets
for the associated toral translation (for more details, see the proof of Corollary~\ref{coro:dynprod}).

\section{Generating discrete planes with substitutions}
\label{sect:gen}

In this section we establish a generic strategy to generate discrete planes with substitutions
associated with multidimensional continued fraction algorithms,
in the sense that this strategy can be applied to
other product families of substitutions and continued fraction algorithms.
We have in mind fibred algorithms in the sense of~\cite{Schweiger95},
which are Markovian (``without memory'') and which are defined by piecewise fractional maps
such as considered in~\cite{Bre81,Sch00}.
The underlying matrices are thus nonnegative unimodular matrices, so we can use the formalism of dual substitutions.

A detailed outline of the strategy is first given in Section~\ref{subsec:mcfgen_strategy}.
In Section~\ref{subsec:coverings} we introduce \emph{strong coverings},
which serve as a combinatorial restriction on the patches of discrete planes, and
which are crucial in order to prove the results of Section~\ref{subsec:annulus_property},
where we give sufficient combinatorial criteria to establish the \emph{annulus property}.
In Section~\ref{subsec:generation_graphs}
we construct \emph{generation graphs} to characterize the sequences of substitutions
that fail to generate an entire discrete plane when starting from a finite pattern.

\subsection{Outline of the strategy}
\label{subsec:mcfgen_strategy}

We recall that a \tdef{pattern} is a finite union of faces
and that a \tdef{seed} is a pattern $V$
such that iterating any sequence of respectively Brun or Jacobi-Perron dual substitutions from $V$
yields patterns with arbitrarily large minimal combinatorial radii centered at the origin.
We will provide adequate tools in order to prove
Theorem~\ref{theo:seeds} (existence of seeds),
Theorems~\ref{theo:bad_Brun} and~\ref{theo:bad_JP} (characterization of sequences of which $U$ is not a seed),
and Theorem~\ref{theo:balls} (iterating from $\mcU$ generates translates of arbitrarily large disks).

This is illustrated by the following example for Brun substitutions.
Let $(i_n) = 232~232~\ldots$
and $(j_n) = 2311~2311~\ldots$
be two infinite (periodic) sequences.
Figure~\ref{fig:gen_goodbad} (left)
suggests that $\smash{\SBrun_{i_1} \cdots \SBrun_{i_n}(\mcU)}$
covers arbitrarily large disks centered at the origin as $n \rightarrow \infty$,
but Figure~\ref{fig:gen_goodbad} (right)
suggests that $\smash{\SBrun_{j_1} \cdots \SBrun_{j_n}(\mcU)}$
does \textbf{not} cover arbitrarily large disks centered at the origin as $n \rightarrow \infty$.
This can be proved rigorously thanks to Theorem~\ref{theo:bad_JP},
and Theorem~\ref{theo:balls} guarantees that \emph{translates} of arbitrarily large disks do occur.

\begin{figure}[ht]
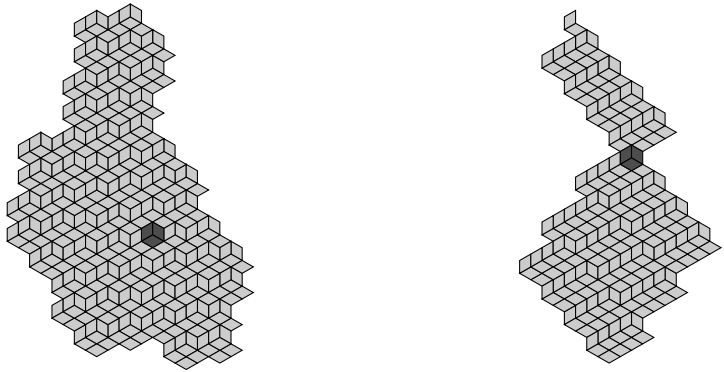

\centering
\myvcenter{\input{fig/iter_brun_232.tex}}
\hfil
\myvcenter{\input{fig/iter_brun_2311.tex}}
\caption{%
    On the left, the pattern $(\SBrun_2 \SBrun_3 \SBrun_2)^4(\mcU)$.
    On the right, the pattern $(\SBrun_2 \SBrun_3 \SBrun_1 \SBrun_1)^3(\mcU)$.
    The pattern $\mcU$ at the origin is shown in dark gray.
    The Rauzy fractals associated with these two products of substitutions
    are plotted in Figure~\ref{fig:rfbrun_zero}~p.~\pageref{fig:rfbrun_zero}.
    We will see that in both cases, the arbitrarily large disks covering property is satisfied.
    }
\label{fig:gen_goodbad}
\end{figure}

\paragraph{Annuli and covering properties}

The patterns generated by iterating dual substitutions can have complicated shapes,
so there is no obvious way of proving that such a sequence covers arbitrarily large disks.
One approach to prove this property in the context of dual substitutions
has been initiated by Ito and Ohtsuki~\cite{IO94}.
The idea is to make sure that the generated patterns
contain an increasing number of nested concentric annuli of positive width,
and hence cover arbitrarily large disks.
This requires that the \emph{annulus property} holds,
\emph{i.e.}, that the image of an annulus by a dual substitution remains an annulus.

The action of a dual substitution is governed mainly by the action of the inverse of the incidence matrix
of the substitution; roughly speaking, a dual substitution can be seen as a discrete analogue of an affine transformation.
Affine transformations preserve the topological property of being an annulus.
However the discretization step makes things more complicated,
and there exist annuli not preserved under the action of a dual substitution:
the annulus property is wrong if no additional assumptions are made on the annuli,
as shown in Example~\ref{exam:badannulus}.

One way around this is the notion of $\mcL$-covering (introduced in~\cite{IO94}),
which intuitively means that a pattern is path-connected by paths of contiguous connected patterns (the elements of $\mcL$),
which provides some combinatorial restrictions on the annuli.
Contrary to as stated in~\cite{IO94}, $\mcL$-coverings alone are not sufficient
to ensure that the image of an annulus remains an annulus.
We thus have extended this notion by considering patterns of more than two faces
to be able to find suitable sets $\mcL$ for the substitutions we study.
(In particular, the set of $7$ patterns given in~\cite{IO94} for Jacobi-Perron substitutions is not sufficient.)

We introduce and discuss covering properties in Section~\ref{subsec:coverings},
as well as the stronger condition (strong $\mcL$-covering) under which the desired annulus property holds.
These restrictions allow us to prove the annulus property in Section~\ref{subsec:annulus_property}.
The proofs require many verifications which are specific to the set of substitutions under study.
Hence the results given in Sections~\ref{subsec:coverings} and~\ref{subsec:annulus_property}
are stated in a generic way (they do not depend on the choice of substitutions),
and proving that they hold for our families of substitutions is reduced to technical verifications,
which are carried out in Section~\ref{sect:tech_proofs}.

\paragraph{Generation graphs}
The annulus property alone is not sufficient:
we must also make sure that, starting from an initial pattern expected to be a candidate seed,
at least one annulus is eventually generated by iterating sufficiently many substitutions.

A first natural strategy to check this is to track all the possible images of the initial pattern $\mcU$
under iteration of admissible products,
and prove that an annulus centered at the origin is always eventually generated around $\mcU$.
This is the strategy which was originally taken in~\cite{IO94} with the Jacobi-Perron substitutions,
by constructing graphs whose vertices are patterns.
The problem with this approach is that the obtained graphs are too large and difficult to analyze,
and require the proof of numerous lemmas.
In some cases, like with the Brun substitutions,
this approach is impracticable;
this is mostly due to the fact that
for some admissible sequences of Brun (or Jacobi-Perron),
the pattern pattern $\mcU$ is not sufficient to generate an entire discrete plane,
as illustrated in Figure~\ref{fig:gen_goodbad} (right).

Hence we introduce in Section~\ref{subsec:generation_graphs} a new tool, \emph{generation graphs},
which are much more manageable (the vertices are single faces and not patterns),
which can be constructed algorithmically,
and which contain the information needed about the annulus generation properties around seeds.
This prevents us from having to deal with numerous cases by hand
(in contrast with the graphs described in~\cite{IO94} for example).
The aim of generation graphs is to track
all the possible sequences of preimages of each \emph{face} of the annuli that we want to generate,
instead of tracking all the possible patterns around a given initial pattern.
The construction given in Section~\ref{subsec:generation_graphs} is generic,
and specific generation graphs are computed
in Sections~\ref{subsec:graph_Brun} and~\ref{subsec:graph_JP}
for the Brun and Jacobi-Perron substitutions.

\begin{rema}
Given a \emph{single} unimodular irreducible Pisot substitution $\sigma$,
it is not difficult to see that there exists a finite seed $V$ such that iterating $\EOS(\sigma)$ from $V$
generates an entire discrete plane.
This follows from the fact that $\Ms$
is contracting on the contracting plane of $\Ms$
thanks to the Pisot condition,
so the discrete affine map $\EOS(\sigma)$ is also contracting on the associated (invariant) discrete plane
(for a proof, see~\cite{CANT}).
In the case of a product family of substitutions,
the existence of such finite seeds $V$ is not guaranteed,
especially if some of the substitutions are not Pisot
(for example $\sBrun_3$ is Pisot, but $\sBrun_1$ and $\sBrun_2$ are not).
Constructing generation graphs will allow us to prove the existence of seeds even in such cases.
\end{rema}

\subsection{Covering properties}
\label{subsec:coverings}

We now introduce $\mcL$-coverings and strong $\mcL$-coverings,
which are the combinatorial tools that will be used in order to prove
the annulus property in Section~\ref{subsec:annulus_property}.

A pattern is said to be \tdef{edge-connected}
if any two faces are connected by a path of faces $f_1, \ldots, f_n$
such that $f_k$ and $f_{k+1}$ share an edge, for all $k \in \{1, \ldots, n-1\}$.
In the definitions below, $\mcL$ will always denote a set of patterns
which is closed by translation of $\bbZ^3$,
so we will define such sets by giving only one element of each translation class.
Let $\mcL$ be a set of patterns.
A pattern $P$ is \tdef{$\mcL$-covered}
if for all faces $e, f \in P$,
there exist patterns $Q_1, \ldots, Q_n \in \mcL$ such that
\begin{enumerate}
  \item\label{item:cov1} $e \in Q_1$ and $f \in Q_n$;
  \item\label{item:cov2} $Q_k \cap Q_{k+1}$ contains at least one face,
    for all $k \in \{1, \ldots, n-1\}$;
  \item\label{item:cov3} $Q_k \subseteq P$ for all $k \in \{1, \ldots, n\}$.
\end{enumerate}

The next proposition, due to~\cite{IO94},
gives a sufficient combinatorial criterion to prove
that a dual substitution preserves $\mcL$-covering.

\begin{prop}
\label{prop:coverprop}
Let $\mcL$ be a set of patterns,
$P$ be an $\mcL$-covered pattern
and $\Sigma$ be a dual substitution.
If $\Sigma(Q)$ is $\mcL$-covered for every $Q \in \mcL$,
then $\Sigma(P)$ is $\mcL$-covered.
\end{prop}

We will see in Example~\ref{exam:badannulus}
that $\mcL$-coverings are not sufficient for our purposes,
so we introduce \emph{strong} $\mcL$-coverings.
A pattern $P$ is \tdef{strongly $\mcL$-covered} if $P$ is $\mcL$-covered and if
for every pattern $X \subseteq P$ that is edge-connected and consists of two faces,
there exists a pattern $Y \in \mcL$ such that $X \subseteq Y \subseteq P$.

\begin{exam}
Let $\smash{\LBrun =
\{
\input{fig/E12b.tex}     ,
\input{fig/E13b.tex}     ,
\input{fig/E23a.tex}     ,
\input{fig/E23b.tex}     ,
\input{fig/E33a.tex}     ,
\input{fig/E33b.tex}     ,
\input{fig/U.tex}        ,
\input{fig/E22aBrun.tex}
\}}$
(a set of patterns which will later be used with the Brun substitutions in Section~\ref{sect:tech_proofs}),
and let
$P_1 =
\myvcenter{%
\begin{tikzpicture}
[x={(-0.173205cm,-0.100000cm)}, y={(0.173205cm,-0.100000cm)}, z={(0.000000cm,0.200000cm)}]
\definecolor{facecolor}{rgb}{0.800,0.800,0.800}
\fill[fill=facecolor, draw=black, shift={(1,-1,0)}]
(0, 0, 0) -- (1, 0, 0) -- (1, 1, 0) -- (0, 1, 0) -- cycle;
\fill[fill=facecolor, draw=black, shift={(0,0,0)}]
(0, 0, 0) -- (1, 0, 0) -- (1, 1, 0) -- (0, 1, 0) -- cycle;
\end{tikzpicture}}
$\,,
$P_2 =
\myvcenter{%
\begin{tikzpicture}
[x={(-0.173205cm,-0.100000cm)}, y={(0.173205cm,-0.100000cm)}, z={(0.000000cm,0.200000cm)}]
\definecolor{facecolor}{rgb}{0.800,0.800,0.800}
\fill[fill=facecolor, draw=black, shift={(0,0,0)}]
(0, 0, 0) -- (0, 0, 1) -- (1, 0, 1) -- (1, 0, 0) -- cycle;
\fill[fill=facecolor, draw=black, shift={(-1,0,0)}]
(0, 0, 0) -- (1, 0, 0) -- (1, 1, 0) -- (0, 1, 0) -- cycle;
\fill[fill=facecolor, draw=black, shift={(1,0,0)}]
(0, 0, 0) -- (0, 0, 1) -- (1, 0, 1) -- (1, 0, 0) -- cycle;
\fill[fill=facecolor, draw=black, shift={(0,0,0)}]
(0, 0, 0) -- (1, 0, 0) -- (1, 1, 0) -- (0, 1, 0) -- cycle;
\end{tikzpicture}}
$\,,
$P_3 =
\smash{\myvcenter{%
\begin{tikzpicture}
[x={(-0.173205cm,-0.100000cm)}, y={(0.173205cm,-0.100000cm)}, z={(0.000000cm,0.200000cm)}]
\definecolor{facecolor}{rgb}{0.800,0.800,0.800}
\fill[fill=facecolor, draw=black, shift={(0,-1,1)}]
(0, 0, 0) -- (1, 0, 0) -- (1, 1, 0) -- (0, 1, 0) -- cycle;
\fill[fill=facecolor, draw=black, shift={(0,0,0)}]
(0, 0, 0) -- (0, 0, 1) -- (1, 0, 1) -- (1, 0, 0) -- cycle;
\fill[fill=facecolor, draw=black, shift={(0,0,0)}]
(0, 0, 0) -- (0, 1, 0) -- (0, 1, 1) -- (0, 0, 1) -- cycle;
\fill[fill=facecolor, draw=black, shift={(-1,0,1)}]
(0, 0, 0) -- (1, 0, 0) -- (1, 1, 0) -- (0, 1, 0) -- cycle;
\fill[fill=facecolor, draw=black, shift={(-1,-1,1)}]
(0, 0, 0) -- (1, 0, 0) -- (1, 1, 0) -- (0, 1, 0) -- cycle;
\end{tikzpicture}}}
$\,,
$P_4 =
\smash{\myvcenter{%
\begin{tikzpicture}
[x={(-0.173205cm,-0.100000cm)}, y={(0.173205cm,-0.100000cm)}, z={(0.000000cm,0.200000cm)}]
\definecolor{facecolor}{rgb}{0.800,0.800,0.800}
\fill[fill=facecolor, draw=black, shift={(-1,-1,1)}]
(0, 0, 0) -- (1, 0, 0) -- (1, 1, 0) -- (0, 1, 0) -- cycle;
\fill[fill=facecolor, draw=black, shift={(-1,0,1)}]
(0, 0, 0) -- (1, 0, 0) -- (1, 1, 0) -- (0, 1, 0) -- cycle;
\fill[fill=facecolor, draw=black, shift={(0,0,0)}]
(0, 0, 0) -- (0, 0, 1) -- (1, 0, 1) -- (1, 0, 0) -- cycle;
\fill[fill=facecolor, draw=black, shift={(0,0,0)}]
(0, 0, 0) -- (1, 0, 0) -- (1, 1, 0) -- (0, 1, 0) -- cycle;
\fill[fill=facecolor, draw=black, shift={(0,-1,1)}]
(0, 0, 0) -- (1, 0, 0) -- (1, 1, 0) -- (0, 1, 0) -- cycle;
\fill[fill=facecolor, draw=black, shift={(0,0,0)}]
(0, 0, 0) -- (0, 1, 0) -- (0, 1, 1) -- (0, 0, 1) -- cycle;
\end{tikzpicture}}}
$\,.
Then
$P_1$ is neither $\LBrun$- nor $\LJP$-covered;
$P_2$ is not $\LBrun$-covered, but it is strongly $\LJP$-covered;
$P_3$ is $\LBrun$- and $\LJP$-covered, but not strongly covered.;
and $P_4$ is strongly $\LBrun$- and $\LJP$-covered.
\end{exam}

\subsection{The annulus property}
\label{subsec:annulus_property}

In this section we define $\mcL$-annuli
and we introduce Property~A,
which is a sufficient condition under which
we can prove that a dual substitution $\Sigma$
verifies the \tdef{annulus property}:
if $A$ is an $\mcL$-annulus of $P$,
then $\Sigma(A)$ is an $\mcL$-annulus of $\Sigma(P)$
(Proposition~\ref{prop:annulus_induction}).

The \tdef{boundary} $\partial P$ of a pattern $P$
is the union of the edges $e$ of the faces of $f$
such that $e$ is contained in one face only.
An \tdef{$\mcL$-annulus}
of a simply connected pattern $P$
is a pattern $A$ such that
$A$ is strongly $\mcL$-covered,
$A$ and $P$ have no face in common,
and $P \cap \partial(P \cup A) = \varnothing$.
Note that the condition $P \cap \partial(P \cup A) = \varnothing$
in the above definition is a concise way to express the fact that
the $\mcL$-covered set $A$ is a ``good surrounding'' of $\mcU$.

\begin{exam}
Let $A_1$, $A_2$, $A_3$ and $A_4$ be defined by
\[
\mcU \cup A_1 = \input{fig/VBrun2_broken.tex}
\quad
\mcU \cup A_2 = \input{fig/VBrun2_broken2.tex}
\quad
\mcU \cup A_3 = \input{fig/VBrun2_mod.tex}
\quad
\mcU \cup A_4 = \input{fig/VBrun2.tex}\,,
\]
where $\mcU$ is depicted in dark gray.
Then,
$A_1$ is not an annulus of $\mcU$ because
$\mcU \cap \partial(\mcU \cup A_1)$ is non-empty (it contains an edge);
$A_2$ is not an annulus of $\mcU$ because
$\mcU \cap \partial(\mcU \cup A_2)$ is non-empty (it contains a point);
$A_3$ is not an $\LBrun$-annulus of $\mcU$ because it is not strongly $\LBrun$-covered;
indeed, if $X = \input{fig/E12a.tex} \subseteq A_3$
is the pattern depicted in white in the picture,
there does not exist a pattern $Y \in \LBrun$ such that $X \subseteq Y \subseteq A_3$;
and $A_4$ is an $\LBrun$-annulus of $\mcU$.
\end{exam}

\begin{exam}
\label{exam:badannulus}
Let $P$ be a pattern that consists of the single face $[\mathbf 0,3]^\star$ (shown in dark gray below)
and let $A$ be the set of faces surrounding $P$ defined in the picture below.
\[
A \cup P = \input{fig/badannulus.tex}
\qquad
\SBrun_1 \SBrun_3(A \cup P) = \hspace{-1em}\input{fig/badannulus_Brun.tex}
\qquad
\SJP_{0,2}(A \cup P) = \hspace{-1em}\input{fig/badannulus_JP.tex}
\]
The pattern $A$ is both $\LBrun$- and $\LJP$-covered, but \emph{not} strongly (so it is not a valid annulus).
Its images by Brun or Jacobi-Perron substitutions fail to be topological annuli (as shown above),
which illustrates the necessity of \emph{strong} coverings
if we want the image of an annulus to remain an annulus.
The faulty two-face pattern and its images are shown in white.
\end{exam}

\begin{defi}
\label{defi:propA}
Let $\Sigma$ be a dual substitution and let $\mcL$ a set of edge-connected patterns.
We say that \tdef{Property~A} holds for $\Sigma$ with respect to $\mcL$ when restricted to a family of discrete planes if
for every connected two-face pattern $f \cup g$ of this family of discrete planes
and for every disconnected two-face pattern $f_0 \cup g_0$ such that
$f \in \Sigma(f_0)$ and $g \in \Sigma(g_0)$, the following holds:
there do not exist a pattern $P$ and an $\mcL$-annulus $A$ of $P$
(both included in a common discrete plane $\Gamma$ of the given family of discrete planes)
such that $f_0 \in P$ and $g_0 \in \Gamma \setminus (A \cup P)$.
\end{defi}

The interest of the above technical definition of Property~A
becomes apparent in the proof of Proposition~\ref{prop:annulus_induction} below,
thanks to which Property~A (and some $\mcL$-covering assumptions)
provide sufficient conditions for the annulus property.
The main interest of Property~A is that it can be checked by handling finitely many cases, namely
by enumerating all the two-face connected patterns $f \cup g$ that admit a disconnected preimage.
This is done for the Brun and Jacobi-Perron substitutions in Proposition
\ref{prop:cov_Brun} and~\ref{prop:cov_JP}.

\begin{prop}
\label{prop:annulus_induction}
Let $\Sigma$ be a dual substitution and $\mcL$ be a set of edge-connected patterns
such that Property~A holds for $\Sigma$ with respect to $\mcL$ and to a given family of discrete planes.
Assume that the image by $\Sigma$ of every strongly $\mcL$-covered pattern is strongly $\mcL$-covered.
Let $P$ be a pattern and $A$ be an $\mcL$-annulus of $P$,
both included in a common discrete plane of the given family.
Then $\Sigma(A)$ is an $\mcL$-annulus of $\Sigma(P)$.
\end{prop}

\begin{proof}
The pattern $A$ is strongly $\mcL$-covered because it is an $\mcL$-annulus,
so $\Sigma(A)$ is also strongly $\mcL$-covered, by assumption.
It remains to show that $\Sigma(P) \cap \partial(\Sigma(P) \cup \Sigma(A)) = \varnothing$.
Suppose the contrary.
This means that there exist faces $f,g,f_0,g_0$ such that
$f \in \Sigma(f_0)$, $g \in \Sigma(g_0)$,
$f \cup g$ is connected, and $f_0 \cup g_0$ is disconnected,
as shown below.
\[
\myvcenter{\input{fig/lemcase0.tex}} \quad \overset{\Sigma}{\longmapsto} \quad
\myvcenter{\input{fig/lemcase1.tex}}
\]
These are precisely the conditions stated in Property~A,
so such a situation cannot occur and the proposition holds.
\end{proof}

The annulus property will be crucially used in the proof of Theorem~\ref{theo:seeds}:
it ensures that the minimal combinatorial radius is strictly increasing under iteration of dual substitutions.

\begin{rema}
\label{rema:arith_assump}
In the definition of Property~A, we consider a restriction to a given family of planes,
because it is a natural assumption for the sets of substitutions that we are studying.
Indeed, for Brun substitutions, vectors $\bfv$ that are expanded satisfy $ 0 < \bfv_1< \bfv_2< \bfv_3$,
and for Jacobi-Perron substitutions, vectors $\bfv$ that are expanded satisfy $ 0 < \bfv_1 < \bfv_3$ and $ 0 < \bfv_2< \bfv_3$.
These restrictions allow us to significantly simplify the proofs of Section~\ref{sect:tech_proofs}.
\end{rema}

\subsection{Generation graphs}
\label{subsec:generation_graphs}
Let $\Sigma_1, \ldots, \Sigma_\ell$ be dual substitutions,
let $\mcV$ be a finite family of candidate seed patterns,
and let $\mcW$ be a finite family of patterns that we want to ``reach'' in the following sense:
given a sequence $(i_1, \ldots, i_n) \in \{1, \ldots, \ell\}^n$ and given $V \in \mcV$,
we want to characterize when $\Sigma_{i_1} \cdots \Sigma_{i_n}(V)$
contains a pattern $W \in \mcW$ for every sufficiently large $n$.
(Note that $W$ might depend on $n$.)
Typically, $\mcV$ will be a set of patterns containing~\input{fig/U.tex}\,,
and $\mcW$ will be the set of all the possible minimal $\mcL$-annuli of the $V \in \mcV$.
The goal is then to prove that iterating substitutions starting from some $V \in \mcV$
eventually generates a pattern containing an annulus $W$ around a pattern in $\mcV$,
in view of ``initializing'' the strategy described in Section~\ref{subsec:mcfgen_strategy}.

This is achieved thanks to the algorithmic construction of generation graphs below,
where we recursively backtrack all the possible preimages of the faces in the patterns of $\mcW$,
and we check that they eventually come back to a pattern of $\mcV$.
This is formalized in Proposition~\ref{prop:gengraph} below,
and used effectively in Sections~\ref{subsec:graph_Brun} and~\ref{subsec:graph_JP}.

For the above method to work in practice,
we will need to filter out some useless faces when backtracking preimages of faces.
We will then use a \tdef{filter set} $\mcF$,
by restricting the allowed preimages to $\mcF$ only.
For example, the set $\FBrun$ that will be used with the Brun substitutions
is the set of all the faces that belong to a discrete plane
$\Gv$ with $0 < \bfv_1 < \bfv_2 < \bfv_3$.

\begin{defi}
\label{defi:gengraph}
Let $\Sigma_1, \ldots, \Sigma_\ell$ be dual substitutions,
let $\mcF$ be a family of faces (the filter set)
and let $\mcX$ be a finite set of faces (the initial set).
The \tdef{generation graph}
associated with $\Sigma_1, \ldots, \Sigma_\ell$, $\mcF$ and $\mcX$
is defined as the graph $\mcG = \bigcup_{n \geq 0} \mcG_n$
where $(\mcG_n)_{n \geq 0}$ is the sequence of directed graphs (whose vertices are faces),
defined by induction as follows.
\begin{enumerate}
\item \emph{Initialization.}
    $\mcG_0$ has no edges and its set of vertices is $\mcX$.
\item \emph{Iteration.}
    Suppose that $\mcG_n$ is constructed for some $n \geq 0$.
    Start with $\mcG_{n+1}$ equal to~$\mcG_n$.
    Then, for each vertex $f$ of $\mcG_n$, for each $i \in \{1, \ldots, \ell\}$,
    and for each $g \in \mcF$ such that $f \in \Sigma_i(g)$,
    add the vertex $g$ and the edge $\smash{g \overset{i}{\rightarrow} f}$ to $\mcG_{n+1}$.
\end{enumerate}
\end{defi}

\begin{rema}
The non-decreasing union $\mcG = \bigcup_{n \geq 1} \mcG_n$
considered in Definition~\ref{defi:gengraph} is not necessarily finite.
It is finite if and only if $\mcG_n = \mcG_{n+1}$ for some $n \geq 1$.
This is the case for Brun substitutions (see Section~\ref{subsec:graph_Brun}).
However, even if $\mcG_n = \mcG_{n+1}$ never occurs,
the infinite graph $\mcG$ can still be successfully exploited,
as will be done for Jacobi-Perron substitutions (see Section~\ref{subsec:graph_JP}).

The orientation of edges in generation graphs agrees with the usual notation of function composition.
In particular, for every path
$\smash{f_n \overset{i_n}{\rightarrow} \cdots \overset{i_2}{\rightarrow} f_1 \overset{i_1}{\rightarrow} f_0}$,
we have $f_0 \in \Sigma_{i_1} \cdots \Sigma_{i_n}(f_n)$.
\end{rema}

\begin{prop}
\label{prop:gengraph}
Let
\begin{itemize}
\item $\Sigma_1, \ldots, \Sigma_{\ell}$ be dual substitutions;
\item $\mcF$ be a set of faces such that if $\Gv \subseteq \mcF$,
    then $\Sigma_i(\Gv) \subseteq \mcF$ for all $i \in \{1, \ldots, \ell\}$;
\item
    $\mcW$ be a set of patterns such that
    for every $\Gv \subseteq \mcF$ we have $W \subseteq \Gv$ for some $W \in \mcW$;
\item
    $\mcG$ be the generation graph constructed
    with substitutions $\Sigma_1, \ldots, \Sigma_\ell$,
    filter set $\mcF$ and initial set $\mcX = \bigcup_{W \in \mcW} W$;
\item $\mcV$ be a set of patterns such that
    \begin{itemize}
    \item for every $\Gv \subseteq \mcF$ we have $V \subseteq \Gv$ for some $V \in \mcV$,
    \item for every $\Gv \subseteq \mcF$ there exists $V \in \mcV$ such that $\Gv \cap \bigcup_{V \in \mcV} V = V$,
    \item for every $V \in \mcV$ there exists $\Gv \subseteq \mcF$ such that $V \subseteq \Gv \subseteq \mcF$.
    \end{itemize}
\item
    $(i_1, \ldots, i_n) \in \{1, \ldots, \ell\}^n$.
\end{itemize}
We have:
\begin{enumerate}
\item
\label{graphprop1}
Suppose that for every path $f_n \overset{i_n}{\rightarrow} \cdots f_1 \overset{i_1}{\rightarrow} f_0$ in $\mcG$
we have $f_n \in \bigcup_{V \in \mcV} V$.
Then for every $V \in \mcV$ there exists $W \in \mcW$
such that $\Sigma_{i_1} \cdots \Sigma_{i_n}(V)$ contains $W$.
\item
\label{graphprop2}
For every path $f_n \overset{i_n}{\rightarrow} \cdots f_1 \overset{i_1}{\rightarrow} f_0$ in $\mcG$
we have $f_0 \notin \Sigma_{i_1} \cdots \Sigma_{i_n}(\mcU)$
if $f_n \notin \mcU$.
\end{enumerate}
\end{prop}

\begin{proof}
We first prove Statement~\ref{graphprop1}.
Let $V \in \mcV$ and let $\Gv \subseteq \mcF$ be a discrete plane containing~$V$.
Let $W \in \mcW$ be such that $W \subseteq \Sigma_{i_1} \cdots \Sigma_{i_n}(\Gv)$
(we use the second and third assumptions of the proposition).
We will prove that $\Sigma_{i_1} \cdots \Sigma_{i_n}(V)$ contains $W$.
Let $f \in W$ and set $f_0 = f$.
By Proposition~\ref{prop:imgplane}, there exists $f_1$ in the discrete plane $\Sigma_{i_1}(\Gv)$
such that $f_0 \in \Sigma_{i_1}(f_1)$.
One has $f_1 \in \mcF$ (we again use the second assumption).
By repeatedly applying Proposition~\ref{prop:imgplane} and by definition of generation graphs,
there exists a path $\smash{f_n \overset{i_n}{\rightarrow} \cdots f_1 \overset{i_1}{\rightarrow} f_0}$ in $\mcG$.
Hence, $f_n \in \bigcup_{V \in \mcV} V$ according to the assumption on $i_1, \ldots, i_n$.
Furthermore $f_n \in \Sigma_{i_1} \cdots \Sigma_{i_n}(\Gv)$.
Hence, by assumption on $\mcV$, we have $f_n \in V$,
so $f \in \Sigma_{i_1} \cdots \Sigma_{i_n}(\mcV)$.
This holds for all $f \in W$ so the statement is proved.

We now prove Statement~\ref{graphprop2}.
Let $\smash{f_n \overset{i_n}{\rightarrow} \cdots f_1 \overset{i_1}{\rightarrow} f_0}$ be a path in $\mcG$
such that $f_n \notin \mcU$.
The patterns $\mcU$ and $f_n$ are disjoint patterns included in a common discrete plane
(because $\mcU$ is included in every discrete plane).
Hence, Proposition~\ref{prop:imgplane} implies that
$\Sigma_{i_1} \cdots \Sigma_{i_n}(\mcU)$
and
$\Sigma_{i_1} \cdots \Sigma_{i_n}(f_n)$
do not have any face in common for all $n \geq 1$,
which implies that $f_0 \notin \Sigma_{i_1} \cdots \Sigma_{i_n}(\mcU)$
since $f_0 \in \Sigma_{i_1} \cdots \Sigma_{i_n}(f_n)$.
\end{proof}

\begin{rema}
The technical assumptions on $\mcV$ and $\mcW$ in the previous proposition cannot be avoided
(we rely on the properties of discrete planes and minimal annuli).
Also, the assumption on $\mcF$ always holds for the set of faces of discrete planes and for
any complete multidimensional continued fraction algorithm,
which includes the additive Brun and Jacobi-Perron algorithms considered here.
Complete means that the continued fraction transformation is onto.

Statement~\ref{graphprop1} of Proposition~\ref{prop:gengraph}
will be used to obtain ``positive'' results,
such as the fact that a given initial pattern $V$ always generates an entire discrete plane
(see Theorem~\ref{theo:seeds}).
The technical assumptions on $\mcF$, $\mcV$ and $\mcW$ are unavoidable
but are natural and easy to check for the substitutions defined by multidimensional continued fraction algorithms
such as the Brun and Jacobi-Perron substitutions in Section~\ref{subsec:seeds}.

Conversely, Statement~\ref{graphprop2}
will be used in Theorems~\ref{theo:bad_Brun} and~\ref{theo:bad_JP}
to characterize the sequences $(i_n)_{n \geq 1}$
for which iterating the corresponding sequence of dual substitutions \emph{fails} to generate an
entire discrete plane when starting from $\mcU$.
\end{rema}

\section{Technical proofs}
\label{sect:tech_proofs}

We have gathered in this section all the proofs which make specific use of Brun and Jacobi-Perron substitutions.
Proofs devoted to minimal annuli and seeds are handled in Section~\ref{subsec:seeds},
to covering properties in Section~\ref{subsec:covprop},
to Property~A in Section~\ref{subsec:propA},
to generation graphs for the Brun substitutions in Section~\ref{subsec:graph_Brun},
to generation graphs for the Jacobi-Perron substitutions in Section~\ref{subsec:graph_JP},
and to generating translates of seeds in Section~\ref{subsec:transseeds}.

The following two sets of patterns will be used throughout this section:
\begin{align*}
\LBrun &=
\big\{
\input{fig/E12b.tex}    \ 
\input{fig/E13b.tex}    \ 
\input{fig/E23a.tex}    \ 
\input{fig/E23b.tex}    \ 
\input{fig/E33a.tex}    \ 
\input{fig/E33b.tex}    \ 
\input{fig/U.tex}       \ 
\input{fig/E22aBrun.tex}
\big\},
\\
\LJP &=
\big\{
\input{fig/E12b.tex}    \ 
\input{fig/E13a.tex}    \ 
\input{fig/E13b.tex}    \ 
\input{fig/E23a.tex}    \ 
\input{fig/E23b.tex}    \ 
\input{fig/E33a.tex}    \ 
\input{fig/E33b.tex}    \ 
\input{fig/U.tex}       \ 
\input{fig/E22aJP.tex}  \ 
\input{fig/E11aJP.tex}
\big\}.
\end{align*}

\begin{rema}
\label{rema:arithplane}
We will often use the arithmetic restrictions in
the definition of discrete planes
in order to simplify the combinatorics of the patterns that occur.
For example, if $\bfv_1 \leq \bfv_3$ and $\bfv_2 \leq \bfv_3$,
then $\Gv$ cannot contain any translate of the two-face pattern
$[\mathbf 0, 1]^\star \cup [(0, 1, 0), 1]^\star = \input{fig/E11a.tex}$
or $[\mathbf 0, 2]^\star \cup [(0, 0, 1), 2]^\star = \input{fig/E22b.tex}$.
If moreover $\bfv_1 \leq \bfv_2$, then the pattern
$[\mathbf 0, 1]^\star \cup [(0, 1, 0), 1]^\star = \input{fig/E11b.tex}$
also never occurs.
\end{rema}

\paragraph{Preimages by dual substitutions}
In some of the following proofs we will need to compute preimages of faces by $\SBrun_i$ or $\SJP_i$.
By abuse of notation, given a dual substitution $\Sigma$,
 we will write $\Sigma^{-1}([\bfx, i]^*)$
for the union of faces $[\bfy, j]^*$ such that $\Sigma([\bfy, j]^*)$ contains the face $[\bfx, i]^*$.
Computing $\Sigma^{-1}([\bfx, i]^*)$
can be done directly from the definition of dual substitutions,
so such computations will be omitted in the following.

\subsection{Minimal annuli and seeds}
\label{subsec:seeds}

The following two propositions list all the possible $\LBrun$- and $\LJP$-annuli in an admissible discrete plane.
Proposition~\ref{prop:minannuli_JP} can be proved in the same way as Proposition~\ref{prop:minannuli_Brun}.

\begin{prop}[Brun minimal annuli]
\label{prop:minannuli_Brun}
Let $P$ be a pattern contained in a discrete plane $\Gv$ with $0 < \bfv_1 < \bfv_2 < \bfv_3$.
If $P$ contains $\mcU$ and an $\LBrun$-annulus of $\mcU$,
then $P$ must contain one of the following two patterns,
each of which contains an $\LBrun$-annulus (shown in light gray)
of $\mcU$ (shown in dark gray):
\begin{align*}
\VBrun_1 &= \input{fig/VBrun1.tex} &
\VBrun_2 &= \input{fig/VBrun2.tex} \ .
\end{align*}
\end{prop}

\begin{proof}
According to Remark~\ref{rema:arithplane}, the algebraic restriction $0 < \bfv_1 < \bfv_2 < \bfv_3$
enables us to enumerate all the surroundings of the pattern $\mcU$ in an admissible discrete plane.
It is then easy to check that such a surrounding
either contains $\VBrun_1$ or $\VBrun_2$,
or contains \smash{\input{fig/E11a.tex}}\,, \smash{\input{fig/E22b.tex}} or \smash{\input{fig/E11b.tex}},
which is forbidden by Remark~\ref{rema:arithplane} because $\bfv_1 < \bfv_2 < \bfv_3$.
\end{proof}

\begin{prop}[Jacobi-Perron minimal annuli]
\label{prop:minannuli_JP}
Let $P$ be a pattern contained in a discrete plane $\Gv$ with $0 < \bfv_1 < \bfv_3$ and $0 < \bfv_2 < \bfv_3$.
If $P$ contains $\mcU$ and an $\LJP$-annulus of $\mcU$,
then $P$ must contain one of the following four patterns,
each of which contains an $\LJP$-annuli (shown in light gray)
of $\mcU$ (shown in dark gray):
\begin{align*}
\VJP_1 &= \input{fig/VJP1.tex} &
\VJP_2 &= \input{fig/VJP2.tex} &
\VJP_3 &= \input{fig/VJP3.tex} &
\VJP_4 &= \input{fig/VJP4.tex} \ .
\end{align*}
\end{prop}

\subsection{Covering properties}
\label{subsec:covprop}

\begin{prop}[Strong $\LBrun$-covering]
\label{prop:cov_Brun}
Let $P$ be an $\LBrun$-covered pattern
such that the patterns
\input{fig/E11b.tex}\,,
\input{fig/E11a.tex} and
\input{fig/E22b.tex}
do not occur in $P$.
Then $\smash{\SBrun_i(P)}$ is strongly $\LBrun$-covered for each $i \in \{1,2,3\}$.
\end{prop}

Observe that the patterns given in the statement of Proposition~\ref{prop:cov_Brun}
cannot occur in a discrete plane with normal vector $\bfv$ for which $0 < \bfv_1 < \bfv_2 $
according to Remark~\ref{rema:arithplane},
which is the natural domain of definition for Brun algorithm.
The same holds for Proposition~\ref{prop:cov_JP} with Jacobi-Perron admissible discrete planes.

\begin{proof}
First, $\SBrun_i(P)$ is $\LBrun$-covered thanks to Proposition~\ref{prop:coverprop},
because $\SBrun_i(Q)$ is $\LBrun$-covered for every $Q \in \LBrun$,
as can be verified from the equalities below:
\begin{align*}
\SBrun_1(\LBrun) &= \Big \{\input{fig/covBrun1.tex} \Big \}\\
\SBrun_2(\LBrun) &= \Big \{\input{fig/covBrun2.tex} \Big \}\\
\SBrun_3(\LBrun) &= \Big \{\input{fig/covBrun3.tex} \Big \}.
\end{align*}
It remains to prove that $\smash{\SBrun_i(P)}$ is \emph{strongly} $\LBrun$-covered.
Let $X \subseteq \smash{\SBrun_i(P)}$ be an edge-connected two-face pattern.
If $X$ is a translate of one of the first $6$ patterns of $\LBrun$,
then the result trivially holds (take $Y = X$ in the definition of strong coverings).

If $X$ is a translate of $f \cup g = \input{fig/E12a.tex}$,
then there must exist $f_0, g_0 \in P$ with $ f_0 \neq g_0$ such that
$f \in \SBrun_i(f_0)$ and $g \in \SBrun_i(g_0)$.
Indeed, we can check that $X$ cannot be contained in the image of a single face,
by definition of the $\SBrun_i$.
Enumerating all the possible such preimages $f_0, g_0$ gives
$f_0 \cup g_0 = \input{fig/E12a.tex}$ (for $\SBrun_1$),
$f_0 \cup g_0 = \input{fig/E13a.tex}$ (for $\SBrun_2$), or
$f_0 \cup g_0 = \input{fig/E23a.tex}$ (for $\SBrun_3$).
For $i \in \{1,2,3\}$
we have $\SBrun_i(f_0 \cup g_0) = \input{fig/U.tex}$
so $X \subseteq \input{fig/U.tex} \subseteq \SBrun_i(P)$,
which enables us to take $Y = \input{fig/U.tex} \in \LBrun$.
The cases $X = \input{fig/E13a.tex}$ or \input{fig/E22a.tex}
can be settled exactly in the same way.
We have treated $9$ cases over $12$ possible cases for $X$,
so the proof is complete because we have ruled out the $3$ remaining patterns in the statement of the proposition.
\end{proof}

\begin{prop}[Strong $\LJP$-covering]
\label{prop:cov_JP}
Let $P$ be an $\LJP$-covered pattern avoiding the patterns
\input{fig/E11a.tex} and \input{fig/E22b.tex}\,.
Then $\SJP_{a,b}(P)$ is strongly $\LJP$-covered for all $0 \leq a \leq b$, with $b \neq 0$.
\end{prop}

\begin{proof}
First we must prove that $\smash{\SJP_{a,b}(P)}$ is $\LJP$-covered.
Thanks to Proposition~\ref{prop:coverprop}
it is enough to prove that $\smash{\SJP_{a,b}(Q)}$ is $\LJP$-covered for every $Q \in \LJP$.
Suppose that $Q = \input{fig/E33a.tex}$.
The pattern $\smash{\SJP_{a,b}(Q)}$ is of the form
\[
\input{fig/JPruleE33a.tex}
\]
(in this example $a = 3$ and $b = 5$).
It is then easy to check that if $a = 0$, $\smash{\SJP_{a,b}(Q)}$ can be $\LJP$-covered using the patterns
\input{fig/E33a.tex}\,, \input{fig/E33b.tex} and \input{fig/E11aJP.tex}\,,
and if $a \neq 0$, $\smash{\SJP_{a,b}(Q)}$ can be $\LJP$-covered using the patterns
\input{fig/E13a.tex}\,, \input{fig/E23a.tex}\,, \input{fig/E12b.tex}\,, \input{fig/E33a.tex} and \input{fig/E33b.tex}\,.
A similar reasoning can be carried out for each of the other $9$ patterns $Q \in \LJP$,
which proves that $P$ is $\LJP$-covered.

It remains to prove the \emph{strong} $\LJP$-covering.
Let $\smash{X \subseteq \SJP_{a,b}(P)}$ be an edge-connected two-face pattern.
If $X$ is a translate of one of the first $7$ patterns of $\LJP$,
then the result is obvious, because $X$ is itself in $\smash{\LJP}$.
If $X$ is a translate of \input{fig/E11b.tex},
then we must have
$X = \input{fig/E11b.tex} \subseteq \input{fig/E11bJP.tex} \subseteq \smash{\SJP_{a,b}(P)}$,
because in an image by $\smash{\SJP_{a,b}}$, a face of type $1$ must come from a face of type $3$,
so there must also be a face of type $3$ at the same position (because $b \neq 0$).
If $X$ is a translate of \input{fig/E12a.tex},
a similar reasoning yields
$X = \input{fig/E12a.tex} \subseteq \input{fig/U.tex} \subseteq \SJP_{a,b}(P)$.
If $X$ is a translate of \input{fig/E22a.tex},
then $X$ is in the image of a face of type $3$,
or has one face in the image of a face of type $2$,
and the other face in the image of a face of type $3$.
In both cases, we must have
$X = \input{fig/E22a.tex} \subseteq \input{fig/E22aJP.tex} \subseteq \SJP_{a,b}(P)$
(since $b \geq a$).

We have handled $10$ cases over $12$ possible cases for $X$,
and the two remaining patterns have been ruled out in the statement of the proposition.
\end{proof}

\subsection{Property~A}
\label{subsec:propA}

\begin{prop}[Property~A for Brun]
\label{prop:propA_Brun}
Property~A holds for Brun substitutions with $\LBrun$,
when restricted to planes $\Gv$ with $ 0 < \bfv_1 < \bfv_2 < \bfv_3$.
\end{prop}

\begin{proof}
There are finitely many two-face connected patterns $f \cup g$,
so we can enumerate all the faces $f,g,f_0,g_0$
that satisfy $f \in \Sigma(f_0)$, $g \in \Sigma(g_0)$,
$f \cup g$ is connected and $f_0 \cup g_0$ is disconnected (see Definition~\ref{defi:propA}),
for $\Sigma=\SBrun_1$, $\SBrun_2$ and $\SBrun_3$.
It turns out that there are $9$ such possibilities,
where the corresponding values for $f_0 \cup g_0$
are shown in the table below.
\[
\begin{array}{lll}
\SBrun_1 & \SBrun_2 & \SBrun_3 \\
\hline
{}[\mathbf 0, 2]^\star \cup [(0,1,0), 1]^\star
& [\mathbf 0, 3]^\star \cup [(1,0,-1), 3]^\star
& [\mathbf 0, 3]^\star \cup [(0,1,-1), 3]^\star
\\
{}[\mathbf 0, 2]^\star \cup [(1,-1,0), 2]^\star
& [\mathbf 0, 3]^\star \cup [(0,1,1), 1]^\star
& [\mathbf 0, 3]^\star \cup [(0,0,1), 2]^\star
\\
{}[\mathbf 0, 2]^\star \cup [(0,1,1), 1]^\star
& [\mathbf 0, 3]^\star \cup [(0,0,1), 1]^\star
& [\mathbf 0, 3]^\star \cup [(1,0,1), 2]^\star
\end{array}
\]
Let us treat the case $f_0 \cup g_0 = [\mathbf 0, 2]^\star \cup [(1,-1,0), 2]^\star$.
Suppose that there exist a pattern $P$ and an $\LBrun$-annulus $A$ of $P$
that is included in a discrete plane such that $f_0 \in P$ and $g_0 \notin A$.
Because $A$ is an annulus of $P$,
every extension of $f_0 \cup g_0$ within a discrete plane must be of the form
\input{fig/fugBruncompl1.tex}
or \input{fig/fugBruncompl2.tex}\,,
where $f_0 \cup g_0$ is shown in light gray and the dark gray faces are in $A$.

Now, an occurrence of \input{fig/E11a_dark.tex}\,
is forbidden, since we are restricted to discrete planes $\Gv$
with $0 < \bfv_1 <\bfv_2 < \bfv_3$ (see Remark~\ref{rema:arithplane}).
An occurrence of \input{fig/E13a_dark.tex} also cannot happen,
because $A$ is strongly $\LBrun$-covered.
Indeed, $\input{fig/E13a_dark.tex}\, \subseteq A$,
so there must exist a translate of a pattern of $\LBrun$
that is included in $A$ and that contains $\input{fig/E13a_dark.tex}$\,.
The only such pattern in $\LBrun$ is $\input{fig/U.tex}$
(note that $\input{fig/E13a.tex}\, \notin \LBrun$).
This is impossible: otherwise $\input{fig/U.tex}$ and $f_0 \cup g_0$ overlap,
which is a contradiction because $f_0, g_0 \notin A$ and $\input{fig/U.tex}\, \subseteq A$.
The same reasoning applies to the eight other cases.
\end{proof}

In order to prove Property~A for Jacobi-Perron substitutions (Proposition~\ref{prop:propA_JP} below),
we first need to prove Lemma~\ref{lemm:3pattJP} below,
which describes all the possible disconnected preimages by $\SJP_{a,b}$ of a two-face connected pattern.
A striking fact is that there are only three possible such preimages,
despite the fact that $a$ and $b$ can take infinitely many values.

\begin{lemm}
\label{lemm:3pattJP}
Let $f$ and $g$ be two faces such that
\begin{enumerate}
\item $f \cup g$ is included in a discrete plane $\Gv$ with $0 < \bfv_1 < \bfv_3$ and $0< \bfv_2 < \bfv_3$;
\label{cond:plane}
\item $f \cup g$ is disconnected;
\label{cond:disco}
\item $\SJP_{a,b}(f \cup g)$ is connected for some $0 \leq a \leq b, b \neq 0$.
\label{cond:conn}
\end{enumerate}
Then $f \cup g$ is a translate of one of the three patterns
$P_1
    = [\mathbf 0, 3]^\star \cup [(1, 0, -1), 3]^\star$,
$P_2
    = [\mathbf 0, 3]^\star \cup [(1, -1, -1), 3]^\star$,
or $P_3
    = [\mathbf 0, 3]^\star \cup [(-1, 1, -1), 3]^\star$.
\end{lemm}

\begin{proof}
We first need the following easy preliminary fact (Equation~(\ref{eq:disco}) below),
which can be checked by inspection of finitely many cases.
A two-face pattern $P = [\bfx, i]^\star \cup [\bfy, j]^\star$ is edge-connected if and only $\bfy - \bfx \in V_{ij}$, where
\begin{equation}
\label{eq:disco}
\begin{split}
V_{11} &= \{\pm (0,1,0), \pm (0,0,1), \pm (0,1,-1), \pm (0,1,1)\}, \\
V_{22} &= \{\pm (1,0,0), \pm (0,0,1), \pm (1,0,-1), \pm (1,0,1)\}, \\
V_{33} &= \{\pm (1,0,0), \pm (0,1,0), \pm (1,-1,0), \pm (1,1,0)\}, \\
V_{12} &= \{(0,0,0), (-1,1,0), \pm (0,0,1), (0,1,-1), (-1,1,-1), (-1,1,1), (-1,0,1)\}, \\
V_{13} &= \{(0,0,0), (-1,0,1), \pm (0,1,0), (-1,1,0), (-1,1,1), (-1,-1,1), (0,-1,1)\}, \\
V_{23} &= \{(0,0,0), (0,-1,1), \pm (1,0,0), (-1,0,1), (-1,-1,1), (1,-1,1), (1,-1,0)\}, \\
V_{21} &= -V_{12}, \quad V_{31} = -V_{13}, \quad V_{32} = -V_{22}.
\end{split}
\end{equation}
In contrast with the Brun substitutions,
we cannot prove Lemma~\ref{lemm:3pattJP} by enumerating all the possibilities
because $a$ and $b$ can take infinitely many values.
However, the problem can be reduced to solving a simple family of linear equations,
which then allows a systematic study of the disconnected preimages of $f \cup g$.

We will show that if Condition~\ref{cond:conn} holds,
then Condition~\ref{cond:plane} or~\ref{cond:disco} fails,
unless $f \cup g$ is a translate of $P_1$, $P_2$ or $P_3$.
We will only treat the case where the types of $f$ and $g$ are both equal to $3$.
This corresponds to the most complicated case;
the remaining cases can be handled analogously.
Recall that $\smash{\bfM_{a,b} = \transp \bfM_{\sJP_{a,b}}}$.
Let $f = [\bfx, 3]^\star$ and $g = [\bfy, 3]^\star$ be two faces contained in a same discrete plane,
and assume that Condition~\ref{cond:conn} holds.
According to Equation (\ref{eq:disco}),
the fact that $\SJP_{a,b}(f \cup g)$ is connected implies one of the following six relations:
\[
\begin{array}{ccccl}
\transp\bfM_{a,b}^{-1}(\bfy - \bfx) &   &   & \in & V_{11}, \\
\transp\bfM_{a,b}^{-1}(\bfy - \bfx) & + & (k, 0, 0) & \in & V_{12} \qquad (0 \leq k \leq a-1), \\
\transp\bfM_{a,b}^{-1}(\bfy - \bfx) & + & (\ell, 0, 0) & \in & V_{13} \qquad (0 \leq \ell \leq b-1), \\
\transp\bfM_{a,b}^{-1}(\bfy - \bfx) & + & (k - k', 0, 0) & \in & V_{22} \qquad (0 \leq k, k' \leq a-1), \\
\transp\bfM_{a,b}^{-1}(\bfy - \bfx) & + & (\ell - k, 0, 0) & \in & V_{23} \qquad (0 \leq k \leq a-1, \ 0 \leq \ell \leq b-1), \\
\transp\bfM_{a,b}^{-1}(\bfy - \bfx) & + & (\ell - \ell', 0, 0) & \in & V_{33} \qquad (0 \leq \ell, \ell' \leq b-1),
\end{array}
\]
for some $0 \leq a \leq b$ with $b \neq 0$.
This is equivalent to
\[
\bfy-\bfx \ \in \
\left\{
\svect{0}{0}{s},
\pm\svect{1}{0}{s},
\svect{-1}{1}{s},
\pm\svect{0}{1}{t},
\pm\svect{1}{1}{t},
\svect{1}{-1}{-t}
:
s \in \bbZ, t \geq 0
\right\}.
\]
This either contradicts Condition~\ref{cond:plane} or~\ref{cond:disco},
or else implies that $f \cup g$ is equal to $P_1$, $P_2$ or $P_3$.
Indeed, let $\Gv$ be a discrete plane that contains both $f$ and $g$.
Thanks to Remark~\ref{rema:arithplane},
we can restrict the possible values of $\bfx - \bfy$ even further.
The remaining valid solutions are shown in the following table,
where $s \in \bbZ$, $t \geq 0$, $0 < \bfv_1 < \bfv_3$ and $0 < \bfv_2 < \bfv_3$.
We then observe that either $\bfy - \bfx \in V_{33}$
(so Condition~\ref{cond:disco} is violated because $f \cup g$ is connected thanks to Equation (\ref{eq:disco})),
or that $f \cup g$ is equal to $P_1$, $P_2$ or $P_3$.
\[
\begin{array}{ll}
\text{Inequalities} & \text{Solutions} \\
\hline
-\bfv_3 < s\bfv_3 < \bfv_3 &
s=0\!: \bfy - \bfx \in V_{33} \\
-\bfv_3 < \bfv_1 + s\bfv_3 < \bfv_3 &
s=0\!: \bfy - \bfx \in V_{33}; \ s=1\!: P_1 \\
-\bfv_3 < -\bfv_1 + \bfv_2 + s\bfv_3 < \bfv_3 &
s=0\!: \bfy - \bfx \in V_{33}; \ s =-1\!: P_3; \ s=1\!: P_2 \\
-\bfv_3 < \bfv_2 + t\bfv_3 < \bfv_3 &
t=0\!: \bfy - \bfx \in V_{33} \\
-\bfv_3 < \bfv_1 + \bfv_2 + t\bfv_3 < \bfv_3 &
t=0\!: \bfy - \bfx \in V_{33} \\
-\bfv_3 < \bfv_1 - \bfv_2 - t\bfv_3 < \bfv_3 &
t=0\!: \bfy - \bfx \in V_{33}; \ t=1\!: P_3
\end{array}
\]
\vspace{-9mm} \\ 
\end{proof}

\begin{prop}[Property~A for Jacobi-Perron]
\label{prop:propA_JP}
Property~A holds for Jacobi-Perron substitutions with $\LJP$,
when restricted to discrete planes $\Gv$ with $0 < \bfv_1 < \bfv_3$ and $0 < \bfv_2 < \bfv_3$.
\end{prop}

\begin{proof}
Thanks to Lemma~\ref{lemm:3pattJP},
if $f,g,f_0,g_0$ satisfy the conditions required in Definition~\ref{defi:propA} for some Jacobi-Perron substitution,
then $f_0 \cup g_0$ must be equal to a translate of one of the three patterns
$P_1$, $P_2$, $P_3$ of the statement of Lemma~\ref{lemm:3pattJP}.

Suppose that there exists a pattern $P$ and an $\LJP$-annulus $A$ of $P$
that is included in a discrete plane such that $f_0 \in P$ and $g_0 \notin A \cup P$.
Because $A$ is an annulus of $P$,
every extension of $f_0 \cup g_0$ within a discrete plane must be of the form
\input{fig/fug1compl.tex}\,,
\input{fig/fug1compl2.tex}\,,
\input{fig/fug2compl.tex}\,,
or \input{fig/fug3compl.tex}\,,
where $f_0 \cup g_0$ is shown in light gray and the dark gray faces are in $A$.
Now, similarly as in the proof of Proposition~\ref{prop:propA_Brun},
a contradiction must occur in each case,
thanks to the strong $\LJP$-covering of $A$
and the precise choice of patterns in $\LJP$.
\end{proof}

\subsection{Generation graphs for the Brun substitutions}
\label{subsec:graph_Brun}

We now construct generation graphs for the substitutions $\SBrun_1, \SBrun_2, \SBrun_3$.
To construct the graphs below we use the filter $\FBrun$ equal to the set of all the faces $f$
that belong to a discrete plane $\Gv$ with $0 < \bfv_1 < \bfv_2 < \bfv_3$.

\paragraph{Graph $\GBrun$ for the seed $\mcU$}

Let $\GBrun$ be the generation graph (see Definition~\ref{defi:gengraph})
associated with $\SBrun_1, \SBrun_2, \SBrun_3$,
the filter $\FBrun$
and the initial set of faces $\mcX = \VBrun_1 \cup \VBrun_2$
(the union of the minimal annuli given in Proposition~\ref{prop:minannuli_Brun}).
Its computation stops after two iterations of the algorithm.
It has $19$ vertices and $47$ edges.

Now we also define $\GBrunmini$ as the subgraph of $\GBrun$
containing only the vertices which are not contained in $\mcU$
and which are at the end of an infinite backward path containing infinitely many $3$s.
The graph $\GBrunmini$ is plotted in Figure~\ref{fig:GBrunmini}.
It will be used in the proofs of Lemma~\ref{lemm:balls_Brun} and Theorem~\ref{theo:bad_Brun}
to characterize the sequences $(i_n)_{n \geq 1}$ for which
the patterns $\SBrun_{i_1} \ldots\SBrun_{i_n}(\mcU)$ generate an entire discrete plane.

\begin{figure}[!ht]
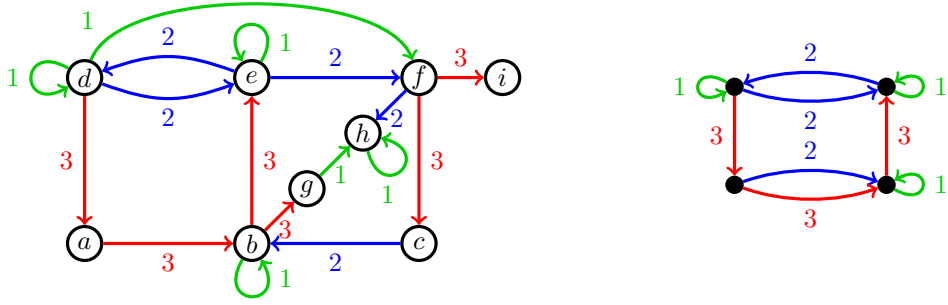

\centering
\myvcenter{\input{fig/graph_bad_Brun.tex}}
\hfil
\myvcenter{\input{fig/graph_bad_Brun_small.tex}}
\caption{
The graph~$\GBrunmini$ (left).
The faces corresponding to the vertices of the graph are
$a = [(1, 1, -1), 1]^\star$,
$b = [(1, -1, 1), 3]^\star$,
$c = [(1, 1, -1), 2]^\star$,
$d = [(-1, 1, 0), 2]^\star$,
$e = [(-1, 0, 1), 3]^\star$,
$f = [(-1, 1, 0), 3]^\star$,
$g = [(-1, 0, 1), 2]^\star$,
$h = [(-1, -1, 1), 3]^\star$,
$i = [(1, 1, -1), 3]^\star$.
The graph on the right is obtained from $\GBrunmini$ by merging some vertices.
It is a reduced version of $\GBrunmini$ in the sense
that the language of the edge labellings of its paths is equal to that of $\GBrunmini$.
}
\label{fig:GBrunmini}
\end{figure}

\paragraph{Graph $\HBrun$ for the seeds $\VBrun_i$}

We now construct another generation graph for the Brun substitutions,
with the seeds $\VBrun_1$ and $\VBrun_2$ instead of $\mcU$.
This time, \emph{every} Brun-admissible sequence
will be proved to generate an $\LBrun$-annulus around $\VBrun_1$ or $\VBrun_2$.
Before computing $\HBrun$ we first enumerate all the possible minimal
$\LBrun$-annuli of $\VBrun_1$ or $\VBrun_2$
(minimal in the sense that any $\LBrun$-annulus of $\VBrun_i$ must contain one of these patterns).
This can be done similarly as in the proof of Proposition~\ref{prop:minannuli_Brun},
and yields the following four possible minimal annuli.
\input{fig/big_minimal_annuli.tex}

Let $\HBrun$ be the generation graph (according to Definition~\ref{defi:gengraph})
constructed with substitutions $\SBrun_1, \SBrun_2, \SBrun_3$,
filter $\FBrun$
and initial set $\mcX = W_1 \cup W_2 \cup W_3 \cup W_4$ (a total of $60$ faces).
Its computation stops after six iterations of the algorithm.
It has $101$ vertices and $240$ edges.
We do not plot the graph below, but it can be easily recovered
from the Sage source code provided with this article.
The aim of computing $\HBrun$ is to prove the following lemma,
which will be used in the proof of Theorem~\ref{theo:seeds}.

\begin{lemm}
\label{lemm:finitegraph_Brun_H}
Let $V \in \{\VBrun_1, \VBrun_2\}$.
There exists $N$ such that for every $(i_1, \ldots, i_n) \in \{1,2,3\}^n$
containing more than $N$ occurrences of $3$,
the pattern $\SBrun_{i_1} \cdots \SBrun_{i_n}(V)$
contains an $\LBrun$-annulus of $\VBrun_1$ or $\VBrun_2$.
\end{lemm}

\begin{proof}
Let $N \geq 1$ be such that for every path
$\smash{f_n \overset{i_n}{\rightarrow} f_{n-1} \overset{i_{n-1}}{\rightarrow} \cdots f_1 \overset{i_1}{\rightarrow} f_0}$
in $\HBrun$ with $i_n = 3$ more than $N$ times,
we have $f_n \in \VBrun_1 \cup \VBrun_2$.
The existence of such an $N$ can be proved computationally in the following way:
one checks that every strongly connected component containing an edge labelled by $3$
consists only of vertices in $\VBrun_1$ or $\VBrun_2$.
Any path containing sufficiently many $3$s must originate from such a component,
and hence from a face with vertex in $\VBrun_1$ or $\VBrun_2$.

The proposition now follows directly from Statement~\ref{graphprop1} of Proposition~\ref{prop:gengraph},
applied to a sequence $(i_1, \ldots, i_n)$ containing at least $N$ occurrences of $3$.
The hypotheses on $\FBrun$, $\mcV = \{\VBrun_1, \VBrun_2\}$, $\mcW = \{W_1, W_2, W_3, W_4\}$
needed in Proposition~\ref{prop:gengraph} can be checked easily.
In particular, the fact that $\Gv \subseteq \FBrun$ implies $\SBrun_i(\Gv) \subseteq \FBrun$ for all $i \in \{1,2,3\}$
follows directly from the definition of the Brun algorithm and from Proposition~\ref{prop:imgplane}.
\end{proof}

\subsection{Generation graphs for the Jacobi-Perron substitutions}
\label{subsec:graph_JP}

In this section we take the filter $\FJP$ to be the set of faces $f$
that belong to a discrete plane $\Gv$ such that and $0 < \bfv_1 < \bfv_3$ and $0 < \bfv_2 < \bfv_3$.

\paragraph{The graph $\GJP$ for the additive Jacobi-Perron substitutions}
Let $\GJP$ be the generation graph (see Definition~\ref{defi:gengraph}) associated with
the additive Jacobi-Perron substitutions $\Theta_1, \Theta_2, \Theta_3, \Theta_4$,
and the starting set $\mcX = \VJP_1 \cup \VJP_2 \cup \VJP_3 \cup \VJP_4$
(defined in Proposition~\ref{prop:minannuli_JP}).
Unlike the graph $\GBrun$, the graph $\GJP$ is infinite
(we never have $\GJP_n = \GJP_{n+1}$).
However, the structure of $\GJP$ is simple enough, so we will be able to describe it precisely below.

Let $\GJP_3$ be the graph obtained at the third step in the computation of $\GJP$, defined in Definition~\ref{defi:gengraph}.
This graph has $33$ vertices and $93$ edges.
In order to use $\GJP_3$ to study $\GJP$ we need the following lemma,
which can be proved easily by studying the preimages of
the additive Jacobi-Perron substitutions $\Theta_1, \Theta_2, \Theta_3$ and $\Theta_4$.

\begin{lemm}
\label{lemm:graphJP}
Define
$e_n = [(-n, 1, 0), 3]^\star$,
$f_n = [(-n, 1, 0), 1]^\star$,
$g_n = [(-n, -1, 1), 3]^\star$,
$h_n = [(-n, -1, 1), 1]^\star$
for $n \geq 0$.
For every integer $n \geq 1$, the only preimages of $e_n$, $f_n$, $g_n$, $h_n$
by one of the additive Jacobi-Perron substitutions $\Theta_1, \Theta_2, \Theta_3$ or $\Theta_4$
are given by:
\[
\begin{array}{llll}
e_n \in \Theta_1(e_{n+1}) & f_n \in \Theta_1(f_n)     & g_n \in \Theta_1(g_n)     & h_n \in \Theta_1(h_{n-1}) \\
e_n \in \Theta_1(f_n)     & f_n \in \Theta_2(h_{n-1}) & g_n \in \Theta_1(h_{n-1}) & h_n \in \Theta_2(h_n).    \\
e_n \in \Theta_2(e_n)     &                           & g_n \in \Theta_2(h_n)     &                           \\
e_n \in \Theta_2(f_{n-1}) &                           & g_n \in \Theta_2(h_{n+1}) &
\end{array}
\]
\end{lemm}

We deduce from Lemma~\ref{lemm:graphJP} that $\GJP$ has two infinite branches,
the first one consisting of vertices $e_n$, $f_n$, and the second one consisting of vertices $g_n$, $h_n$,
with the only edge labels occurring in these infinite branches being $1$ and $2$.
Indeed, the only faces of $\GJP_3$ with preimages allowed by the filter set $\FJP$
which are not vertices of the graph $\GJP_3$ are $e_4$, $f_3$, $g_4$ and $h_3$.

Since we are interested only in the additive Jacobi-Perron-admissible sequences $(i_n)_{n \geq 1}$
(that is, containing infinitely many edges labelled by $3$ or $4$ by Remark~\ref{rema:JPAdditive}),
we remove from $\GJP$ all the vertices which are not at the beginning
of a backward infinite path labelled by such a sequence $(i_n)_{n \geq 1}$.
This yields the following subgraph of $\GJP$ (which now contains only one infinite branch):

\begin{center}
\begin{tikzpicture}[x={1.4cm}, y={1.4cm}]
\draw (0,0) node[minimum size=6mm, inner sep=0pt, draw,circle,thick, very thick] (b) {$f_b$};
\draw (1,0) node[fill=lightgray, minimum size=6mm, inner sep=0pt, draw,circle,thick, very thick] (d) {$f_d$};
\draw (2,0) node[minimum size=6mm, inner sep=0pt, draw,circle,thick, very thick] (g0) {$g_0$};
\draw (3,0) node[minimum size=6mm, inner sep=0pt, draw,circle,thick, very thick] (g1) {$g_1$};
\draw (4,0) node[minimum size=6mm, inner sep=0pt, draw,circle,thick, very thick] (g2) {$g_2$};
\draw (5,0) node (g3) {};
\draw (0,1) node[minimum size=6mm, inner sep=0pt, draw,circle,thick, very thick] (a) {$f_a$};
\draw (1,1) node[fill=lightgray, minimum size=6mm, inner sep=0pt, draw,circle,thick, very thick] (c) {$f_c$};
\draw (2,1) node[fill=lightgray, minimum size=6mm, inner sep=0pt, draw,circle,thick, very thick] (e) {$f_e$};
\draw (3,1) node[minimum size=6mm, inner sep=0pt, draw,circle,thick, very thick] (h0) {$h_0$};
\draw (4,1) node[minimum size=6mm, inner sep=0pt, draw,circle,thick, very thick] (h1) {$h_1$};
\draw (5,1) node[minimum size=6mm, inner sep=0pt, draw,circle,thick, very thick] (h2) {$h_2$};
\draw (6,1) node (h3) {};
\draw[->, very thick, blue] (a) .. controls +(60:10mm) and +(120:10mm) .. node [pos=0.15,right] {\small$2$} (a);
\draw[->, very thick, blue] (h0) .. controls +(60:10mm) and +(120:10mm) .. node [pos=0.15,right] {\small$2$} (h0);
\draw[->, very thick, blue] (h1) .. controls +(60:10mm) and +(120:10mm) .. node [pos=0.15,right] {\small$2$} (h1);
\draw[->, very thick, blue] (h2) .. controls +(60:10mm) and +(120:10mm) .. node [pos=0.15,right] {\small$2$} (h2);
\draw[->, very thick, green!80!black] (e) .. controls +( 60:10mm) and +(120:10mm) .. node [pos=0.15,right] {\small$1$} (e);
\draw[->, very thick, green!80!black] (d) .. controls +(240:10mm) and +(300:10mm) .. node [pos=0.85,right] {\small$1$} (d);
\draw[->, very thick, green!80!black] (g0) .. controls +(240:10mm) and +(300:10mm) .. node [pos=0.85,right] {\small$1$} (g0);
\draw[->, very thick, green!80!black] (g1) .. controls +(240:10mm) and +(300:10mm) .. node [pos=0.85,right] {\small$1$} (g1);
\draw[->, very thick, green!80!black] (g2) .. controls +(240:10mm) and +(300:10mm) .. node [pos=0.85,right] {\small$1$} (g2);
\path[->, very thick, blue] (a) edge node [left] {\small$2$} (b);
\path[->, very thick, blue] (a) edge node [above] {\small$2$} (c);
\path[->, very thick, orange] (c) edge node [above] {\small$3$} (e);
\path[->, very thick, red] (d) edge node [left] {\small$4$} (a);
\path[->, very thick, red] (d) edge node [below] {\small$4$} (b);
\path[->, very thick, red] (d) edge node [left] {\small$4$} (c);
\path[->, very thick, orange] (e) edge node [right] {\small$3$} (g0);
\path[->, very thick, orange] (e) edge node [above] {\small$3$} (h0);
\path[->, very thick, blue] (g0) edge node [below] {\small$2$} (d);
\path[->, very thick, blue]           (h0) edge node [right] {\small$2$} (g0);
\path[->, very thick, green!80!black] (h0) edge node [right] {\small$1$}  (g1);
\path[->, very thick, green!80!black] (h0) edge node [above] {\small$1$} (h1);
\path[->, very thick, blue]           (g1) edge node [below] {\small$2$} (g0);
\path[->, very thick, blue]           (h1) edge node [right] {\small$2$} (g1);
\path[->, very thick, green!80!black] (h1) edge node [right] {\small$1$}  (g2);
\path[->, very thick, green!80!black] (h1) edge node [above] {\small$1$} (h2);
\path[->, very thick, blue]           (g2) edge node [below] {\small$2$} (g1);
\path[->, very thick, blue, dashed]           (h2) edge node [right] {\small$2$} (g2);
\path[->, very thick, green!80!black, dashed] (h2) edge node [right] {\small$1$} (g3);
\path[->, very thick, green!80!black, dashed] (h2) edge node [above] {\small$1$} (h3);
\path[->, very thick, blue, dashed]           (g3) edge node [below] {\small$2$} (g2);
\end{tikzpicture}
\end{center}
where the vertices are given by
$f_a = [(1, 1, -1), 1]^\star$,
$f_b = [(1, 1, -1), 3]^\star$,
$f_c = [(1, 1, -1), 2]^\star$,
$f_d = [(1, -1, 1), 3]^\star$,
$f_e = [(-1, 1, 1), 3]^\star$,
and the edges are the ones shown above, plus, for every $n \geq 0$,
the edges
$\smash{g_n \overset{1}{\rightarrow} g_n}$,
$\smash{h_n \overset{2}{\rightarrow} h_n}$,
$\smash{h_n \overset{2}{\rightarrow} g_n}$,
$\smash{h_n \overset{1}{\rightarrow} g_{n+1}}$,
$\smash{h_n \overset{1}{\rightarrow} h_{n+1}}$
and $\smash{g_{n+1} \overset{2}{\rightarrow} g_n}$.
Moreover we have deleted the edge $\smash{f_e \overset{4}{\rightarrow} f_d}$ in the above graph,
because every infinite path containing it is incompatible
with the condition $a_n = b_n \Rightarrow a_{n+1} \neq 0$
on Jacobi-Perron expansions $(a_n, b_n)_{n \geq 1}$ (see Section~\ref{subsec:JP}).
Indeed, if the edge $\smash{f_e \overset{4}{\rightarrow} f_d}$ were allowed,
then the forbidden product $\Theta_4 \Theta_1^k\Theta3 = \SJP_{1,1}\SJP_{0,k+1}$
would be allowed for some $k \geq 0$, a contradiction.

\paragraph{The graph $\HJP$}
We now construct a generation graph $\HJP$
for the seeds $\VJP_1$, $\VJP_2$, $\VJP_3$ and $\VJP_4$.
Before computing $\HJP$ we must enumerate all the possible minimal $\LJP$-annuli of the $\VJP_i$.
This can be done similarly as for the proof of Proposition~\ref{prop:minannuli_Brun},
and yields eight possible minimal annuli $W_1, \ldots, W_8$,
which are similar to the annuli $W_1, \ldots, W_4$ computed for $\HBrun$.

Let $\HJP$ be the generation graph (according to Definition~\ref{defi:gengraph})
associated with the additive Jacobi-Perron substitutions $\Theta_1, \Theta_2, \Theta_3, \Theta_4$,
the initial set $\mcX = W_1 \cup \cdots \cup W_8$
and the filter $\FJP$.
Thanks to the structure of $\HJP$ we have the following lemma,
which can be proved in exactly the same way as
the corresponding Lemma~\ref{lemm:finitegraph_Brun_H} for the Brun substitutions.

\begin{lemm}
\label{lemm:finitegraph_JP_H}
Let $V \in \{\VJP_1, \VJP_2, \VJP_3, \VJP_4\}$.
There exists $N$ such that
for every $(i_1, \ldots, i_n) \in \{1,2,3,4\}^n$
containing more than $N$ occurrences of $3$ or $4$,
the pattern $\Theta_{i_1} \cdots \Theta_{i_n}(V)$
contains an $\LJP$-annulus of
$\VJP_1$, $\VJP_2$, $\VJP_3$ or $\VJP_4$.
\end{lemm}

\subsection{Generating translates of seeds}
\label{subsec:transseeds}

\begin{lemm}
\label{lemm:balls_Brun}
There exists $N$ such that if a finite sequence $(i_1, \ldots, i_n) \in \{1,2,3\}^n$
contains more than $N$ occurrences of $3$,
then $\SBrun_{i_1} \cdots \SBrun_{i_n}(\mcU)$ contains a translate of $\VBrun_1$ or $\VBrun_2$.
\end{lemm}

\begin{proof}
We can assume that the sequences $(i_1, \ldots, i_n)$
we consider are the edge labellings of a backward path
$\smash{\bullet \overset{i_n}{\rightarrow} \cdots \bullet \overset{i_1}{\rightarrow} \bullet}$
in the four-vertex graph given in Section~\ref{subsec:graph_Brun}, Figure~\ref{fig:GBrunmini}.
Indeed, the language of the labellings of the four-vertex graph is
the same as the one of the graph $\GBrunmini$ (given in the same figure),
and every path containing sufficiently many $3$s which is not contained
in $\GBrunmini$ must originate from $U$ (by construction of $\GBrunmini$).

Using this restriction, we can assume that $(i_n, \ldots, i_1)$
necessarily starts with a word in one of the following languages, for $n$ large enough:
$21^\star321^\star31^\star2,$ 
$21^\star331^\star31^\star21^\star2,$ 
$21^\star331^\star31^\star21^\star33,$ 
$21^\star331^\star31^\star21^\star321^\star31^\star2(1 \cup 2 \cup 3).$ 
This enumeration can be carried out easily by enumerating the possible path labellings in the four-vertex graph
(we must turn counter-clockwise when sufficiently many $3$s are read).
Also note that each of these languages starts with $21^\star32$ or $21^\star33$, which corresponds to starting from the
top-left vertex in the four-vertex graph.
This is not a restriction: indeed, if $(i_n, \ldots, i_1)$ does not start in this way,
then we can consider longer sequences $(i_{n+k}, \ldots, i_1)$ which agree with the above languages.

It now remains to prove that if $(i_n, \ldots, i_1)$ is in one of the following languages
then $\SBrun_{i_1} \cdots \SBrun_{i_n}(\mcU)$ contains a translate of $\VBrun_1$ or $\VBrun_2$.
To do so we track all the possible sequences of iterations starting from $\mcU$,
and we explicitly check that a translate of one of the seeds $V_1$ or $V_2$ is eventually generated.
This is done thanks to the graph shown in Figure~\ref{fig:ball_Brun}.
The vertices are (translation classes of) patterns,
and for every edge $\smash{P \overset{i}{\longrightarrow} Q}$ in the graph,
$\SBrun_i(P)$ contains a translated copy of $Q$.
The pattern $P_0 = [(0, -1, 1), 3]^\star, [(0, 0, 0), 2]^\star, [(1, 0, 0), 2]^\star, [(1, 0, 0), 3]^\star$
(the leftmost vertex of the graph)
can be easily proved to belong to the image of $U$
after iterating any sequence of substitutions containing $\SBrun_3$ at least $3$ times.

Note that unlike the generation graphs constructed in Sections~\ref{subsec:graph_Brun} and~\ref{subsec:graph_JP},
the graph of Figure~\ref{fig:ball_Brun} used in the present proof
has not been constructed algorithmically.
Constructing such a graph is tedious,
but checking its validity is easy (by using computer algebra).
\end{proof}

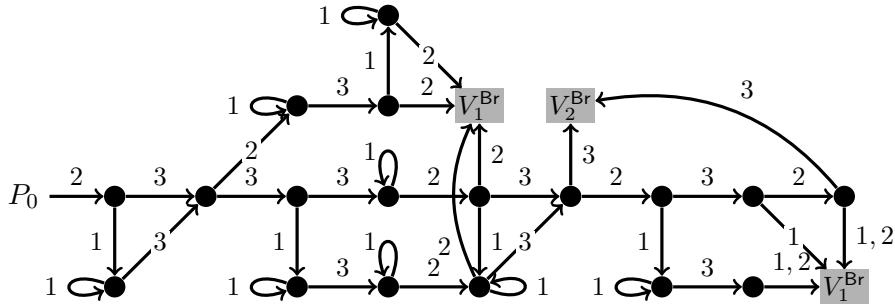
\begin{figure}[!ht]
\centering
\begin{tikzpicture}[x={(0,12mm)},y={(12mm,0)}]
\node at (0,0) (0) {$P_0$};
\node[circle,fill=black,minimum size=8pt,inner sep=0pt] at (0,1) (2) {};
\node[circle,fill=black,minimum size=8pt,inner sep=0pt] at (-1,1) (12) {};
\node[circle,fill=black,minimum size=8pt,inner sep=0pt] at (0,2) (32) {};
\node[circle,fill=black,minimum size=8pt,inner sep=0pt] at (-1,3) (1332) {};
\node[circle,fill=black,minimum size=8pt,inner sep=0pt] at (0,3) (332) {};
\node[circle,fill=black,minimum size=8pt,inner sep=0pt] at (1,3) (232) {};
\node[circle,fill=black,minimum size=8pt,inner sep=0pt] at (-1,4) (31332) {};
\node[circle,fill=black,minimum size=8pt,inner sep=0pt] at (0,4) (3332) {};
\node[circle,fill=black,minimum size=8pt,inner sep=0pt] at (1,4) (3232) {};
\node[circle,fill=black,minimum size=8pt,inner sep=0pt] at (2,4) (13232) {};
\node[circle,fill=black,minimum size=8pt,inner sep=0pt] at (-1,5) (123332) {};
\node[circle,fill=black,minimum size=8pt,inner sep=0pt] at (0,5) (23332) {};
\node[fill=black!30,inner sep=1pt] at (1,5) (A1) {\small$\VBrun_1$};
\node[circle,fill=black,minimum size=8pt,inner sep=0pt] at (0,6) (323332) {};
\node[fill=black!30,inner sep=1pt] at (1,6) (A2) {\small$\VBrun_2$};
\node[circle,fill=black,minimum size=8pt,inner sep=0pt] at (0,7) (2323332) {};
\node[circle,fill=black,minimum size=8pt,inner sep=0pt] at (-1,7) (12323332) {};
\node[circle,fill=black,minimum size=8pt,inner sep=0pt] at (0,8) (32323332) {};
\node[circle,fill=black,minimum size=8pt,inner sep=0pt] at (-1,8) (312323332) {};
\node[circle,fill=black,minimum size=8pt,inner sep=0pt] at (0,9) (232323332) {};
\node[fill=black!30,inner sep=1pt] at (-1,9) (A1b) {\small$\VBrun_1$};
\path[->, very thick] (0) edge node [above] {\small$2$} (2);
\path[->, very thick] (2) edge node [left] {\small$1$} (12);
\path[->, very thick] (2) edge node [above] {\small$3$} (32);
\path[->, very thick] (12) edge node [fill=white,inner sep=1pt] {\small$3$} (32);
\path[->, very thick] (32) edge node [above] {\small$3$} (332);
\path[->, very thick] (32) edge node [fill=white,inner sep=1pt] {\small$2$} (232);
\path[->, very thick] (332) edge node [left] {\small$1$} (1332);
\path[->, very thick] (332) edge node [above] {\small$3$} (3332);
\path[->, very thick] (1332) edge node [above] {\small$3$} (31332);
\path[->, very thick] (232) edge node [above] {\small$3$} (3232);
\path[->, very thick] (3332) edge node [above] {\small$2$} (23332);
\path[->, very thick] (31332) edge node [above] {\small$2$} (123332);
\path[->, very thick] (3232) edge node [left] {\small$1$} (13232);
\path[->, very thick] (3232) edge node [above] {\small$2$} (A1);
\path[->, very thick] (13232) edge node [fill=white,inner sep=1pt] {\small$2$} (A1);
\path[->, very thick] (23332) edge node [right] {\small$1$} (123332);
\path[->, very thick] (23332) edge node [right] {\small$2$} (A1);
\path[->, very thick] (23332) edge node [above] {\small$3$} (323332);
\path[->, very thick] (123332) edge node [fill=white,inner sep=1pt] {\small$3$} (323332);
\draw[->, very thick] (123332) .. controls +(115:10mm) and +(245:10mm) .. node [pos=0.18,left] {\small$2$} (A1);
\path[->, very thick] (323332) edge node [right] {\small$3$} (A2);
\path[->, very thick] (323332) edge node [above] {\small$2$} (2323332);
\path[->, very thick] (2323332) edge node [left] {\small$1$} (12323332);
\path[->, very thick] (2323332) edge node [above] {\small$3$} (32323332);
\path[->, very thick] (12323332) edge node [above] {\small$3$} (312323332);
\path[->, very thick] (32323332) edge node [above] {\small$2$} (232323332);
\path[->, very thick] (32323332) edge node [fill=white,inner sep=1pt] {\small$1$} (A1b);
\path[->, very thick] (312323332) edge node [above] {\small$1,2$} (A1b);
\path[->, very thick] (232323332) edge node [right] {\small$1,2$} (A1b);
\path[->, very thick] (232323332) edge [bend right] node [above right] {\small$3$} (A2);
%
\draw[->, very thick] (12) .. controls +(160:8mm) and +(200:8mm) .. node [left] {\small$1$} (12);
\draw[->, very thick] (1332) .. controls +(160:8mm) and +(200:8mm) .. node [left] {\small$1$} (1332);
\draw[->, very thick] (232) .. controls +(160:8mm) and +(200:8mm) .. node [left] {\small$1$} (232);
\draw[->, very thick] (13232) .. controls +(160:8mm) and +(200:8mm) .. node [left] {\small$1$} (13232);
\draw[->, very thick] (3332) .. controls +(70:8mm) and +(110:8mm) .. node [left] {\small$1$} (3332);
\draw[->, very thick] (31332) .. controls +(70:8mm) and +(110:8mm) .. node [left] {\small$1$} (31332);
\draw[->, very thick] (123332) .. controls +(-20:8mm) and +(20:8mm) .. node [right] {\small$1$} (123332);
\draw[->, very thick] (12323332) .. controls +(160:8mm) and +(200:8mm) .. node [left] {\small$1$} (12323332);
\end{tikzpicture}
\caption{The graph used in the proof of Lemma~\ref{lemm:balls_Brun}.
The vertices are patterns (up to translation).
There is an edge from $P$ to $Q$ labelled by $i$
if $\SBrun_i(P)$ contains a \emph{translated copy}~of~$Q$.}
\label{fig:ball_Brun}
\end{figure}

We now prove an analogue of the previous lemma for Jacobi-Perron substitutions.
The proof is easier (we need a smaller graph); indeed the set of ``bad'' Jacobi-Perron expansions
is more constrained (see Theorem~\ref{theo:bad_JP}).

\begin{lemm}
\label{lemm:balls_JP}
Let $(a_n,b_n)_{n \geq 1}$ be an admissible Jacobi-Perron expansion.
Then there exists $N \geq 0$ such that
$\SJP_{a_1,b_1} \cdots \SJP_{a_N,b_N}(\mcU)$ contains a translate of
$\VJP_1$, $\VJP_2$, $\VJP_3$ or $\VJP_4$.
\end{lemm}

\begin{proof}
Similarly as in the proof of Lemma~\ref{lemm:balls_Brun},
we can apply Theorem~\ref{theo:bad_JP},
so that the result only needs to be proved for products of the form
$\Theta_{i_1} \cdots \Theta_{i_n}$
with $i_1 \cdots i_n \in 31^\star31^\star22^\star4$,
which is a translation of the condition given in Theorem~\ref{theo:bad_JP}
in terms of additive Jacobi-Perron substitutions.
This is settled thanks to the following graph,
in which every edge $\smash{P \overset{i}{\longrightarrow} Q}$
means that $\Theta_i(P)$ contains a translated copy of $Q$.
\begin{center}
\begin{tikzpicture}[x=1cm,y=0.85cm]
\node at (0,0) (0) {%
    \myvcenter{\begin{tikzpicture}
    [x={(-0.173205cm,-0.100000cm)}, y={(0.173205cm,-0.100000cm)}, z={(0.000000cm,0.200000cm)}]
    \definecolor{facecolor}{rgb}{0.800000,0.800000,0.800000}
    \fill[fill=facecolor, draw=black, shift={(0,0,0)}]
    (0, 0, 0) -- (0, 1, 0) -- (0, 1, 1) -- (0, 0, 1) -- cycle;
    \fill[fill=facecolor, draw=black, shift={(0,0,0)}]
    (0, 0, 0) -- (0, 0, 1) -- (1, 0, 1) -- (1, 0, 0) -- cycle;
    \fill[fill=facecolor, draw=black, shift={(0,0,0)}]
    (0, 0, 0) -- (1, 0, 0) -- (1, 1, 0) -- (0, 1, 0) -- cycle;
    \end{tikzpicture}}%
};
\node at (2,0) (1) {\input{fig/WJP3.tex}};
\node at (4,0) (2) {\input{fig/WJP33.tex}};
\node at (6,0) (3) {\input{fig/WJP233.tex}};
\node at (8,0) (4) {\input{fig/WJP4233.tex}};

\path[->, very thick] (0) edge node [above] {\small$3$} (1);
\path[->, very thick] (1) edge node [above] {\small$3$} (2);
\path[->, very thick] (2) edge node [above] {\small$2$} (3);
\path[->, very thick] (3) edge node [above] {\small$4$} (4);

\draw[->, very thick] (1) .. controls +(110:9mm) and +(70:9mm) .. node [pos=0.20,left] {\small$1$} (1);
\draw[->, very thick] (2) .. controls +(110:9mm) and +(70:9mm) .. node [pos=0.20,left] {\small$1$} (2);
\draw[->, very thick] (3) .. controls +(110:9mm) and +(70:9mm) .. node [pos=0.20,left] {\small$2$} (3);
\end{tikzpicture}
\end{center}
Finally, it can be checked that the last pattern
\smash{\input{fig/WJP4233.tex}}
is sufficient to generate a seed $\VJP_i$,
thanks to Theorem~\ref{theo:bad_JP} below
(the proof of which does not make use of Lemma~\ref{lemm:balls_JP}).
\end{proof}

\section{Main results}
\label{sect:mainresults}

We first state discrete plane generation results in Section~\ref{subsec:generation}.
We then state corollaries of topological, dynamical and number-theoretical nature
in Sections~\ref{subsec:applis_topo},~\ref{subsec:applis_dyn}~and~\ref{subsec:applis_nt}.
We recall that a product $\sBrun_{i_1} \cdots \sBrun_{i_n}$
is \tdef{Brun-admissible} if $i_k = 3$ for some $n \in \{1, \ldots, n\}$,
and that a product $\sJP_{a_1,b_1} \cdots \sJP_{a_n,b_n}$
is \tdef{Jacobi-Perron-admissible}
if for every $1 \leq k \leq n$, we have
(1) $0 \leq a_k \leq b_k$,
(2) $b_k \neq 0$, and
(3) $a_k = b_k$ implies $a_{k+1} \neq 0$.
We also recall that the \tdef{combinatorial radius} $\rad(P)$ of a pattern $P$ containing the origin
measures the minimal distance between $\mathbf 0$ and the boundary of $P$
(see Section~\ref{subsec:dissub} for a precise definition).

\subsection{Discrete plane generation}
\label{subsec:generation}

\paragraph{Existence of finite seeds}

We start by stating the existence of finite seeds
such that iterating \emph{any} (Brun or Jacobi-Perron) admissible sequence of substitutions on them
generates patterns with arbitrarily large minimal combinatorial radius.
These finite seeds $\smash{\VBrun_1, \VBrun_2}$ and $\smash{\VJP_1, \VJP_2, \VJP_3, \VJP_4}$ are given explicitly
in Propositions~\ref{prop:minannuli_Brun} and~\ref{prop:minannuli_JP},
and are proved to be minimal in the sense that every admissible discrete plane must contain one of them.

\begin{theo}
\label{theo:seeds}
For every $R \geq 0$ there exists $N \geq 0$ such that
\begin{enumerate}
\item
\label{seed_statement1}
$\rad(\SBrun_{i_1} \cdots \SBrun_{i_n}(V)) \geq R$
for every $(i_1 \cdots i_n) \in \{1,2,3\}^n$ that contains more than $N$ occurrences of $3$,
and for every $V \in \{\VBrun_1, \VBrun_2\}$;
\item
\label{seed_statement2}
$\rad(\SJP_{a_1,b_1} \cdots \SJP_{a_n,b_n}(V)) \geq R$
for every admissible Jacobi-Perron expansion $(a_1,b_1)$, \ldots, $(a_n,b_n)$ such that $n \geq N$,
and for every $V \in \{\VJP_1, \VJP_2, \VJP_3, \VJP_4\}$.
\end{enumerate}
\end{theo}

\begin{proof}[Proof of Theorem~\ref{theo:seeds}]
We fix $V \in \{\VBrun_1,\VBrun_2\}.$
The idea of the proof is that given $n$ concentric nested annuli around the initial pattern $V$,
applying dual substitutions preserves these $n$ concentric nested annuli (thanks to the annulus property),
and moreover eventually generates a new annulus around $V$ (thanks to generation graphs).
This gives $n+1$ concentric nested annuli, and repeating this reasoning yields the result.

We now formalize the above reasoning to prove Statement~\ref{seed_statement1}
of Theorem~\ref{theo:seeds} about Brun substitutions.
By virtue of Lemma~\ref{lemm:finitegraph_Brun_H},
there exists $M \geq 1$ such that
for every $(i_1, \ldots, i_n) \in \{1,2,3\}^n$
containing more than $M$ occurrences of $3$,
the pattern $\smash{\SBrun_{i_1} \cdots \SBrun_{i_n}(V)}$
contains an $\LBrun$-annulus of $\VBrun_1$ or $\VBrun_2$.

Let $R \geq 0$,
let $N = R \times M$,
and let $(i_1, \ldots, i_n) \in \{1,2,3\}^n$
containing more than $N$ occurrences of $3$.
Write $\smash{\SBrun_{i_1} \cdots \SBrun_{i_n} = \Sigma_1 \cdots \Sigma_R}$,
where each $\Sigma_k$ is a product of $\SBrun_i$ containing $\SBrun_3$ at least $M$ times.
To prove that $\rad(\Sigma_1 \cdots \Sigma_R(V)) \geq R$,
it is enough to prove that $\Sigma_1 \cdots \Sigma_R(V)$
contains $R$ nested annuli $A_1, \ldots, A_R$, that is,
annuli such that $A_1$ is an $\LBrun$-annulus of $\smash{\VBrun_j}$ for some $j \in \{1,2\}$,
and for all $k \in \{1, \ldots, R-1\}$,
$A_k$ is an $\LBrun$-annulus of $\smash{A_{k-1} \cup \cdots \cup A_1 \cup \VBrun_j}$.
Observe that $V$ is not necessarily equal to ${\VBrun_j}$.

The reasoning goes by induction.
First, $\Sigma_R(V)$ contains an $\LBrun$-annulus of $\smash{\VBrun_j}$ for some $j \in \{1,2\}$,
since $\Sigma_R$ contains at least $M$ occurrences of $\SBrun_3$, according to Lemma~\ref{lemm:finitegraph_Brun_H}.
Now suppose that for some $k \in \{1, \ldots, R-1\}$,
the pattern $\Sigma_{R-k+1} \cdots \Sigma_R(V)$
contains $k$ concentric nested annuli $A_1, \ldots, A_k$ around $\smash{\VBrun_j}$, with $j \in \{1,2\}$.
We use the annulus property for Brun substitutions:
thanks to
Proposition~\ref{prop:annulus_induction} (the annulus property),
Proposition~\ref{prop:cov_Brun} (strong covering conditions)
and Proposition~\ref{prop:propA_Brun} (Property~A),
the patterns
$\Sigma_{R-k}(A_1), \ldots, \Sigma_{R-k}(A_k)$
are $k$ concentric $\LBrun$-annuli.
Moreover, $\smash{\Sigma_{R-k}(\VBrun_j)}$ contains an $\LBrun$-annulus of $\VBrun_1$ or $\VBrun_2$.
In total this makes $k+1$ concentric annuli contained in $\Sigma_{R-k}\Sigma_{R-k+1} \cdots \Sigma_R(V)$.
They have no face in common thanks to Proposition~\ref{prop:imgplane},
so the induction step holds and Statement~\ref{seed_statement1} is proved.

Statement~\ref{seed_statement2} for Jacobi-Perron substitutions can be proved by the same reasoning.
We apply Lemma~\ref{lemm:finitegraph_Brun_H} (to generate an annulus around the pattern),
Proposition~\ref{prop:annulus_induction} (the annulus property),
Proposition~\ref{prop:cov_JP} (strong covering conditions),
and Proposition~\ref{prop:propA_JP} (Property~A).

Note that in the above proofs, we have implicitly used the fact that if a pattern contains $\mcU$
and an $\LBrun$- or $\LJP$-annulus of $\mcU$, then it must contain
one of the patterns $\VBrun_i$ or $\VJP_i$.
This is proved in Proposition~\ref{prop:minannuli_Brun} for Brun substitutions
and in Proposition~\ref{prop:minannuli_JP} for Jacobi-Perron substitutions.
\end{proof}

\paragraph{Characterizing sequences for which $\mcU$ is not a seed}

We now give characterizations of the sequences of Brun and Jacobi-Perron substitutions such that
the pattern $\mcU = \input{fig/U.tex}$ is not a seed,
that is, such that the patterns generated by iterating the substitutions from $\mcU$
do not generate an entire discrete plane.

\begin{theo}
\label{theo:bad_Brun}
Let $(i_n)_{n \geq 1} \in \{1,2,3\}^\bbN$ be a Brun-admissible sequence
and let $P_n = \SBrun_{i_1} \cdots \SBrun_{i_n}(\mcU)$ for $n \geq 1$.
Then $(P_n)_{n\geq 0}$ is an increasing sequence of patterns which are contained in the discrete plane $\Gv$,
where $\bfv \in \bbR^3_{>0}$ is the vector whose Brun expansion is $(i_n)_{n \geq 1}$.
Moreover, we have $\bigcup_{n\geq 0} P_n \varsubsetneq \Gv$
if and only if there exist $k$ such that $(i_n)_{n \geq k}$
is the labelling of an infinite backward path
$\vphantom{a^a_a}\smash{\cdots \overset{i_{k+1}}{\rightarrow} \bullet \overset{i_k}{\rightarrow} \bullet}$
in the following graph:
\begin{center}
\input{fig/graph_bad_Brun_small.tex}
\end{center}
\end{theo}

We remark that for a \emph{finite} product of Brun substitutions $\sigma = \sBrun_{i_1} \cdots \sBrun_{i_m}$,
Theorem~\ref{theo:bad_Brun} can be interpreted in the following way:
the sequence of patterns $(\EOSS^n(\mcU))_{n \geq 1}$
does not generate an entire discrete plane if and only if there is a loop labelled by
(some power of) the word $i_m \ldots i_1$ in the graph.
This is the case for example with $i_4i_3i_2i_1 = 3332$,
but not with $i_3i_2i_1 = 332$ (it is easy to check that there is no loop labelled by a power of $332$).
Note that in some cases it is necessary to take a power of $i_m \ldots i_1$ for such a loop to exist,
for example with $i_2i_1 = 32$:
indeed, there is a loop labelled by $3232$ but no loop labelled by $32$.

\begin{proof}
The fact that $(P_n)_{n\geq 0}$ is an increasing sequence contained in $\Gv$
easily follows from Proposition~\ref{prop:imgplane}
and the fact that $\mcU$ is contained in every discrete plane.
We now prove the second part of the proposition.
First, recall that the four-vertex graph given in the statement of Theorem~\ref{theo:bad_Brun}
and the graph $\GBrunmini$ given in Section~\ref{subsec:graph_Brun}, Figure~\ref{fig:GBrunmini},
admit the same sets of infinite backward path labellings with infinitely many $3$s.
(The four-vertex graph is simply a ``reduced'' version of $\GBrunmini$ with some identified vertices.)

Assume now that there does not exist $k$ such that $(i_n)_{n \geq k}$
is the labelling of an infinite backward path
in the four-vertex graph.
Then, for every $k \geq 0$, there exists $M \geq 0$ such that for every path
$\vphantom{a^a_a}\smash{f_M \overset{i_{k+M}}{\rightarrow} f_{M-1} \cdots f_1\overset{i_k}{\rightarrow} f_0}$ in $\GBrun$,
we have $f_M \in \mcU$.
Indeed, the only vertices in $\GBrun \setminus \GBrunmini$ which are
at the end of a backward infinite path with infinitely many $3$s
are the vertices of $\mcU$.
It then follows from Statement~\ref{graphprop1} of Proposition~\ref{prop:gengraph}
that $\smash{\SBrun_{i_k} \cdots \SBrun_{i_{k+M}}(\mcU)}$
contains an $\LBrun$-annulus of $\mcU$.
By considering $k$ large enough in the above reasoning,
we can now apply Theorem~\ref{theo:seeds} with $\mcV = \{\mcU\}$ and $\mcW = \{\VBrun_1, \VBrun_2\}$
to prove that the patterns $P_n=\SBrun_{i_1} \cdots \SBrun_{i_n}(\mcU)$ have arbitrarily large combinatorial radius,
which implies that $\bigcup_{n\geq 0} P_n = \Gv$.

Conversely, if there exists $k$ such that $(i_n)_{n \geq k}$
is the labelling of an infinite backward path
in the four-vertex graph, then
Proposition~\ref{prop:gengraph}~Statement~\ref{graphprop2}
implies that the patterns $(\SBrun_{i_k} \cdots \SBrun_{i_{k + n}}(\mcU))_{n \geq 0}$
contain no $\LBrun$-annulus of $\mcU$ for any $n \geq 0$,
so the discrete plane $\Gv$ is not generated by the patterns
$(\SBrun_{i_1} \cdots \SBrun_{i_{k+n}}(\mcU))_{n \geq 1}
= (\SBrun_{i_1} \cdots \SBrun_{i_{k-1}} \SBrun_{i_k} \cdots \SBrun_{i_{k + n}}(\mcU))_{n \geq 0}$.
\end{proof}

Theorem~\ref{theo:bad_JP} below
recovers the same characterization as the one stated in~\cite{IO94},
by using a different notion of generation graphs,
and which thus does not require the same succession of lemmas as found in~\cite{IO94}.

\begin{theo}
\label{theo:bad_JP}
Let $(a_n,b_n)_{n \geq 1}$ be an admissible Jacobi-Perron sequence
and for $n \geq 1$ define $P_n = \SJP_{a_1,b_1} \cdots \SJP_{a_n,b_n}(\mcU)$.
Then $(P_n)_{n\geq 0}$ is an increasing sequence of patterns which are contained in the discrete plane $\Gv$,
where $\bfv \in \bbR^3_{>0}$ is the vector whose Jacobi-Perron expansion is $(a_n,b_n)_{n \geq 1}$.
Moreover, we have $\bigcup_{n\geq 0} P_n \varsubsetneq \Gv$
if and only if there exists $\ell \geq 1$ such that for every $k \geq 0$, we have
(1) $a_{\ell+3k} = 0$,
(2) $a_{\ell+3k+1} = b_{\ell+3k+1}$
and (3) $0 < a_{\ell+3k+2} < b_{\ell+3k+2}$.
\end{theo}

\begin{proof}
Let $(i_n)_{n \geq 1}$ be the additive Jacobi-Perron expansion corresponding to $(a_n,b_n)_{n \geq 1}$
(defined in Section~\ref{subsec:JP}).
Since this additive expansion is admissible, it must contain infinitely many $3$s or $4$s (see Remark~\ref{rema:JPAdditive}).
Similarly as in the proof of Theorem~\ref{theo:bad_Brun},
the infinite edge labellings of the graph $\GJP$ defined in Section~\ref{subsec:graph_JP}
give us the desired characterization, thanks to Proposition~\ref{prop:gengraph} and Theorem~\ref{theo:seeds}.

Indeed, in the graph $\GJP$,
following an infinite path containing infinitely many $3$s or $4$s
forces us to turn clockwise, visiting $f_c$, $f_e$ and $f_d$ (and possibly other states between $f_d$ and $f_e$)
cyclically and in this order (see Section~\ref{subsec:graph_JP}).
The result then follows from the definition of the additive decomposition
of $\SJP_{a,b}$ by $\Theta_1, \Theta_2, \Theta_3, \Theta_4$,
and by the following facts, which can easily be checked:
any path from $f_c$ to $f_e$ in $\GJP$ corresponds to $a_n = 0$;
any path from $f_d$ to $f_c$ in $\GJP$ corresponds to $a_n = b_n$;
any path from $f_e$ to $f_d$ in $\GJP$ corresponds to $0 < a_n < b_n$.
\end{proof}

\paragraph{Translates of seeds always occur}

The next theorem states that the initial pattern $\mcU$ is always sufficient
to generate \emph{translates} of patterns with arbitrarily large radius
(even though $\mcU$ is not a seed).
Figure~\ref{fig:gen_goodbad}~p.~\pageref{fig:gen_goodbad} illustrates an example of such a situation for Brun substitutions.

\begin{theo}
\label{theo:balls}
For every $R \geq 0$ there exists $N \geq 0$ such that
\begin{enumerate}
\item
\label{ball_statement1}
$\SBrun_{i_1} \cdots \SBrun_{i_n}(\mcU)$ contains a translate of a pattern $P$ with $\rad(P) \geq R$
for every $(i_1 \cdots i_n) \in \{1,2,3\}^n$ that contains more than $N$ occurrences of $3$;
\item
\label{ball_statement2}
$\SJP_{a_1,b_1} \cdots \SJP_{a_n,b_n}(\mcU)$ contains a translate of a pattern $P$ with $\rad(P) \geq R$
for every admissible Jacobi-Perron expansion $(a_1,b_1), \ldots, (a_n,b_n)$ such that $n \geq N$.
\end{enumerate}
\end{theo}

\begin{proof}
Since we are dealing with \emph{translates} of seeds,
we cannot use generation graphs like in the previous theorems,
we must track ``by hand'' all the possible sequences of iterations starting from $\mcU$,
and explicitly check that a translate of one of the seeds of Theorem~\ref{theo:seeds} is eventually generated.
This was done in detail in Lemma~\ref{lemm:balls_Brun} for Statement~\ref{ball_statement1}
and in Lemma~\ref{lemm:balls_JP} for Statement~\ref{ball_statement2}.
The result then follows directly from Theorem~\ref{theo:seeds}:
once a translate of a seed occurs,
its images are translates of patterns with arbitrarily large minimal combinatorial radius.
(Note that by ``linearity'' of $\EOS$ dual substitutions (Proposition~\ref{prop:imgplane}~\ref{imgplanestatement2}),
the image of a translate of a pattern is always a translate of its image.)
\end{proof}

\subsection{Topological properties of Rauzy fractals for products of substitutions}
\label{subsec:applis_topo}

Before stating the results of this section,
we stress the fact that there are many connections between topological properties of Rauzy fractals
(such as being connected or having zero interior point)
and number-theoretical properties of the associated dominant Pisot eigenvalue.
This is described in more detail in Section~\ref{subsec:applis_nt}.

\paragraph{Zero interior point}
We consider a further interpretation of a
discrete plane generation property in terms of Rauzy fractals.
More precisely, the origin is an interior point of the Rauzy fractal of $\sigma$
if and only if the patterns $\EOS(\sigma)^n(\mcU)$
generate patterns with arbitrarily large minimal combinatorial radius centered at the origin
(see~\cite{BS05} or~\cite{SieT09}).
An immediate consequence is the following theorem
for Brun and Jacobi-Perron substitutions
(obtained thanks to Theorem~\ref{theo:bad_Brun} and Theorem~\ref{theo:bad_JP}),
which is illustrated on some finite products of Brun substitutions in Figure~\ref{fig:rfbrun_zero}~p.~\pageref{fig:rfbrun_zero}.

\begin{coro}
\label{coro:fractal_zero}
We have:
{
\makeatletter
\@beginparpenalty=10000
\makeatother
\begin{itemize}
\item
The origin is \textbf{not} an interior point of
the Rauzy fractal of an admissible product $\sBrun_{i_1} \cdots \sBrun_{i_n}$
if and only if there exists a cycle labelled by $i_1, \ldots, i_n$
in the directed graph given in Theorem~\ref{theo:bad_Brun}.
\item
The origin is \textbf{not} an interior point of
the Rauzy fractal of an admissible product $\sJP_{a_1,b_1} \cdots \sJP_{a_n,b_n}$
if and only if the infinite sequence $((a_1, b_1), \ldots, (a_n,b_n))^\infty$
satisfies the condition given in Theorem~\ref{theo:bad_JP}.
\end{itemize}
}
\end{coro}

Note that in the case of Arnoux-Rauzy substitutions,
the seed $\mcU$ always generates an entire discrete plane when iterating an admissible sequence,
so for Arnoux-Rauzy substitutions the origin is always an interior point of the Rauzy fractal~\cite{BJS12}.
Also note that the fact that the origin is an interior point of the Rauzy fractal
implies pure discrete spectrum of the associated dynamical system \cite{SieT09,CANT}.
It also has interpretations in number-theoretical terms, which will be discussed in
Sections~\ref{subsec:applis_dyn} and~\ref{subsec:applis_nt}.

\begin{figure}[!ht]%
\centering
\subfloat[][$\sBrun_2\sBrun_3\sBrun_2$]{%
    \label{subfig:rfbruna}%
    \includegraphics[height=29mm]{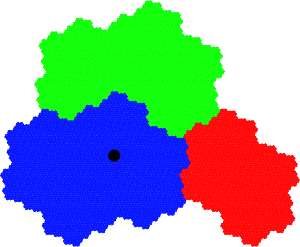}
}
\hfil
\subfloat[][$\sBrun_2\sBrun_3\sBrun_1\sBrun_1$]{%
    \label{subfig:rfbrunb}%
    \includegraphics[height=29mm]{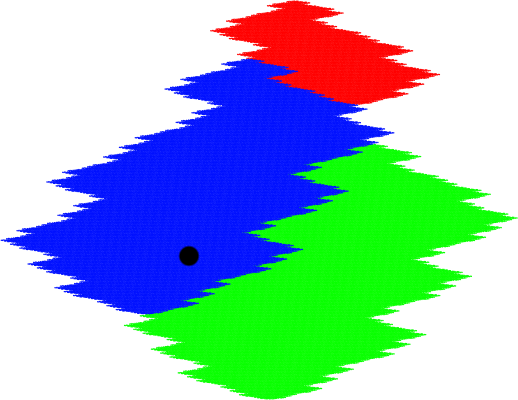}
}
\hfil
\subfloat[][$\sBrun_1\sBrun_1\sBrun_3\sBrun_2$]{%
    \label{subfig:rfbrunc}%
    \includegraphics[height=29mm]{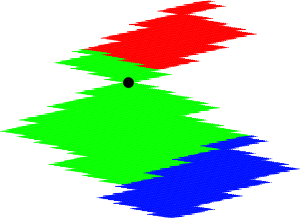}
}
\caption[]{%
Some Rauzy fractals of products of Brun substitutions.
Only~\subref{subfig:rfbrunc}
does not have the origin as an interior point.
An example of a product of substitutions whose Rauzy fractal has zero as an interior point
but whose reverse product does not is provided by \subref{subfig:rfbrunb} (and its reverse product~\subref{subfig:rfbrunc}).
The fractals of \subref{subfig:rfbruna} and \subref{subfig:rfbrunc}
correspond to the two examples of substitutions shown in Figure~\ref{fig:gen_goodbad}~p.~\pageref{fig:gen_goodbad}.
}
\label{fig:rfbrun_zero}%
\end{figure}

\paragraph{Connectedness}
The following result states that the Rauzy fractals of the substitutions
in the families we consider are all connected.
It has been proved in the case of Arnoux-Rauzy substitutions using similar tools~\cite{BJS13}.
One interest of such a topological property comes from the fact
that the Rauzy fractal provides a geometric realization for the two dynamics
that act on $(X_{\sigma}, S)$, namely the dynamics of the shift
(the toral translation in the case of pure discrete spectrum), but also the
dynamics provided by the (positive entropy) action of the substitution.
Indeed, the subtiles of the Rauzy fractals
form a Markov partition for the toral automorphism given by the incidence matrix of the substitution
(in case of pure discrete spectrum).
See Corollary~\ref{coro:dynprod} and~\ref{coro:markov} below for more detail.

\begin{coro}
\label{coro:fractal_connected}
Let $\sigma$ be an admissible finite product of Brun or Jacobi-Perron substitutions.
The Rauzy fractal $\mcT_\sigma$ and its subtiles $\mcT_\sigma(1)$, $\mcT_\sigma(2)$, $\mcT_\sigma(3)$
 are all connected.
\end{coro}

Note that such fractals are not necessarily simply connected.
This raises the question of characterizing the admissible finite product of Brun or Jacobi-Perron substitutions
that yield a simply connected fractal, which is an interesting and seemingly difficult question.

\begin{proof}
Let $\sigma$ be an admissible product of Brun or a Jacobi-Perron substitutions.
Thanks to Propositions~\ref{prop:cov_Brun} and Proposition~\ref{prop:cov_JP},
the patterns $\EOS(\sigma)^n(\mcU)$ are all $\LBrun$-connected (or respectively $\LJP$-connected).
Let $\mcL$ stand for $\LBrun$ or for $\LJP$ according to the fact that $\sigma$ is
an admissible product of Brun or a Jacobi-Perron substitutions.
Since $\mcL$ consists of connected patterns only,
the definition of $\mcL$-coverings implies that every $\mcL$-covered pattern is path-connected.
It follows that the patterns $\EOS(\sigma)^n(\mcU)$ are connected for every $n \geq 0$.
Consequently, since the Rauzy fractal $\mcT_\sigma$ of $\sigma$ is a Hausdorff limit of connected sets, it is also connected.
The same reasoning yields the connectedness of $\mcT_\sigma(1)$, $\mcT_\sigma(2)$, $\mcT_\sigma(3)$,
since it can easily be checked that the images of $[\mathbf 0,1]^\star$, $[\mathbf 0,2]^\star$, $[\mathbf 0,3]^\star$
are all eventually $\mcL$-connected.
\end{proof}

\subsection{Dynamical properties of products of substitutions}
\label{subsec:applis_dyn}

The tools developed in Section~\ref{sect:gen}
enable us to establish dynamical properties of the infinite product families
consisting of arbitrary finite products of Brun or Jacobi-Perron substitutions.
This will be done in particular in Corollary~\ref{coro:dynprod} below.
Even though our tools are not fully algorithmic,
they can be seen as an extension of the algorithms used in the study of a \emph{single}
substitution, where many algorithms have been developed based on the constructions of graphs
which are proved to be finite thanks to the Pisot property (see~\cite{SieT09,CANT,AL11}).

\begin{coro}
\label{coro:dynprod}
Let $\sigma$ be an admissible finite product $\sigma$ of Brun or Jacobi-Perron substitutions.
Then the system $(X_\sigma,S)$ has pure discrete spectrum.
Furthermore, the Rauzy fractal of $\sigma$ (and its subtiles)
provide the domain and partition of a symbolic natural coding of this translation.
These subtiles are moreover connected bounded remainder sets for this translation.
\end{coro}

\begin{proof}
This theorem follows from Theorem~\ref{theo:balls} together with~\cite{IR06},
where it is proved that $(X_\sigma,S)$ is measure-theoretically conjugate to a toral translation
if the sequence $(\EOS(\sigma)^n(\mcU))_n$ satisfies the arbitrarily large disks covering property,
that is, if it contains translates of patterns
with arbitrarily large minimal combinatorial radius.
One then deduces that $(X_\sigma,S)$ is a symbolic natural coding of this translation,
with the Rauzy fractal $\mcT_\sigma$ together with its subtiles
providing the fundamental domain and its partition
(the translation vectors $\pic (\bfe_i)$ of the exchange of pieces $(\mcT_\sigma, E)$
described in Section~\ref{subsec:dynamics} provide the restriction of the translation
to the fundamental domain $\mcT_\sigma$).
Connectedness is a consequence of Corollary~\ref{coro:fractal_connected}.

It remains to check that the subtiles of the Rauzy fractal provide bounded remainder sets.
A bi-infinite word $u \in \mcA^{\mathbb Z}$ is said to have bounded discrepancy if there exist $f_i$, $i \in \mcA$,
and $C>0$ such that
$\limsup_{i \in \mcA, N \in \bbN} | | u_{-N} \cdots u_0 u_1 \ldots u_N | _i -(2N+1) f_i| \leq C$.
If a coding $u$ of a minimal translation is
such that its elements have bounded symbolic discrepancy,
then the elements of the associated partition are bounded remainder sets.
We conclude by noticing that bi-infinite words in $(X_{\sigma}, S)$
have bounded symbolic discrepancy since $\sigma$ is an irreducible Pisot substitution (see e.g.~\cite{Adamdis}).
\end{proof}

\paragraph{Markov partition for toral automorphisms}
Markov partitions for the toral automorphisms $(\bbT^3, \bfM)$ are known to exist in the case
where $\bfM$ is hyperbolic (no eigenvalues lie on the unit circle)~\cite{Bow08,Manning}.
There is a general way to construct them explicitly in dimension two with rectangles~\cite{AW70,Adl98},
but no such constructions are known for dimensions $d \geq 3$.
Indeed a result of Bowen~\cite{Bow78} states that such partitions must have fractal boundary if $d \geq 3$,
which reduces the hopes of finding explicit Markov partitions with a simple shape.

In the particular case where $\bfM$ is the incidence matrix of a Pisot unimodular substitution
(or the companion matrix of a unit Pisot number $\beta$ in the framework of beta-numeration),
Rauzy fractals provide a way to construct explicit Markov partitions for $(\bbT^n, \bfM)$.
Such an approach was initiated in~\cite{IO93,KV98,Pra99}
and was proved to be successful in the Pisot case~\cite{IR06,Sie00},
provided that the associated Rauzy fractal satisfies some tiling properties
(which are equivalent to pure discrete spectrum of the associated symbolic dynamical system).

Consequently, another dynamical application of Theorems~\ref{theo:seeds} and~\ref{theo:balls}
is that for every matrix $\bfM$
which is the incidence matrix of an admissible product of Brun or Jacobi-Perron substitutions,
an explicit Markov partition can be constructed for the toral automorphisms $(\bbT^3, \bfM)$.

\begin{coro}
\label{coro:markov}
Let $\sigma$ be an admissible finite product $\sigma$ of Brun or Jacobi-Perron substitutions.
Then, there exists a three-set partition of the torus $\bbT^3$ with connected elements
which is a Markov partition for the toral automorphism $(\bbT^3, \bfM_\sigma)$
provided by the incidence matrix $\Ms$ of $\sigma$.
\end{coro}

This result is a consequence of the results from~\cite{IO93,Pra99,Sie00}
(which establish links between the Rauzy fractal of $\sigma$ and Markov partitions of $\bfM_\sigma$),
together with Corollary~\ref{coro:dynprod}.
A similar result for Arnoux-Rauzy substitutions is stated in~\cite{BJS13}.
Also note that the connectedness of the domains of the Markov partition in question
is due to the connectedness of Rauzy fractals (Corollary~\ref{coro:fractal_connected} in Section~\ref{subsec:applis_topo}).

\subsection{Number-theoretical implications}
\label{subsec:applis_nt}

\paragraph{Fractal tiles for cubic number fields}
By a result of Paysant-Le-Roux and Dubois~\cite{PD84} together with Corollary~\ref{coro:dynprod},
we are able to associate natural codings with respect to a domain and a partition provided
by Rauzy fractals with every cubic field.

\begin{theo}
\label{theo:cubicfrac}
For every cubic extension $\bbK$ of $\bbQ$,
there exist $\alpha, \beta \in \bbK$ and a three-letter unimodular irreducible Pisot substitution $\sigma$
such that $\bbK = \bbQ(\alpha, \beta)$
and such that the toral translation $(\bbT^2, x \mapsto x + (\begin{smallmatrix}\alpha \\ \beta\end{smallmatrix}))$
is measure-theoretically conjugate to $(X_\sigma, S)$.
Furthermore, it admits a symbolic natural coding with respect to a partition made of connected bounded remainder sets.
\end{theo}

\begin{proof}
According to~\cite[Proposition IV]{PD84}, any cubic field is generated by some Pisot number $\beta$
with minimal polynomial $P(X)=X^3 - c_2 X^2+c_1 X-1=0$,
where $c_2 \geq 2 c_1-2$ and $c_1 \geq 3$.
If we set $\sigma = \sJP_{0,1}\sJP_{0,1}\sJP_{c_1-3,c_2-c_1}$
(an admissible product because $c_2-c_1 >c_1-2$), then $\Ms$ has characteristic polynomial $P(X)$.
Since $\Ms$ is primitive, it admits a positive eigenvector $\bfv$ with $\bfv_1+\bfv_2+\bfv_3=1$
and we have $\bbQ(\bfv_2,\bfv_3)=\bbK$.
The Jacobi-Perron expansion of $\bfv$ has period $(c_1-3,c_2-c_1)(0,1),(0,1)$
(this is an admissible expansion).
Now consider the projection $\pi_{\bfv,\mathbf 1^\bot}$ along $\bfv$ onto the antidiagonal plane with normal vector $(1,1,1)$.
This projection expresses as
\[
\pi_{\mathbf {v}, \mathbf 1^\bot} (x,y,x)
= (\bfv_2(x+y+z)-x) (\bfe_3 - \bfe_1) + (\bfv_3(x+y+z)-y) (\bfe_3 - \bfe_2).
\]
In particular, the vector $\pi_{\mathbf {v},\mathbf{1}^\bot}({\bf e}_3)$ expressed in the basis
$(\bfe_3 - \bfe_1, \bfe_3 - \bfe_2)$ has coordinates $(\bfv_2,\bfv_3)$.
Consider the translation by the vector $(\pi_{\bfv,\mathbf 1^\bot} \bfe_3)$
modulo the lattice generated by $\pi_{\bfv,\mathbf 1^\bot} (\bfe_3 - \bfe_1)$
and $\pi_{\bfv,\mathbf 1^\bot} (\bfe_3 - \bfe_2)$.
It is measure-theoretically conjugate to the translation by $(\bfv_2,\bfv_3)$ modulo $\bbZ^2$.
The projection $\pi_{\bfv,\mathbf 1^\bot}$ can be used to define a Rauzy fractal for $\sigma$
instead of the projection $\pi_\sigma$ used in the definition of Rauzy fractals in Section~\ref{subsec:dynamics}.
This means that we are working with a domain exchange whose translation vectors are given by $\pi$ instead of $\pic$.
According to Corollary~\ref{coro:dynprod}, the domain exchange $(\mcT_\sigma(i),E)$ (obtained with $\pic$) factors onto a translation,
so the domain exchange provided by $\pi$ also factors onto a translation~\cite{IR06}.
This yields the desired result by taking $\alpha = \bfv_2, \beta = \bfv_3$.
\end{proof}

Observe that the existence of a symbolic natural coding
provides equidistribution properties for the associated translation.
See~\cite{Adam04} for an illustration in the case of one-dimensional translations.

\paragraph{Number-theoretical applications of topological properties of Rauzy fractals}
Several of the topological properties of Rauzy fractals have direct number-theoretical consequences.
Consider first the framework of (Pisot) $\beta$-numeration. The connectedness of the Rauzy fractal
(called central tile in this context) is conjectured to guarantee explicit relations between
the norm of $\beta$ and the $\beta$-expansion of $1$~\cite{AG05}.
Furthermore, the properties of rational numbers with purely periodic $\beta$-expansions
are now known to be closely related to the shape of the boundary of the Rauzy fractal~\cite{ABBS08,AFSS10}.
More generally, at the interplay between substitutions and Diophantine approximation for cubic number fields,
explicit computation of the size of the largest ball contained in the Rauzy fractal is used in~\cite{HM06}
to determine the sequence of best approximations with respect to a specific norm,
for some two-dimensional vectors provided by non-totally real cubic Pisot units.
See also~\cite{Ito03,ItoYasu07} for a study of the limit set of the points $\sqrt q (\|q \alpha\|, \|q \beta\|$),
with $q$ being a positive integer, and with $(\alpha,\beta)$ such that $(1,\alpha,\beta)$ forms
a basis of a non-totally real cubic number field.
This limit set is known to be a union of homothetic ellipses centered at the origin.
The question of when the closest of these ellipses is provided by Brun algorithm is investigated in~\cite{Ito03,ItoYasu07}.

The fact that the origin is an interior point of the Rauzy fractal is also a particularly interesting property.
It is closely related to the so-called finiteness properties in numeration.
Indeed, the \emph{(F)~property} (introduced and called finiteness property in~\cite{FS92}),
which expresses some finiteness properties of digital expansions in a Pisot base $\beta$,
is equivalent to the fact that the origin is an interior point of the central tile
associated with $\beta$~\cite{Aki99,Aki02}.
Several variants of the (F)~property have then been proposed,
one of them being the \emph{extended (F)~property} introduced in~\cite{BS05,FT06}
to extend the classical (F)~property to the numeration systems associated with substitutions
(called the Dumont-Thomas numeration).
The extended (F)~property can again be stated in topological terms:
it holds if and only if the origin is an interior point
of the Rauzy fractal associated with $\sigma$ (see~\cite{SieT09,CANT}).
Consequently, Corollary~\ref{coro:fractal_zero} above,
which characterizes the products of Arnoux-Rauzy, Brun and Jacobi-Perron substitutions
for which the origin is an interior point of the Rauzy fractal,
also provides a characterization of when such a finite product of substitutions satisfies the extended (F)~property.
Furthermore, this is a sufficient condition for pure discrete spectrum which can also be interpreted as
a coincidence type condition~\cite{CANT}.

\bibliographystyle{amsalpha}
\bibliography{biblio}

\newcommand{\etalchar}[1]{$^{#1}$}
\providecommand{\bysame}{\leavevmode\hbox to3em{\hrulefill}\thinspace}
\providecommand{\MR}{\relax\ifhmode\unskip\space\fi MR }
\providecommand{\MRhref}[2]{%
  \href{http://www.ams.org/mathscinet-getitem?mr=#1}{#2}
}
\providecommand{\href}[2]{#2}
\begin{thebibliography}{DFPLR04}

\bibitem[Ada04a]{Adam04}
Boris Adamczewski, \emph{R\'epartition des suites {$(n\alpha)_{n\in\mathbb N}$}
  et substitutions}, Acta Arith. \textbf{112} (2004), no.~1, 1--22.

\bibitem[Ada04b]{Adamdis}
\bysame, \emph{Symbolic discrepancy and self-similar dynamics}, Ann. Inst.
  Fourier (Grenoble) \textbf{54} (2004), no.~7, 2201--2234 (2005).

\bibitem[AFSS10]{AFSS10}
Boris Adamczewski, Christiane Frougny, Anne Siegel, and Wolfgang Steiner,
  \emph{Rational numbers with purely periodic $\beta$-expansion}, Bull. Lond.
  Math. Soc. \textbf{42} (2010), no.~3, 538--552.

\bibitem[Adl98]{Adl98}
Roy~L. Adler, \emph{Symbolic dynamics and {M}arkov partitions}, Bull. Amer.
  Math. Soc. (N.S.) \textbf{35} (1998), no.~1, 1--56.

\bibitem[AW70]{AW70}
Roy~L. Adler and Benjamin Weiss, \emph{Similarity of automorphisms of the
  torus}, Memoirs of the American Mathematical Society, No. 98, American
  Mathematical Society, Providence, R.I., 1970.

\bibitem[Aki99]{Aki99}
Shigeki Akiyama, \emph{Self affine tiling and {P}isot numeration system},
  Number theory and its applications ({K}yoto, 1997), Dev. Math., vol.~2,
  Kluwer Acad. Publ., Dordrecht, 1999, pp.~7--17.

\bibitem[Aki00]{Aki00}
\bysame, \emph{Cubic {P}isot units with finite beta expansions}, Algebraic
  number theory and {D}iophantine analysis ({G}raz, 1998), de Gruyter, Berlin,
  2000, pp.~11--26.

\bibitem[Aki02]{Aki02}
\bysame, \emph{On the boundary of self affine tilings generated by {P}isot
  numbers}, J. Math. Soc. Japan \textbf{54} (2002), no.~2, 283--308.

\bibitem[ABBS08]{ABBS08}
Shigeki Akiyama, Guy Barat, Val{\'e}rie Berth{\'e}, and Anne Siegel,
  \emph{Boundary of central tiles associated with {P}isot beta-numeration and
  purely periodic expansions}, Monatsh. Math. \textbf{155} (2008), no.~3-4,
  377--419.

\bibitem[ABB{\etalchar{+}}]{ABBLS}
Shigeki Akiyama, Marcy Barge, Val{\'e}rie Berth{\'e}, Jeong-Yup Lee, and Anne
  Siegel, \emph{On the {P}isot substitution conjecture}, Aperiodic order,
  Progress in mathematics, Birkh{\"a}user, to appear.

\bibitem[AG05]{AG05}
Shigeki Akiyama and Nertila Gjini, \emph{Connectedness of number theoretic
  tilings}, Discrete Math. Theor. Comput. Sci. \textbf{7} (2005), no.~1,
  269--312 (electronic).

\bibitem[AL11]{AL11}
Shigeki Akiyama and Jeong-Yup Lee, \emph{Algorithm for determining pure
  pointedness of self-affine tilings}, Adv. Math. \textbf{226} (2011), no.~4,
  2855--2883.

\bibitem[AL14]{AL14}
\bysame, \emph{Overlap coincidence to strong coincidence in substitution tiling
  dynamics}, European J. Combin. \textbf{39} (2014), 233--243.

\bibitem[AI01]{AI01}
Pierre Arnoux and Shunji Ito, \emph{Pisot substitutions and {R}auzy fractals},
  Bull. Belg. Math. Soc. Simon Stevin \textbf{8} (2001), no.~2, 181--207.

\bibitem[AR91]{AR91}
Pierre Arnoux and G{\'e}rard Rauzy, \emph{Repr\'esentation g\'eom\'etrique de
  suites de complexit\'e {$2n+1$}}, Bull. Soc. Math. France \textbf{119}
  (1991), no.~2, 199--215.

\bibitem[AD13]{AD13}
Artur Avila and Vincent Delecroix, \emph{Pisot property for the {B}run and
  fully subtractive algorithms}, preprint, 2013.

\bibitem[BBK06]{BBK06}
Veronica Baker, Marcy Barge, and Jaroslaw Kwapisz, \emph{Geometric realization
  and coincidence for reducible non-unimodular {P}isot tiling spaces with an
  application to {$\beta$}-shifts}, Ann. Inst. Fourier (Grenoble) \textbf{56}
  (2006), no.~7, 2213--2248, Num{\'e}ration, pavages, substitutions.

\bibitem[Bar14]{Barge14}
Marcy Barge, \emph{Pure discrete spectrum for a class of one-dimensional
  substitution tiling systems}, preprint, 2014.

\bibitem[BD02]{BD02}
Marcy Barge and Beverly Diamond, \emph{Coincidence for substitutions of {P}isot
  type}, Bull. Soc. Math. France \textbf{130} (2002), no.~4, 619--626.

\bibitem[BK06]{BK06}
Marcy Barge and Jaroslaw Kwapisz, \emph{Geometric theory of unimodular {P}isot
  substitutions}, Amer. J. Math. \textbf{128} (2006), no.~5, 1219--1282.

\bibitem[B{\v{S}}W13]{BSW13}
Marcy Barge, Sonja {\v{S}}timac, and Robert~F. Williams, \emph{Pure discrete
  spectrum in substitution tiling spaces}, Discrete Contin. Dyn. Syst.
  \textbf{33} (2013), no.~2, 579--597.

\bibitem[Ber11]{Ber11}
Valerie Berth{\'e}, \emph{Multidimensional {E}uclidean algorithms, numeration
  and substitutions}, Integers \textbf{11B} (2011), Paper No. A2, 34.

\bibitem[BBJS13]{BBJS13}
Val{\'e}rie Berth{\'e}, J{\'e}r{\'e}mie Bourdon, Timo Jolivet, and Anne Siegel,
  \emph{Generating discrete planes with substitutions}, WORDS, Lecture Notes in
  Computer Science, vol. 8079, 2013, pp.~58--70.

\bibitem[BD14]{BD13}
Val{\'e}rie Berth{\'e} and Vincent Delecroix, \emph{Beyond substitutive
  dynamical systems: S-adic expansions}, proceedings of the RIMS conference,
  `Numeration and Substitution 2012', to appear, 2014.

\bibitem[BFZ05]{BFZ05}
Val{\'e}rie Berth{\'e}, S{\'e}bastien Ferenczi, and Luca~Q. Zamboni,
  \emph{Interactions between dynamics, arithmetics and combinatorics: the good,
  the bad, and the ugly}, Algebraic and topological dynamics, Contemp. Math.,
  vol. 385, Amer. Math. Soc., Providence, RI, 2005, pp.~333--364.

\bibitem[BJJP13]{BJJP13}
Val\'erie Berth\'e, Damien Jamet, Timo Jolivet, and Xavier Proven\c{c}al,
  \emph{Critical connectedness of thin arithmetical discrete planes}, Discrete
  Geometry for Computer Imagery, Lecture Notes in Computer Science, vol. 7749,
  2013, pp.~107--118.

\bibitem[BJS12]{BJS12}
Val{\'e}rie Berth{\'e}, Timo Jolivet, and Anne Siegel, \emph{Substitutive
  {A}rnoux-{R}auzy sequences have pure discrete spectrum}, Unif. Distrib.
  Theory \textbf{7} (2012), no.~1, 173--197.

\bibitem[BJS14]{BJS13}
\bysame, \emph{Connectedness of {R}auzy fractals associated with
  {A}rnoux-{R}auzy substitutions}, RAIRO Theor. Inform. Appl. (2014), to
  appear.

\bibitem[BS05]{BS05}
Val{\'e}rie Berth{\'e} and Anne Siegel, \emph{Tilings associated with
  beta-numeration and substitutions}, Integers \textbf{5} (2005), no.~3, A2, 46
  pp.

\bibitem[BST10]{CANT}
Val{\'e}rie Berth{\'e}, Anne Siegel, and J{\"o}rg~M. Thuswaldner,
  \emph{{Substitutions, Rauzy fractals, and tilings}}, Combinatorics, Automata
  and Number Theory, Encyclopedia of Mathematics and its Applications, vol.
  135, Cambridge University Press, 2010.

\bibitem[BST14]{BST14}
Val{\'e}rie Berth{\'e}, Wolfgang Steiner, and J{\"o}rg Thuswaldner,
  \emph{Tilings with {$S$}-adic {R}auzy fractals}, preprint, 2014.

\bibitem[Bow78]{Bow78}
Rufus Bowen, \emph{Markov partitions are not smooth}, Proc. Amer. Math. Soc.
  \textbf{71} (1978), no.~1, 130--132.

\bibitem[Bow08]{Bow08}
\bysame, \emph{Equilibrium states and the ergodic theory of {A}nosov
  diffeomorphisms}, revised ed., Lecture Notes in Mathematics, vol. 470,
  Springer-Verlag, Berlin, 2008, With a preface by David Ruelle, Edited by
  Jean-Ren{\'e} Chazottes.

\bibitem[Bre81]{Bre81}
Arne~J. Brentjes, \emph{Multidimensional continued fraction algorithms},
  Mathematical Centre Tracts, vol. 145, Mathematisch Centrum, 1981.

\bibitem[Bru58]{Bru58}
Viggo Brun, \emph{Algorithmes euclidiens pour trois et quatre nombres},
  Treizi\`eme congr\`es des math\'ematiciens scandinaves, tenu \`a {H}elsinki
  18-23 ao\^ut 1957, Mercators Tryckeri, Helsinki, 1958, pp.~45--64.

\bibitem[CS01]{CS01}
Vincent Canterini and Anne Siegel, \emph{Geometric representation of
  substitutions of {P}isot type}, Trans. Amer. Math. Soc. \textbf{353} (2001),
  no.~12, 5121--5144.

\bibitem[CS03]{Clark-Sadun:03}
Alex Clark and Lorenzo Sadun, \emph{When size matters: subshifts and their
  related tiling spaces}, Ergodic Theory Dynam. Systems \textbf{23} (2003),
  no.~4, 1043--1057.

\bibitem[DFPLR04]{DFP04}
Eug{\`e}ne Dubois, Ahmed Farhane, and Roger Paysant-Le~Roux, \emph{The
  {J}acobi-{P}erron algorithm and {P}isot numbers}, Acta Arith. \textbf{111}
  (2004), no.~3, 269--275.

\bibitem[Fer06]{Fer06}
Thomas Fernique, \emph{Multidimensional {S}turmian sequences and generalized
  substitutions}, Internat. J. Found. Comput. Sci. \textbf{17} (2006), no.~3,
  575--599.

\bibitem[Fer09]{Fer09}
Thomas Fernique, \emph{Generation and recognition of digital planes using
  multi-dimensional continued fractions}, Pattern Recognition \textbf{42}
  (2009), no.~10, 2229--2238.

\bibitem[Fog02]{Fog02}
N.~Pytheas Fogg, \emph{Substitutions in dynamics, arithmetics and
  combinatorics}, Lecture Notes in Mathematics, vol. 1794, Springer-Verlag,
  Berlin, 2002.

\bibitem[FS92]{FS92}
Christiane Frougny and Boris Solomyak, \emph{Finite beta-expansions}, Ergodic
  Theory Dynam. Systems \textbf{12} (1992), no.~4, 713--723.

\bibitem[FT06]{FT06}
Clemens Fuchs and Robert Tijdeman, \emph{Substitutions, abstract number systems
  and the space filling property}, Ann. Inst. Fourier (Grenoble) \textbf{56}
  (2006), no.~7, 2345--2389.

\bibitem[FIY13]{FIY13}
Maki Furukado, Shunji Ito, and Shin-Ichi Yasutomi, \emph{The condition for the
  generation of the stepped surfaces in terms of the modified {J}acobi-{P}erron
  algorithm}, preprint, 2013.

\bibitem[HS03]{HS03}
Michael Hollander and Boris Solomyak, \emph{Two-symbol {P}isot substitutions
  have pure discrete spectrum}, Ergodic Theory Dynam. Systems \textbf{23}
  (2003), no.~2, 533--540.

\bibitem[Hos86]{hos2}
Bernard Host, \emph{Valeurs propres des syst\`emes dynamiques d\'efinis par des
  substitutions de longueur variable}, Ergodic Theory Dynam. Systems \textbf{6}
  (1986), no.~4, 529--540.

\bibitem[HM06]{HM06}
Pascal Hubert and Ali Messaoudi, \emph{Best simultaneous {D}iophantine
  approximations of {P}isot numbers and {R}auzy fractals}, Acta Arith.
  \textbf{124} (2006), no.~1, 1--15.

\bibitem[IFHY03]{Ito03}
Shunji Ito, Junko Fujii, Hiroko Higashino, and Shin-ichi Yasutomi, \emph{On
  simultaneous approximation to {$(\alpha,\alpha^2)$} with
  {$\alpha^3+k\alpha-1=0$}}, J. Number Theory \textbf{99} (2003), no.~2,
  255--283.

\bibitem[IO93]{IO93}
Shunji Ito and Makoto Ohtsuki, \emph{Modified {J}acobi-{P}erron algorithm and
  generating {M}arkov partitions for special hyperbolic toral automorphisms},
  Tokyo J. Math. \textbf{16} (1993), no.~2, 441--472.

\bibitem[IO94]{IO94}
\bysame, \emph{Parallelogram tilings and {J}acobi-{P}erron algorithm}, Tokyo J.
  Math. \textbf{17} (1994), no.~1, 33--58.

\bibitem[IR06]{IR06}
Shunji Ito and Hui Rao, \emph{Atomic surfaces, tilings and coincidence. {I}.
  {I}rreducible case}, Israel J. Math. \textbf{153} (2006), 129--155.

\bibitem[IY07]{ItoYasu07}
Shunji Ito and Shin-ichi Yasutomi, \emph{On simultaneous {D}iophantine
  approximation to periodic points related to modified {J}acobi-{P}erron
  algorithm}, Probability and number theory---{K}anazawa 2005, Adv. Stud. Pure
  Math., vol.~49, Math. Soc. Japan, Tokyo, 2007, pp.~171--184.

\bibitem[Jac68]{Jac68}
Carl Gustav~Jacob Jacobi, \emph{{Allgemeine Theorie der kettenbruch\"ahnlichen
  Algorithmen, in welchen jede Zahl aus drei vorhergehenden gebildet wird}}, J.
  Reine Angew. Math. (1868), 29--64.

\bibitem[KV98]{KV98}
Richard Kenyon and Anatoly Vershik, \emph{Arithmetic construction of sofic
  partitions of hyperbolic toral automorphisms}, Ergodic Theory Dynam. Systems
  \textbf{18} (1998), no.~2, 357--372.

\bibitem[Man02]{Manning}
Anthony Manning, \emph{A {M}arkov partition that reflects the geometry of a
  hyperbolic toral automorphism}, Trans. Amer. Math. Soc. \textbf{354} (2002),
  no.~7, 2849--2863.

\bibitem[PD84]{PD84}
Roger Paysant{-}Le{~}Roux and Eug\'ene Dubois, \emph{Une application des
  nombres de {P}isot \`a l'algorithme de {J}acobi-{P}erron}, Monatsh. Math.
  \textbf{98} (1984), no.~2, 145--155.

\bibitem[Per07]{Per07}
Oskar Perron, \emph{Grundlagen f\"ur eine {T}heorie des {J}acobischen
  {K}ettenbruchalgorithmus}, Math. Ann. \textbf{64} (1907), no.~1, 1--76.

\bibitem[Pra99]{Pra99}
Brenda Praggastis, \emph{Numeration systems and {M}arkov partitions from
  self-similar tilings}, Trans. Amer. Math. Soc. \textbf{351} (1999), no.~8,
  3315--3349.

\bibitem[Que10]{Que10}
Martine Queff{\'e}lec, \emph{Substitution dynamical systems---spectral
  analysis}, second ed., Lecture Notes in Mathematics, vol. 1294,
  Springer-Verlag, Berlin, 2010.

\bibitem[Rau82]{Rau82}
G\'erard Rauzy, \emph{Nombres alg\'ebriques et substitutions}, Bull. Soc. Math.
  France \textbf{110} (1982), no.~2, 147--178.

\bibitem[Rev91]{Rev91}
Jean-Pierre Reveill{\`e}s, \emph{G\'eom\'etrie discr\`ete, calculs en nombres
  entiers et algorithmes}, Ph.D. thesis, Universit\'e Louis Pasteur,
  Strasbourg, 1991.

\bibitem[Rob04]{Rob04}
E.~Arthur Robinson, Jr., \emph{Symbolic dynamics and tilings of {$\mathbb
  R^d$}}, Symbolic dynamics and its applications, Proc. Sympos. Appl. Math.,
  vol.~60, Amer. Math. Soc., Providence, RI, 2004, pp.~81--119.

\bibitem[Sag]{Sage}
The Sage{~}Development{~}Team, \emph{{S}age {M}athematics {S}oftware},
  \url{http://www.sagemath.org}.

\bibitem[SAI01]{SAI01}
Yuki Sano, Pierre Arnoux, and Shunji Ito, \emph{Higher dimensional extensions
  of substitutions and their dual maps}, J. Anal. Math. \textbf{83} (2001),
  183--206.

\bibitem[Sch73]{Sch73}
Fritz Schweiger, \emph{The metrical theory of {J}acobi-{P}erron algorithm},
  Lecture Notes in Mathematics, Vol. 334, Springer-Verlag, Berlin, 1973.

\bibitem[Sch95]{Schweiger95}
\bysame, \emph{Ergodic theory of fibred systems and metric number theory},
  Oxford Science Publications, The Clarendon Press Oxford University Press, New
  York, 1995.

\bibitem[Sch00]{Sch00}
\bysame, \emph{Multidimensional continued fractions}, Oxford Science
  Publications, Oxford University Press, Oxford, 2000.

\bibitem[Sid03]{Sidorov}
Nikita Sidorov, \emph{Arithmetic dynamics}, Topics in dynamics and ergodic
  theory, London Math. Soc. Lecture Note Ser., vol. 310, Cambridge Univ. Press,
  Cambridge, 2003, pp.~145--189.

\bibitem[Sie00]{Sie00}
Anne Siegel, \emph{Repr\'esentations g\'eom\'etrique, combinatoire et
  arithm\'etique des syst\`emes substitutifs de type {P}isot}, Ph.D. thesis,
  Universit\'e de la M\'editerran\'ee, 2000.

\bibitem[ST09]{SieT09}
Anne Siegel and J{\"o}rg~M. Thuswaldner, \emph{Topological properties of
  {R}auzy fractals}, M\'em. Soc. Math. Fr. (N.S.) (2009), no.~118, 140.

\bibitem[Sol97]{Sol97}
Boris Solomyak, \emph{Dynamics of self-similar tilings}, Ergodic Theory Dynam.
  Systems \textbf{17} (1997), no.~3, 695--738, See also \emph{Corrections to:
  ``{D}ynamics of self-similar tilings''}, {E}rgodic {T}heory {D}ynam.\
  {S}ystems {\bf 19} (1999), no.\ 6, 1685.

\end{thebibliography}

\end{document}